\documentclass[10pt]{article}
\usepackage{amsmath}
\usepackage{amsfonts}
\usepackage{amssymb}
\usepackage{enumitem}
\usepackage[square, sort, numbers]{natbib}
\usepackage[margin=1.3in]{geometry}
\usepackage{fancyhdr}
\usepackage{amsthm}
\usepackage[hidelinks]{hyperref}
\usepackage{color}
\usepackage{wrapfig}
\usepackage{float}
\usepackage{mathtools}
\usepackage{verbatim}
\usepackage{appendix}
\usepackage{subcaption}

\usepackage{graphicx}
\usepackage{fullpage}

\usepackage{mathrsfs}
\usepackage{dsfont}
\title{Variational approach to coarse-graining of generalized gradient flows}

\usepackage{amsthm}
\theoremstyle{plain}
\newtheorem{theorem}{Theorem}[section]
\newtheorem{lemma}[theorem]{Lemma}

\newtheorem{corollary}[theorem]{Corollary}

\newtheorem{proposition}[theorem]{Proposition}
\theoremstyle{definition}
\newtheorem{remark}[theorem]{Remark}

\newcommand{\R}{\mathbb{R}}

\newcommand{\ve}{\varepsilon}
\DeclareMathAlphabet\gothic{U}{euf}{m}{n}
\DeclareMathOperator{\supp}{supp}

\def\div{\mathop{\mathrm{div}}\nolimits}

\usepackage{tocloft}
\def\div{\mathop{\mathrm{div}}\nolimits}
\def\opN{\mathcal N}
\def\opL{\mathscr L}
\def\RelEnt{\mathcal H}
\def\RelFI{\mathcal I}
\def\Hausdorff{\mathscr H}
\def\Lebesgue{\mathcal L}
\def\LDJ{\mathcal J}
\DeclareMathOperator\Int{Int}
\newcommand\TA{T\!A}

\let\oldnabla\nabla
\def\nabla{\oldnabla\sbcorr{-2mu}{0mu}} 
\def\sbcorr#1#2{\def\tmpa{#1}\def\tmpb{#2}\futurelet\next\sbcorrA}
\def\sbcorrA{\ifx\next_\expandafter\sbcorrB\fi}
\def\sbcorrB_#1{_{\mkern\tmpa#1}\futurelet\next\sbcorrC}
\def\sbcorrC{\ifx\next^\expandafter\sbcorrD\fi}
\def\sbcorrD^#1{^{\mkern\tmpb#1}}

\usepackage{tocloft}
\def\vep{\varepsilon}
\let\e\vep
\def\hrho{\hat{\rho}}
\def\trho{\tilde{\rho}}

\def\d{\delta}

\def\g{\gamma}
\def\G{\Gamma}
\def\k{\kappa}
\def\M{\mathcal}
\def\Rd{{\mathbb{R}^{2d}}}
\def\rd{{\mathbb{R}^{d}}}
\definecolor{darkviolet}{rgb}{0.58, 0.0, 0.83}

\DeclareRobustCommand{\rchi}{{\mathpalette\irchi\relax}}
\newcommand{\irchi}[2]{\raisebox{\depth}{$#1\chi$}}

\author{Manh Hong Duong, Agnes Lamacz, Mark A. Peletier and Upanshu Sharma}

\begin{document}
\maketitle

\begin{abstract}
In this paper we present a variational technique that handles coarse-graining and passing to a limit in a unified manner. The technique is based on a duality structure, which is present in many gradient flows and other variational evolutions, and which often arises from a large-deviations principle. It has three main features: (A) a natural interaction between the duality structure and the coarse-graining, (B) application to systems with non-dissipative effects, and (C) application to coarse-graining of approximate solutions which solve the equation only to some error. As examples, we use this technique to solve three limit problems, the overdamped limit of the Vlasov-Fokker-Planck equation and the small-noise limit of randomly perturbed Hamiltonian systems with one and with many degrees of freedom.  
\end{abstract}
\tableofcontents

\section{Introduction}

Coarse-graining  is the procedure of 
approximating a system by a simpler or lower-dimensional one, often in some limiting regime.
It arises naturally in various fields such as thermodynamics, quantum mechanics, and molecular dynamics, just to name a few. Typically coarse-graining requires a separation of temporal and/or spatial scales, i.e.\ the presence of fast and slow variables. As the ratio of  `fast' to `slow' increases, some form of averaging or homogenization should allow one to remove the fast scales, and obtain a limiting system that focuses on the slow ones.

Coarse-graining limits are by nature \emph{singular limits}, since information is lost in the coarse-graining procedure; therefore rigorous proofs of such limits are always non-trivial. Although the literature abounds with cases that have been treated successfully, and some fields can even be called well-developed---singular limits in ODEs and homogenization theory, to name just two---many more cases seem out of reach, such as coarse-graining in materials~\cite{PabloCurtin07}, climate prediction~\cite{StainforthAllenTredgerSmith07}, and complex systems~\cite{GrandChallenges07,NicolisNicolis12}. 

All proofs of singular limits hinge on using certain \emph{special structure} of the equations; well-known examples are compensated compactness~\cite{Tartar79,Murat87}, the theories of viscosity solutions~\cite{CrandallIshiiLions92} and entropy solutions~\cite{Kruzkov70,Smoller94}, and the methods of periodic unfolding~\cite{CioranescuDamlamianGriso02,CioranescuDamlamianGriso08} and two-scale convergence~\cite{Allaire92}. \emph{Variational-evolution structure}, such as in the case of gradient flows and variational rate-independent systems,  also facilitates limits~\cite{SandierSerfaty04,Stefanelli08,MielkeRoubicekStefanelli08,DaneriSavare10TR,Serfaty11,MielkeRossiSavare12,Mielke14TR}. 

In this paper we introduce and study such a structure, which arises from the theory of \emph{large deviations} for stochastic processes. In recent years we have discovered that many gradient flows, and also many `generalized' gradient systems, can be matched one-to-one to the large-deviation characterization of some stochastic process~\cite{AdamsDirrPeletierZimmer11,AdamsDirrPeletierZimmer13,DuongPeletierZimmer14,DuongPeletierZimmer13,DirrLaschosZimmer12,MielkePeletierRenger14}. The large-deviation rate functional, in this connection, can be seen to \emph{define} the generalized gradient system. This connection has many philosophical and practical implications, which are discussed in the references above.

We show how in such systems, described by a rate functional, `passing to a limit' is facilitated by the duality structure that a rate function inherits from the large-deviation context, in a way that meshes particularly well with coarse-graining.

\subsection{Variational approach---an outline}\label{DOG-sec:Outline-Var-App}
\label{sec:coarse-graining-intro}
The systems that we consider in this paper are evolution equations in a space of measures
. Typical examples are the forward Kolmogorov equations associated with stochastic processes, but also various nonlinear equations, as in one of the examples below.

Consider the family of evolution equations
\begin{equation}
\label{DOG-eq:Formal-Evolution}
\begin{aligned}
&\partial_t\rho^\ve =\opN^\ve \rho^\ve,\\
&\rho^\ve|_{t=0}=\rho_0^\ve, 
\end{aligned}
\end{equation}
where $\opN^\ve$ is a linear or nonlinear operator. 
The unknown $\rho^\ve$ is a time-dependent Borel measure on a state space~$\mathcal{X}$, i.e. $\rho^\ve:[0,T]\rightarrow \mathcal{M}(\mathcal{X})$. 
In the systems of this paper, \eqref{DOG-eq:Formal-Evolution} has a variational formulation characterized by a functional $I^\ve$ such that 
\begin{equation}
\label{eq: variational formulation}
I^\ve\geq 0 \qquad\text{and}\qquad \rho^\ve \text{ solves } \eqref{DOG-eq:Formal-Evolution}\  \Longleftrightarrow\  I^\ve(\rho^\ve)=0.
\end{equation}
This variational formulation is closely related to the Brezis-Ekeland-Nayroles variational principle~\cite{BrezisEkeland76,Nayroles76,Stefanelli08,Ghoussoub09} and the integrated energy-dissipation identity for gradient flows~\cite{AmbrosioGigliSavare08}; see Section~\ref{sec: discussion}.

Our interest in this paper is the limit $\e\to0$, and we wish to study the behaviour of the system in this limit. If we postpone the aspect of coarse-graining for the moment, this corresponds to studying the limit of $\rho^\ve$ as $\ve\to0$. Since $\rho^\ve$ is characterized by $I^\ve$, establishing the limiting behaviour consists of answering two questions:
\begin{enumerate}
\item{\textit{Compactness}: Do solutions of $I^\ve(\rho^\ve)=0$ have useful compactness properties, allowing one to extract a subsequence that converges in a suitable topology, say $\varsigma$?}
\item{\textit{Liminf inequality}: Is there a limit functional $I\geq 0$ such that 
\begin{equation}
\label{prop:liminf-intro-abstract}
\rho^\ve\stackrel \varsigma \longrightarrow\rho\ \Longrightarrow\ \liminf\limits_{\ve\rightarrow 0} I^\ve(\rho^\ve)\geq I(\rho)?
\end{equation}
And if so, does one have
\begin{align*}
I(\rho)=0\ \Longleftrightarrow\  \rho \text{ solves }  \partial_t\rho =\opN\rho,
\end{align*}
for some operator $\opN$?
}
\end{enumerate}
A special aspect of the method of the present paper is that it also applies to \emph{approximate} solutions. By this we mean that we are interested in sequences of time-dependent Borel measures $\rho^\varepsilon$ such that $\sup_{\varepsilon>0} I^\varepsilon(\rho^\varepsilon)\leq C$ for some $C\geq 0$. The exact solutions are special cases when $C=0$. The main message of our approach is that all the results then follow from this uniform bound and assumptions on well-prepared initial data.

The compactness question will be answered by the first crucial property of the functionals $I^\ve$, which is that they provide an \emph{a priori} bound of the type
\begin{equation}
\label{eq: bound of energy and fisher information}
S^\ve(\rho^\ve_t) + \int_0^t R^\ve(\rho^\ve_s)\, ds \leq S^\ve(\rho^\ve_0) +  I^\ve(\rho^\ve),
\end{equation}
where $\rho^\vep_t$ denotes the time slice at time $t$ and $S^\ve$ and $R^\ve$ are functionals. In the examples of this paper $S^\ve$ is a free energy and $R^\ve$ a relative Fisher Information, but the structure is more general. This inequality is reminiscent of the energy-dissipation inequality in the gradient-flow setting. The uniform bound, by assumption, of the right-hand side of \eqref{eq: bound of energy and fisher information} implies that each term in the left-hand side of \eqref{eq: bound of energy and fisher information}, i.e., the free energy at any time $t>0$ and the integral of the Fisher information,  is also bounded. This will be used to apply the Arzel\`a-Ascoli theorem to obtain certain compactness  and `local-equilibrium' properties. All this discussion will be made clear in each example in this paper.

\medskip

The second crucial property of the functionals $I^\ve$ is that they satisfy a duality relation of the type 
\begin{align}
I^\ve(\rho)=\sup\limits_f \LDJ^\ve(\rho,f),\label{DOG-eq:Abstract-Duality-Form-Rate-Fn}
\end{align}
where the supremum is taken over a class of smooth functions $f$. It is well known how such duality structures give rise to good convergence properties such as~\eqref{prop:liminf-intro-abstract}, but the focus in this paper is on how this duality structure combines well with coarse-graining.

In this paper we define \emph{coarse-graining} to be a shift to a reduced, lower dimensional description via a coarse-graining map $\xi:\mathcal{X}\rightarrow \mathcal{Y}$ which identifies relevant information
 and is typically highly non-injective. Note that $\xi$ may depend on $\varepsilon$. A typical example of such a coarse-graining map is a `reaction coordinate' in molecular dynamics. The coarse-grained equivalent of $\rho^\e:[0,T]\rightarrow \mathcal{M}(\mathcal{X})$ is the push-forward $\hrho^\e:=\xi_\#\rho^\e:[0,T]\rightarrow\mathcal{M}(\mathcal{Y})$.  If $\rho^\e$ is the law of a stochastic process $X^\e$, then $\xi_\#\rho^\e$ is the law of the process $\xi(X^\e)$. 
 
There might be several reasons to be interested in $\xi_\#\rho^\e$ rather than $\rho^\e$ itself. The push-forward $\xi_\#\rho^\e$ obeys a dynamics with fewer degrees of freedom, since $\xi$ is non-injective; this might allow for more efficient computation. Our first example (see Section~\ref{DOG-sec:Concrete-Problems}), the overdamped limit in the Vlasov-Fokker-Planck equation, is an example of this. As a second reason, by removing certain degrees of freedom, some specific behaviour of $\rho^\e$ might become clearer; this is the case with our second and third examples (Section~\ref{DOG-sec:Concrete-Problems}), where the effect of $\xi$ is to remove a rapid oscillation, leaving behind a slower diffusive movement. Whatever the reason, in this paper we assume that some $\xi$ is given, and that we wish to study the limit of $\xi_\#\rho^\e$ as $\e\rightarrow 0$.
 
\medskip
The core of the arguments of this paper, that leads to the characterization of the equation satisfied by the limit of $\xi_\#\rho^\e$, is captured by the following formal calculation: 
\begin{eqnarray*}
I^\ve (\rho^\ve) &=& \sup_f \; \LDJ^\ve(\rho^\ve,f)\\
&\stackrel{f=g\circ \xi} \geq & \sup_g \;\LDJ^\ve(\rho^\ve,g\circ \xi)\\
&& \phantom{\sup\; \widehat \LDJ^\ve(}\Big\downarrow\;\ve\rightarrow 0\\[\jot]
&&\sup_g\;  {\LDJ}({\rho},g\circ\xi)\\
&\stackrel{(*)}\eqqcolon & \sup_g\; {\hat \LDJ}({\hrho},g)\quad
\stackrel{(**)}=:\quad{\hat I}({\hat \rho})
\end{eqnarray*}

Let us go through the lines one by one.
The first line is the duality characterization~\eqref{DOG-eq:Abstract-Duality-Form-Rate-Fn} of $I^\e$. The inequality in the second line is due to the reduction to a subset of special functions $f$, namely those of the form $f=g\circ \xi$. This is in fact an implementation of coarse-graining: in the supremum we decide to limit ourselves to observables of the form $g\circ \xi$ which only have access to the information provided by $\xi$. After this reduction we pass to the limit and show that $\LDJ^\ve(\rho^\ve,g\circ\xi)$ converges to some $\LDJ(\rho,g\circ\xi)$---at least for appropriately chosen coarse-graining maps.

In the final step $(*)$ one requires that the  loss-of-information in passing from $\rho$ to $\hrho$ is consistent with the loss-of-resolution in considering only functions $f=g\circ \xi$. This step requires a proof of \emph{local equilibrium}, which describes how the behaviour of $\rho$ that is \emph{not} represented explicitly by the push-forward~$\hrho$, can nonetheless be deduced from $\hrho$. This local-equilibrium property is at the core of various coarse-graining methods and is typically determined case by case. 

We finally define $\hat{I}$ by duality in terms of $\hat{J}$ as in $(**)$.
In a \emph{successful} application of this method, the resulting functional $\hat I$ at the end has `good' properties \emph{despite} the loss-of-accuracy introduced by the restriction to functions of the form $g\circ \xi$, and this fact acts as a test of success. Such good properties should include, for instance, the property that $\hat I = 0$ has a unique solution in an appropriate sense.

Now let us explain the origin of the functionals $I^\varepsilon$.

\subsection{Origin of the functional $I^\varepsilon$: large deviations of a stochastic particle system}
\label{sec:stochastic-intro}
The abstract methodology that we described above arises naturally in the context of \emph{large deviations}, and we now describe this in the context of the three examples that we discuss in the next section. All three originate from (slight modifications of) one stochastic process, that models a collection of interacting particles with inertia in the physical space $\R^d$:
\begin{subequations}\label{DOG-eq:Intro-VFP-SDE}
\begin{align}
&dQ^n_{i}(t)=\frac{P^n_{i}(t)}{m}dt,\label{DOG-eq:Intro-VFP-SDE-Q}\\
&dP^n_{i}(t)=-\nabla V(Q^n_{i}(t))dt-\frac1n \sum\limits_{j=1}^n\nabla\psi(Q^n_{j}(t)-Q^n_{i}(t)) dt-\frac{\gamma}{m}P^n_{i}(t)dt+\sqrt{2\gamma\theta}\,dW_i(t).\label{DOG-eq:Intro-VFP-SDE-P}
\end{align}
\end{subequations}
Here $Q^n_i\in\R^d$ and $P^n_i\in\R^d$ are the position and momentum of particles $i=1,\ldots,n$ with mass $m$. Equation~\eqref{DOG-eq:Intro-VFP-SDE-Q} is the usual relation between $\dot Q^n_i$ and $P^n_i$, and \eqref{DOG-eq:Intro-VFP-SDE-P} is a force balance which describes the forces acting on the particle. For this system, corresponding to the first example below, these forces are (a) a force arising from a fixed potential $V$, (b) an interaction force deriving from a potential $\psi$, (c) a friction force, and (d) a stochastic force characterized by independent $d$-dimensional Wiener measures $W_i$. Throughout this paper we collect $Q_i^n$ and $P_i^n$ into a single variable $X^n_i= (Q^n_i,P^n_i)$.

The parameter $\gamma$ characterizes the intensity of collisions of the particle with the solvent; it is present  in both the friction term and the noise term, since they both arise from these collisions (and in accordance with the Einstein relation). The parameter $\theta=kT_a$, where $k$ is the Boltzmann constant and $T_a$ is the absolute temperature, measures the mean kinetic energy of the solvent molecules, and therefore characterizes the magnitude of collision noise. Typical applications of this system are for instance as a simplified model for chemical reactions, or as a model for particles interacting through Coulomb, gravitational, or volume-exclusion forces. However, our focus in this paper is on methodology, not on technicality, so we will assume that $\psi$ is sufficiently smooth later on.

\medskip

We now consider the many-particle limit $n\to\infty$ in~\eqref{DOG-eq:Intro-VFP-SDE}. It is a well-known fact that the empirical measure
\begin{align}\label{DOG-eq:Emperical-Measure}
\rho_{n}(t)=\frac{1}{n}\sum\limits_{i=1}^n\delta_{X^n_{i}(t)}
\end{align}
converges almost surely to the unique solution of the \emph{Vlasov-Fokker-Planck  (VFP) equation}~\cite{Oelschlager84}
\begin{alignat}2
\label{DOG-eq:Intro-VFP}
\partial_t\rho&= (\opL_{\rho})^* \rho, &\qquad
 (\opL_\mu)^*\rho&:=  -\div_q\big{(}\rho\frac{p}{m}\big{)}+\div_p\rho\Bigl(\nabla_q V+\nabla_q\psi\ast\mu+\gamma\frac{p}{m}\Bigr)+\gamma\theta\,\Delta_p\rho ,\\
&&&\phantom{:}= -\div \rho J\nabla (H+\psi*\mu) + \gamma\div_p \rho\frac pm + \gamma \theta \Delta_p \rho,
\label{DOG-eq:Intro-VFPa}
\end{alignat}
with an initial datum that derives from the initial distribution of $X_{i}^n$.
The spatial domain here is $\mathbb{R}^{2d}$ with coordinates $(q,p)\in\mathbb{R}^{d}\times \mathbb{R}^{d}$, and subscripts such as in $\nabla _q $ and $\Delta_p$ indicate that differential operators act only on corresponding variables. The convolution is defined by $(\psi\ast\rho)(q)=\int_{\mathbb{R}^{2d}}\psi(q-q')\rho(q',p')dq'dp'$. 
In the second line above we use a slightly shorter way of writing $\opL_\mu^*$, by introducing the Hamiltonian $H(q,p) = p^2/2m + V(q)$ and the canonical symplectic matrix $J = \bigl(\begin{smallmatrix}0&I\\-I&0\end{smallmatrix}\bigr)$. This way of writing also highlights that the system is a combination of conservative effects, described by $J$, $H$, and $\psi$, and dissipative effects, which are parametrized by $\gamma$. The primal form $\opL_\mu$ of the operator~$(\opL_\mu)^*$ is
\[
\opL_\mu f = J\nabla (H+\psi*\mu)\cdot \nabla f - \gamma \frac pm\cdot \nabla_p f + \gamma\theta \Delta_p f.
\] 


The almost-sure convergence of $\rho_n$ to the solution $\rho$ of the (deterministic) VFP equation is the starting point for a \emph{large-deviation} result. In particular it has been shown that the sequence $(\rho_n)$ has a \emph{large-deviation property}~\cite{DawsonGartner87,BudhirajaDupuisFischer12,DuongPeletierZimmer13} which characterizes the probability of finding the empirical measure far from the limit $\rho$, written informally as 
\begin{align*}
\text{Prob}(\rho_n\approx\rho)\sim \text{exp}\Big(-\frac{n}{2}I(\rho)\Big),
\end{align*}
in terms of a \emph{rate functional} $I:C([0,T];\mathcal{P}(\R^{2d}))\rightarrow\mathbb{R}$. If we assume  that the initial data $X_i^n$ are chosen to be deterministic, and such that the initial empirical measure $\rho_n(0)$ converges narrowly to some $\rho_0$,  then $I$ has the form~\cite{DuongPeletierZimmer13}
\begin{align}\label{DOG-eq:Large-Dev-Rate-Fn-Generator-Form}
I(\rho):=\sup\limits_{f\in C^{1,2}_b(\R\times\R^{2d})}\ 
\int\limits_{\R^{2d}}f_T\,d\rho_T
-\int\limits_{\R^{2d}}f_0\,d\rho_0
-\int\limits_0^T\int\limits_{\R^{2d}}
\big{(}\partial_t f+ \opL_{\rho_t} f\big{)}\,d\rho_tdt
-\frac{1}{2}\int\limits_0^T\int
\limits_{\R^{2d}}\Lambda(f,f)\, d\rho_tdt,
\end{align} 
provided $\rho_t |_{t=0} = \rho_0$, where $\Lambda$ is the carr\'e-du-champ operator (e.g.~\cite[Section 1.4.2]{Bakry2014})
\begin{equation*}
\Lambda(f,g):=\frac{1}{2}\bigl(\opL_\mu(fg)-f\opL_\mu g-g\opL_\mu f\bigr) 
= \gamma\theta\,  \nabla_p f\nabla_p g  .
\end{equation*}
If the initial measure $\rho_t|_{t=0}$ is not equal to the limit $\rho_0$ of the stochastic initial empirical measures, then $I(\rho)=\infty$.

\medskip

Note that the functional $I$ in~\eqref{DOG-eq:Large-Dev-Rate-Fn-Generator-Form} is non-negative, since $f\equiv 0$ is admissible. If $I(\rho)=0$, then by replacing $f $ by $\lambda f$ and letting $\lambda$ tend to zero we find that $\rho$ is the weak solution of~\eqref{DOG-eq:Intro-VFP} (which is unique, given initial data~$\rho_0$~\cite{Funaki84}). Therefore $I$ is of the form that we discussed in Section~\ref{DOG-sec:Outline-Var-App}: $I\geq0$, and $I(\rho)=0$ iff $\rho$ solves~\eqref{DOG-eq:Intro-VFP}, which is a realization of \eqref{DOG-eq:Formal-Evolution}. 

\medskip

\subsection{Concrete Problems }\label{DOG-sec:Concrete-Problems}

We now apply the coarse-graining method of Section~\ref{sec:coarse-graining-intro} to three limits: the \emph{overdamped} limit $\gamma\to\infty$, and two \emph{small-noise} limits $\theta\to0$. In each of these three limits, the VFP equation~\eqref{DOG-eq:Intro-VFP} is the starting point, and we prove convergence to a limiting system using appropriate coarse-graining maps. Note that the convergence is therefore from one deterministic equation to another one; but the method makes use of the large-deviation structure that the VFP equation has inherited from its stochastic origin.

\subsubsection{Overdamped limit of the Vlasov-Fokker-Planck  equation}
The first limit that we consider is the limit of large friction, $\g\rightarrow\infty$, in the Vlasov-Fokker-Planck  equation \eqref{DOG-eq:Intro-VFP}, setting $\theta=1$ for convenience. To motivate what follows, we divide \eqref{DOG-eq:Intro-VFP} throughout by $\g$ and formally let $\g\to\infty$ to find
\begin{align*}
\div_p\rho\Bigl(\frac{p}{m}\Bigr)+\Delta_p\rho=0,
\end{align*}
which suggests that in the limit $\g\rightarrow\infty$, $\rho$ should be Maxwellian in $p$, i.e.
\begin{align}\label{DOG-eq:VFP-Intro-Local-Eq}
\rho_t(dq,dp)=Z^{-1}\exp\Bigl(-\frac{p^2}{2m}\Bigr)\,dp \;\sigma_t(dq),
\end{align}
where $Z=(2m\pi)^{d/2}$ is the normalization constant for the Maxwellian distribution. 
The main result in Section~\ref{sec:overdamped-limit} shows that after an appropriate time rescaling, in the limit $\gamma\to\infty$, the remaining unknown $\sigma\in C([0,T];\mathcal{P}(\mathbb{R}^d))$ solves the Vlasov-Fokker-Planck equation 
\begin{align}\label{DOG-eq:VFP-Intro-Limiting-Eq}
\partial_t\sigma=\div(\sigma\nabla V(q))+\div(\sigma(\nabla\psi\ast\sigma))+\Delta\sigma.
\end{align}

In his seminal work~\cite{Kramers1940}, Kramers formally discussed these results for the `Kramers equation', which corresponds to~\eqref{DOG-eq:Intro-VFP} with $\psi= 0$, and this limit has become known as the \emph{Smoluchowski-Kramers approximation}. Nelson made these ideas rigorous~\cite{Nelson1967} by studying the corresponding stochastic differential equations (SDEs); he showed that under suitable rescaling  the solution to the Langevin equation converges almost surely to the solution of~\eqref{DOG-eq:VFP-Intro-Limiting-Eq} with $\psi=0$. Since then  various generalizations and related results have been proved \cite{Freidlin2004a,Cerrai2006,Narita1994,Hottovy2012}, mostly using stochastic and asymptotic techniques. 

In this article we recover some of the results mentioned above for the VFP equation using the variational technique described in Section~\ref{DOG-sec:Outline-Var-App}. Our proof is made up of the following three steps.
Theorem~\ref{DOG-thm:VFP-Compactness} provides the necessary compactness properties to pass to the limit, 
 Lemma~\ref{DOG-lem:VFP-Local-Eq} gives the characterization \eqref{DOG-eq:VFP-Intro-Local-Eq} of the limit, and in Theorem~\ref{DOG-thm:VFP-Liminf-Inequality} we prove the convergence of the solution of the VFP equation to the solution of~\eqref{DOG-eq:VFP-Intro-Limiting-Eq}.

\subsubsection{Small-noise limit of a randomly perturbed Hamiltonian system with  one degree of freedom}
In our second example we consider the following equation
\begin{align}\label{DOG-eq:DOG-Intro-1}
\partial_t\rho=-\div_q\Bigl(\rho\frac{p}{m}\Bigr)+\div_p(\rho\nabla_q V)+\vep\Delta_p\rho \qquad\text{on } [0,T]\times \R^2, 
\end{align}
where $(q,p)\in\mathbb{R}^2$, $t\in [0,T]$ and $\div_q,\,\div_p,\,\Delta_p$ are one-dimensional derivatives. This equation can also be written as 
\begin{align}\label{DOG-eq:Intro-DOG}
\partial_t\rho=-\div(\rho J\nabla H)+\vep\Delta_p\rho,
\qquad\text{on } [0,T]\times \R^2.
\end{align}
This corresponds to the VFP equation~\eqref{DOG-eq:Intro-VFP} with $\psi=0$, without friction and with small noise $\e=\gamma\theta$. 

In addition to the interpretation as the many-particle limit of~\eqref{DOG-eq:Intro-VFP-SDE}, Equation \eqref{DOG-eq:Intro-DOG} also is the forward Kolmogorov equation  of a  randomly perturbed Hamiltonian system in $\R^2$ with Hamiltonian $H$:
\begin{align}
\label{DOG-SDE:Intro-DOG}
X = \begin{pmatrix} Q\\P\end{pmatrix},\qquad
dX_t=J\nabla H (X_t)+\sqrt{2\vep}\begin{pmatrix}
0 \\ 1
\end{pmatrix}dW_t, 
\end{align}
where $W_t$ is now a $1$-dimensional Wiener process. When the amplitude $\e$ of the noise is small, the dynamics \eqref{DOG-eq:Intro-DOG} splits into fast and slow components. The fast component approximately follows an unperturbed trajectory of the Hamiltonian system, which is a level set of $H$. The slow component is visible as a slow modification of the value of $H$, corresponding to a motion transverse to the level sets of $H$.  Figure~\ref{DOG-fig:StochasticTrajectories} illustrates this. 

Following~\cite{Freidlin1994} and others, in order to focus on the slow, Hamiltonian-changing motion, we rescale time such that the Hamiltonian, level-set-following motion is fast, of rate $O(1/\e)$, and the level-set-changing motion is of rate $O(1)$. In other words, the process~\eqref{DOG-SDE:Intro-DOG} `whizzes round' level sets of $H$ at rate $O(1/\e)$, while shifting from one level set to another at rate $O(1)$. 

This behaviour suggests choosing a coarse-graining map $\xi:\R^2\to\Gamma$, which maps a whole \emph{level set} to a single point in a new space $\Gamma$; because of the structure of level sets of $H$, the set $\Gamma$ has a structure that is called a \emph{graph}, a union of one-dimensional intervals locally parametrized by the value of the Hamiltonian. Figure~\ref{DOG-fig:Hamiltonian-Graph} illustrates this, and in Section~\ref{DOG-sec:DOG} we discuss it in full detail.

\begin{figure}[t]
\centering
\begin{subfigure}{.45\textwidth}
  \centering
  \includegraphics[width=.90\linewidth]{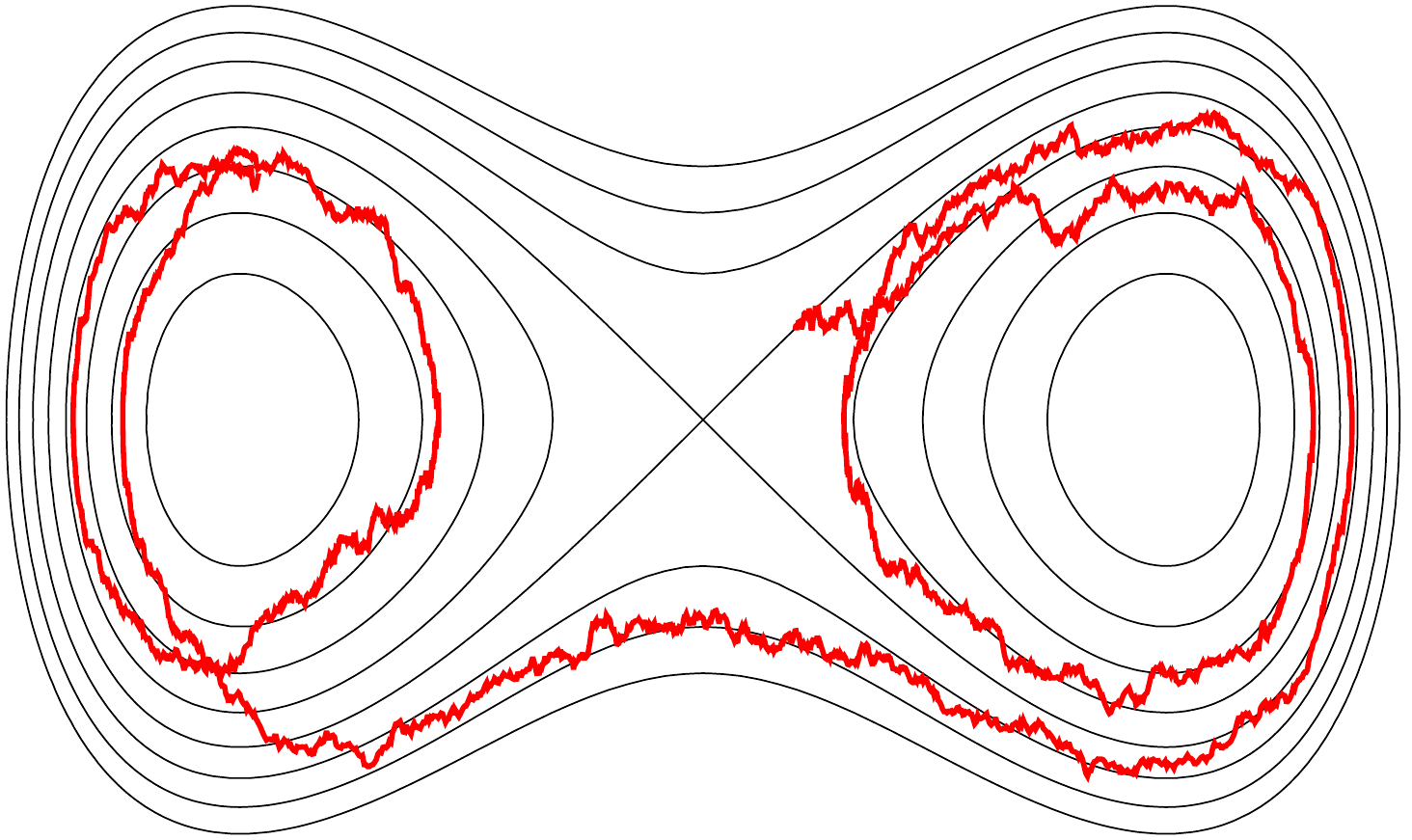}
  \caption{$\vep=0.005$}
\end{subfigure}%
\begin{subfigure}{.45\textwidth}
  \centering
  \includegraphics[width=.90\linewidth]{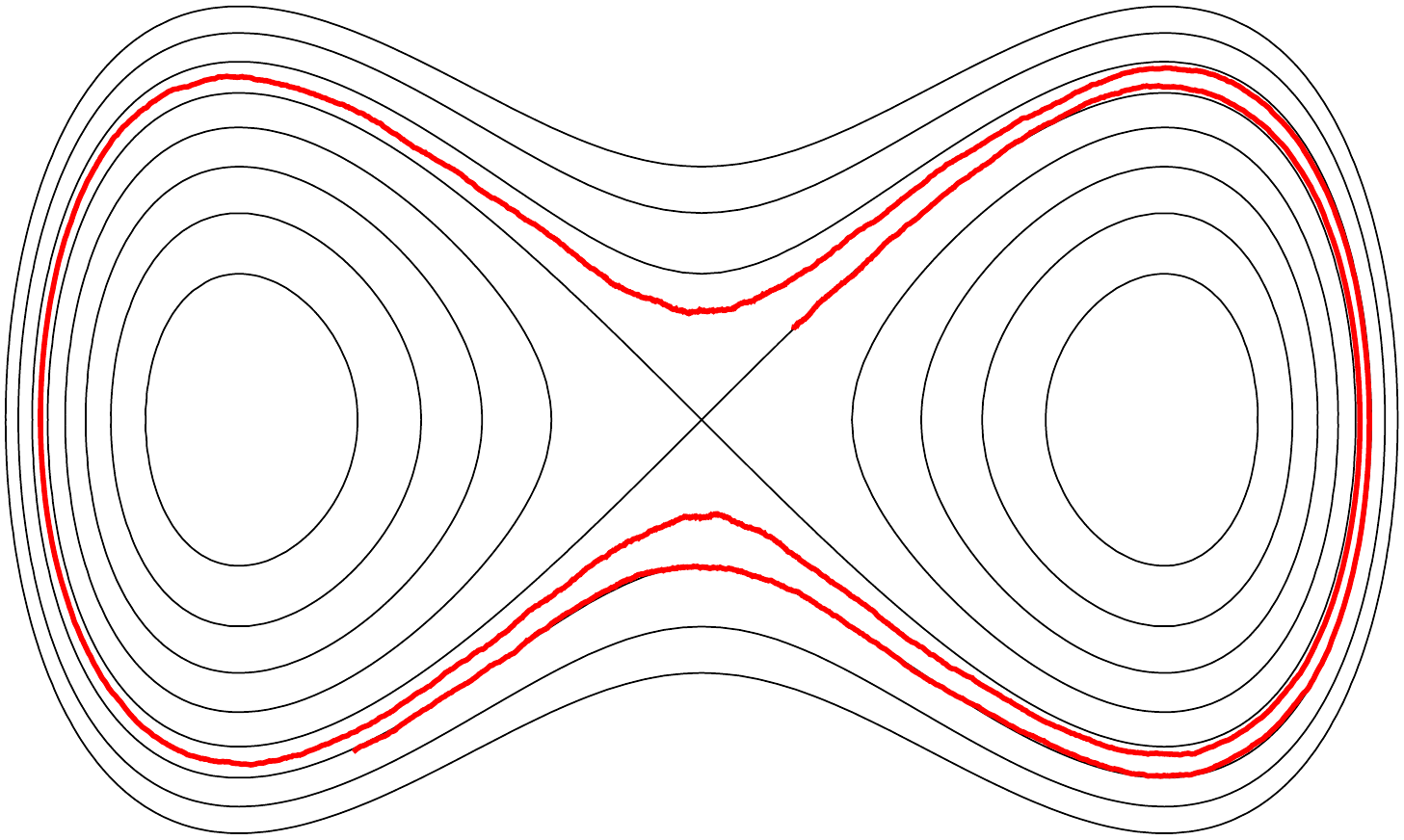}
  \caption{$\vep=0.00005$}
\end{subfigure}
\caption{Simulation of \eqref{DOG-SDE:Intro-DOG} for varying $\vep$. Shown are the level curves of the Hamiltonian $H$ and for each case a single trajectory.}
\label{DOG-fig:StochasticTrajectories}
\end{figure}

\begin{figure}[b]
\centering
\includegraphics[scale=0.55]{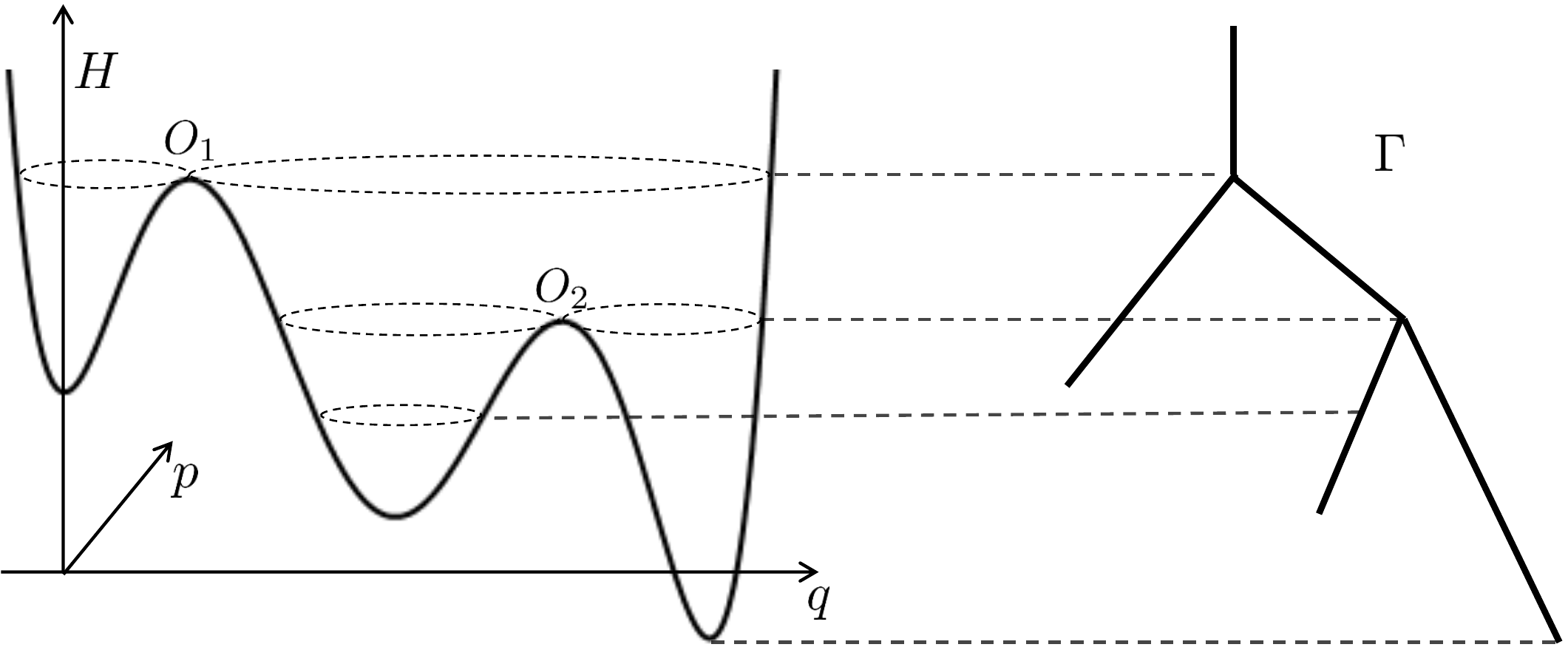}
\caption[Hamiltonian]{Left: Hamiltonian $\mathbb{R}^2\ni(q,p)\mapsto H(q,p)$, Right: Graph $\Gamma$}
\label{DOG-fig:Hamiltonian-Graph}
\end{figure}

After projecting onto the graph $\Gamma$, the process turns out to behave like a diffusion process on $\Gamma$. 
This property was first made rigorous in \cite{Freidlin1994} for a system with one degree of freedom, as here, and non-degenerate noise, using probabilistic techniques. In \cite{Freidlin1998} the authors consider the case of degenerate noise by using  probabilistic and analytic techniques based on hypoelliptic operators. More recently this problem has been handled using PDE techniques \cite{Ishii2012} (the elliptic case) and Dirichlet forms \cite{Barret2014}. In Section~\ref{DOG-sec:DOG} we give a new proof, using the structure outlined in Section~\ref{DOG-sec:Outline-Var-App}.

\subsubsection{Small-noise limit of a randomly perturbed Hamiltonian system with $d$ degrees of freedom}

The convergence of solutions of~\eqref{DOG-eq:Intro-DOG} as $\vep\rightarrow 0$ to a diffusion process on a graph requires that the non-perturbed system has a unique invariant measure on each connected component of a level set. 
While this is true for a Hamiltonian system with one degree of freedom, in the higher-dimensional case one might have additional first integrals of motion.
In such a system the slow component will not be a one-dimensional process but a more complicated object---see \cite{Freidlin2004}.
However, by introducing an additional stochastic perturbation that destroys all first integrals except the Hamiltonian, one can regain the necessary ergodicity, such that 
the slow  dynamics again lives on a graph. 

In Section~\ref{DOG-sec:DOG-d>1} we discuss this case. Equation~\eqref{DOG-eq:Intro-DOG}
gains an additional noise term, and reads 
\begin{align}\label{DOG-eq:Intro-DOG-d>1}
\partial_t\rho=-\div(\rho J\nabla H)+\kappa\div(a\nabla \rho)+\vep\Delta_p\rho,
\end{align}
where $a:\Rd\rightarrow \R^{2d\times2d}$ with $a\nabla H=0$, $\text{dim}(\text{Kernel}(a))=1$, and $\kappa,\vep>0$ with $\kappa\gg\vep$. The spatial domain is $\mathbb{R}^{2d}, \ d>1$  with coordinates $(q,p)\in\mathbb{R}^d\times \mathbb{R}^d$ and the unknown is a trajectory in the space of probability measures $\rho:[0,T]\rightarrow\mathcal{P}(\mathbb{R}^{2d})$. As before the aim is to derive the dynamics as $\vep\rightarrow 0$. This problem was studied in \cite{Freidlin2001} and 
the results closely mirror the previous case. The main difference lies in the proof of the local equilibrium statement, which we discuss in Section~\ref{DOG-sec:DOG-d>1}.

%
%

\subsection{Comparison with other work}

The novelty of the present paper lies in the following.

\begin{enumerate}
\item \textit{In comparison with existing literature on the three concrete examples treated in this paper:} The results of the three  examples are known in the literature (see for instance \cite{Nelson1967, Freidlin1994, Freidlin1998, Freidlin2001}), but they are proved by different techniques and in a different setting. The variational approach of this paper, which has a clear microscopic interpretation from the large-deviation principle, to these problems is new. We provide alternative proofs, recovering known results, in a unified framework. In addition, we obtain all the results on compactness, local-equilibrium properties and liminf inequalities solely from the variational structures. The approach also is applicable to approximate solutions, which obey the original fine-grained dynamics only to some error. This allows us to work with larger class of measures and to relax many regularity conditions required by the exact solutions. Furthermore, our abstract setting has  potential applications to many other systems.

\item \textit{In comparison with recently developed variational-evolutionary methods:} Many recently developed variational techniques for `passing to a limit' such as the Sandier-Saferty method based on the $\Psi$-$\Psi^*$ structure~\cite{SandierSerfaty04, ArnrichMielkePeletierSavareVeneroni12, Mielke14TR} only apply to gradient flows, i.e.\ dissipative systems. The approach of this paper also applies to certain variational-evolutionary systems that include non-dissipative effects, such as GENERIC systems \cite{Ottinger2005,DuongPeletierZimmer13}; our examples illustrate this. Since our approach only uses the duality structure of the rate functionals, which holds true for more general systems, this method also works for other limits in non-gradient-flow systems such as the Langevin limit of the Nos\'{e}-Hoover-Langevin thermostat \cite{FG11,OP11,Sharma17TH}.

\item \textit{Quantification of the coarse-graining error.} The use of the rate functional as a central ingredient in `passing to a limit' and coarse-graining also allows us to obtain quantitative estimates of the coarse-graining error. One intermediate result of our analysis is a functional inequality similar to the energy-dissipation inequality in the gradient-flow setting (see~\eqref{eq: bound of energy and fisher information}). This inequality provides an upper bound on the free energy and the integral of the Fisher information by the rate functional and initial free energy. 
To keep the paper to a reasonable length, we address this issue in details separately in a companion article \cite{DLPSS-TMP}.
\end{enumerate}
We provide further comments in Section \ref{sec: discussion}.

\subsection{Outline of the article}
The rest of the paper is devoted to the study of three concrete problems: the overdamped limit of the VFP equation in Section \ref{sec:overdamped-limit}, diffusion on a graph with one degree of freedom in Section \ref{DOG-sec:DOG}, and diffusion on a graph with many degrees of freedom in Section \ref{DOG-sec:DOG-d>1}. In each section, the main steps in the abstract framework are performed in detail. Section \ref{sec: discussion} provides further discussion. Finally, detailed proofs of some theorems are given in Appendices A and B.


\subsection{Summary of notation}
\begin{tabular}{c l c}
$\pm^{}_{kj}$ &$\pm1$, depending on which end vertex $O_j$ lies of edge $I_k$ & Sec.~\ref{DOG-sec:Graph}\\

$\mathcal F$ & Free energy & \eqref{def:FreeEnergy}, \eqref{def:FreeEnergyDOG}\\

$\gamma$ (Sec.~\ref{sec:overdamped-limit}) & large-friction parameter\\

$\Gamma,\g$ (Sec.~\ref{DOG-sec:DOG}) & The graph $\Gamma$ and its elements $\g$ & Sec.~\ref{DOG-sec:Graph}\\

$\mathcal{H}(\cdot|\cdot)$ & relative entropy  & \eqref{DOG-eq:Relative-Entropy} \\ 

$H(q,p)$ & $H(q,p) = p^2/2m + V(q)$, the  Hamiltonian\\

$\Hausdorff^n$ & $n$-dimensional Hausdorff measure  &   \\ 

$\mathcal{I}(\cdot|\cdot)$ & relative Fisher Information & \eqref{DOG-eq:Relative-Fisher-Information} \\ 

$\Int$ & The interior of a set\\

$I^\e$ & Large-deviation rate functional for the diffusion-on-graph problem & \eqref{DOG-eq:DOG-Large-Dev-Rate-Fn}\\

$I^\gamma$ & Large-deviation rate functional for the VFP equation & \eqref{def:I-gamma}\\

$J$ & $J = \bigl(\begin{smallmatrix}0&I\\-I&0\end{smallmatrix}\bigr)$, the canonical symplectic matrix \\
 


$\Lebesgue$ & Lebesgue measure \\

$\opL_\mu$, $(\opL_\mu)^*$ & primal and dual generators & Sec.~\ref{sec:stochastic-intro}\\

$\mathcal{M}(\mathcal X)$ & space of finite, non-negative Borel measures on $\mathcal X$  \\ 
 
$\mathcal{P}(\mathcal X)$ & space of probability measures on $ \mathcal X$ &  \\ 

$\hat \rho$ & push-forward under $\xi$ of $\rho$ & \eqref{def:hrho}\\

$T(\g)$ & period of the periodic orbit at $\g\in \Gamma$ & \eqref{DOG-eq:DOG-Time-Period}\\

$V(q)$ & potential on position (`on-site')\\

$x$ & $x = (q,p)$ joint variable\\

$\xi^\g,\xi$ & coarse-graining maps & \eqref{DOG-def:xi1}, \eqref{def:xi-DOG}\\

\end{tabular} 

\medskip

Throughout we use measure notation and terminology. For a given topological space $\mathcal X$, the space $\mathcal M(\mathcal X)$ is the space of non-negative, finite Borel measures on $\mathcal X$; $\mathcal P(\mathcal X)$ is the space of probability measures on $\mathcal X$. For a measure $\rho\in \M M([0,T]\times \R^{2d})$, for instance, we often write $\rho_t\in \M M(\R^{2d})$ for the time slice at time $t$; we also often use both the notation $\rho(x)dx$ and $\rho(dx)$ when $\rho$ is Lebesgue-absolutely-continuous. We equip $\mathcal M(\mathcal X)$ and $\mathcal P(\mathcal X)$ with the \emph{narrow} topology, in which convergence is characterized by duality with continuous and bounded functions on $\mathcal X$.

\section{Overdamped Limit of the VFP equation}
\label{sec:overdamped-limit}

\subsection{Setup of the system}
In this section we prove the large-friction limit $\g\rightarrow\infty$ of the VFP equation
\eqref{DOG-eq:Intro-VFP}.  Setting $\theta=1$ for convenience, and speeding time up by a factor $\gamma$, the VFP equation reads 
\begin{align}
\partial_t\rho = \opL_\rho^* \rho, \qquad
\opL_\nu^*\rho :=  -\gamma\div\rho J\nabla (H + \psi\ast\nu)+\gamma^2\bigg{[}\div_p\bigg{(}\rho\frac pm \bigg{)}+\Delta_p \rho\bigg{]}, 
\label{DOG-eq:Rescaled-VFP}
\end{align}
where, as before, $J = \begin{psmallmatrix}0 &I\\-I& 0\end{psmallmatrix}$ and $H(q,p) = p^2/2m + V(q)$. The spatial domain is $\mathbb{R}^{2d}$ with coordinates $(q,p)\in \mathbb{R}^{d}\times \mathbb{R}^{d}$  with $d\geq 1$, and  $\rho\in C([0,T];\mathcal{P}(\mathbb{R}^{2d}))$. For later reference we also mention the primal form of the operator $\opL_\nu^*$:
\begin{equation}
\label{def:L-VFP-rescaled}
\opL_\nu f = \gamma J\nabla (H+\psi*\nu)\cdot \nabla f - \gamma^2 \frac pm \cdot\nabla_p f + \gamma^2 \Delta_p f.
\end{equation}

We assume 
\begin{enumerate}[label=({V}\arabic*)]
\item\label{cond:VFP:V1}The potential $V\in C^2(\mathbb{R}^{d})$ has globally bounded second derivative. Furthermore $V\geq0$,  $|\nabla V|^2 \leq C(1+V)$ for some $C>0$, and $e^{-V}\in L^1(\R^{d})$. 
\item 
\label{cond:VFP:V2}
The interaction potential $\psi\in C^2(\mathbb{R}^{d})\cap W^{1,1}(\R^d)$ is symmetric, has globally bounded first and second derivatives, and the mapping $\nu\mapsto \int \nu*\psi\, d\nu$ is convex (or equivalently non-negative). 
\end{enumerate}

\medskip

As we described in Section~\ref{sec:coarse-graining-intro}, the study of the limit $\g\to\infty$ contains the following steps:
\begin{enumerate}
\item Prove compactness;
\item Prove a local-equilibrium property;
\item Prove a liminf inequality.
\end{enumerate}
According to the framework detailed by \eqref{DOG-eq:Formal-Evolution}, \eqref{eq: variational formulation}, 
each of these results is based on the large-deviation structure, which for Equation \eqref{DOG-eq:Rescaled-VFP} is associated to the functional 
$I^\gamma: C([0,T];\mathcal{P}(\R^{2d}))\rightarrow\mathbb{R}$ with
\begin{multline}
I^\gamma(\rho)= \sup\limits_{f\in C_b^{1,2}(\mathbb{R}\times\mathbb{R}^{2d})}\bigg{[} 
\int\limits_{\mathbb{R}^{2d}}f_T\,d\rho_T
-\int\limits_{\mathbb{R}^{2d}}f_0\,d\rho_0
-\int\limits_0^T\int\limits_{\mathbb{R}^{2d}}
\Bigl{(}\partial_t f_t+\opL_{\rho_t} f_t\Bigr{)}\,d\rho_tdt
-\frac{\gamma^2}{2}\int\limits_0^T\int
\limits_{\mathbb{R}^{2d}}\left|\nabla_p f_t\right|^2d\rho_tdt\bigg{]},
\label{def:I-gamma}
\end{multline}
where $\opL_\nu$ is given in~\eqref{def:L-VFP-rescaled}. 
Alternatively the rate functional can be written as~\cite[Theorem 2.5]{DuongPeletierZimmer13}
\begin{align}\label{DOG-eq:VFP-Alternate-LDRF}
I^\gamma(\rho)=
\begin{cases}
\displaystyle
\frac{1}{2}\int\limits_0^T\int\limits_{\mathbb{R}^{2d}}|h_t|^2\,d\rho_t dt
\ &\text{ if } \partial_t\rho_t=\opL_{\rho_t}^*\rho_t-\g\div_p(\rho_t h_t), \text{ for } h\in L^2(0,T;L^2_\nabla(\rho)), \text{ and } \rho|_{t=0} = \rho_0
\\
+\infty & \text{otherwise,}
\end{cases}
\end{align}
where $\opL_\nu^*$ is given in~\eqref{DOG-eq:Rescaled-VFP}. For fixed $t$, the space $L^2_\nabla(\rho_t)$ is the closure of the set $\{\nabla_p \varphi\,:\,\varphi\in C^\infty_c(\Rd)\}$ in $L^2(\rho_t)$, the $\rho_t$-weighted $L^2$-space. Similarly, $L^2(0,T; L^2_\nabla(\rho))$ is defined as the closure of 
$\{\nabla_p \varphi\,:\,\varphi\in C^\infty_c((0,T)\times\Rd)\}$ in the $L^2$-space associated to the space-time density $\rho$. 
This second form of the rate functional shows clearly how $I^\g(\rho) = 0$ is equivalent to the property that $\rho$ solves the VFP equation~\eqref{DOG-eq:Rescaled-VFP}. It also shows that if $I^\g(\rho)>0$, then $\rho$ is an approximate solution in the sense that it satisfies the VFP equation up to some error $-\g\div_p(\rho_t h_t)$ whose norm is controlled by the rate functional.

\subsection{A priori bounds}
We give ourselves a sequence, indexed by $\g$, of solutions $\rho^\g$ to the VFP equation~\eqref{DOG-eq:Rescaled-VFP} with initial datum $\rho^\g_t|_{t=0}=\rho_0$.  We will deduce the compactness of the sequence $\rho^\g$ from \emph{a priori} estimates, that are themselves derived from the rate function $I^\g$.

For probability measures $\nu,\zeta$ on $\Rd$ we first introduce:
\begin{itemize}
\item{Relative entropy: 
\begin{align}
\RelEnt(\nu|\zeta)=\begin{cases}
\displaystyle
\int_{\R^{2d}}[f\log f]\,d\zeta &\text{if}\quad \nu = f\zeta,\\
\infty&\text{otherwise}.
\end{cases}\label{DOG-eq:Relative-Entropy}
\end{align}
}
\item The free energy for this system:
\begin{equation}
\label{def:FreeEnergy}
\mathcal F(\nu) :=  
 \RelEnt(\nu|Z_H^{-1}e^{-H}dx) + \frac12\int_{\R^{2d}}  \psi*\nu \, d\nu
 = \int_{\R^{2d}} \Bigl[ \log g + H + \frac12 \psi*g \Bigr]\, gdx + \log Z_H,
 \end{equation}
where $Z_H = \int e^{-H}$ and the second expression makes sense whenever $\nu = gdx$. 
\end{itemize}
The convexity of the term involving $\psi$ (condition~\ref{cond:VFP:V2}) implies that the free energy $\mathcal F$ is strictly convex and has a unique minimizer $\mu\in\mathcal{P}(\mathbb{R}^{2d})$.
This minimizer is a stationary point of the evolution~\eqref{DOG-eq:Rescaled-VFP}, and has the implicit characterization
\begin{align}
\mu\in\mathcal{P}(\mathbb{R}^{2d}): \ \mu(dqdp)=Z^{-1}\exp\Bigl(-\bigl[H(q,p)+(\psi\ast\mu)(q)
\bigr]
\Bigr)\, dqdp\label{DOG-eq:VFP-Stationary-Measure},
\end{align}
where $Z$ is the normalization constant for $\mu$. Note that $\nabla_p \mu = -\mu\nabla_p  H = -p\mu/m $.

\medskip
We also define the \emph{relative Fisher Information} with respect to $\mu$ (in the $p$-variable only): 
\begin{align}
\RelFI(\nu|\mu)=
\sup_{\varphi\in C_c^\infty(\R^{2d})}
2\int_{\R^{2d}}\Bigl[ \Delta_p \varphi - \frac pm \nabla_p \varphi - \frac12 |\nabla_p\varphi|^2\Bigr]\, d\nu.
\label{DOG-eq:Relative-Fisher-Information}
\end{align}
Note that the right hand side of \eqref{DOG-eq:Relative-Fisher-Information} depends on $\mu$ via 
$\nabla_p\left(\log\mu\right)=-\nabla_p H(q,p) = -p/m$. 
In the more common case in which the derivatives $\Delta_p$ and $\nabla_p$ are replaced by the full derivatives $\Delta$ and $\nabla$, the relative Fisher Information has an equivalent formulation in terms of the Lebesgue density of $\nu$. In our case such equivalence only holds when $\nu$ is absolutely continuous with respect to the Lebesgue measure in both $q$ and~$p$: 

\begin{lemma}[Equivalence of relative-Fisher-Information expressions for a.c.\ measures]
\label{DOG-lem:two-defs-RFI}
If $\nu\in \M P(\R^{2d})$, $\nu(dx) = f(x)dx$ with $f\in L^1(\R^{2d})$, then 
\begin{equation}
\label{def:RelFI-ac}
\RelFI(\nu|\mu)=
\begin{cases}
\displaystyle
\int_{\R^{2d}}\Bigl|\frac{\nabla_p f}{f}\mathds{1}_{\{f>0\}} + \frac pm \Bigr|^2 f\,dqdp,\qquad&\text{if}\quad \nabla_p f\in L^1_{\mathrm{loc}}(dqdp),\\
\infty&\text{otherwise},
\end{cases}
\end{equation}
where $\mathds{1}_{\{f>0\}}$ denotes the indicator function of the set $\{x\in\R^{2d}\,|\,f(x)>0\}$ and $\nabla_p f$ is the distributional gradient of $f$ in the $p$-variable only. 
\end{lemma}
\noindent

For a measure of the form $\zeta(dq)f(p)dp$, with $\zeta \not\ll dq$, the functional $\RelFI$ in~\eqref{DOG-eq:Relative-Fisher-Information} may be finite while the integral in~\eqref{def:RelFI-ac} is not defined.
Because of the central role of duality in this paper, definition~\eqref{DOG-eq:Relative-Fisher-Information} is a natural one, as we shall see below. The proof of Lemma~\ref{DOG-lem:two-defs-RFI} is given in Appendix~\ref{app:RFI-lemma-proof}.

\medskip

In the introduction we mentioned that we expect $\rho^\g$ to become Maxwellian in the limit $\g\to\infty$. This will be driven by a vanishing relative Fisher Information, as we shall see below. For absolutely continuous measures, the characterization~\eqref{def:RelFI-ac} already provides the property
\[
\RelFI(fdx|\mu)=0 \qquad\Longrightarrow \qquad
f(q,p) = \tilde f(q) \exp\Bigl(-\frac{p^2}{2m}\Bigr).
\]
This property holds more generally:
\begin{lemma}[Zero relative Fisher Information implies Maxwellian]
\label{DOG-lem:zero-RFI}
If $\nu\in\M P(\Rd)$ with $\RelFI(\nu|\mu)=0$, then there exists $\sigma\in\M P(\R^d)$ such that 
\[
\nu(dqdp) = Z^{-1}\exp\left(-\frac{p^2}{2m}\right)\sigma(dq)dp,
\]
where $Z = \int_{\R^d} e^{-p^2/2m}dp$ is the normalization constant for the Maxwellian distribution. 
\end{lemma}

\begin{proof}
From 
\begin{align}\label{DOG-eq:VFP-Local-Eq-Limit-Desc}
\RelFI(\nu|\mu)= \sup_{\varphi\in C_c^\infty(\mathbb{R}^{2d})}\ 2 \int_{\Rd}\bigg{(}\Delta_p\varphi - 
\frac{p}{m}\cdot\nabla_p \varphi- \frac{1}{2}|\nabla_p\varphi|^2\bigg)d\nu = 0
\end{align}
we conclude upon disintegrating $\nu$ as $\nu(dqdp) = \sigma(dq) \nu_q(dp)$,  
\[
\text{for $\sigma$-a.e. $q$:} \qquad
\sup_{\phi\in C_c^\infty(\mathbb{R}^d)}\ \int_{\mathbb{R}^{d}}\bigg{(}\Delta_p\phi - 
\frac{p}{m}\cdot\nabla_p \phi- \frac{1}{2}|\nabla_p\phi|^2\bigg) \, \nu_q(dp) = 0.
\]
By replacing $\phi$ by $\lambda\phi$, $\lambda>0$, and taking $\lambda\to0$ we find
\[
\forall \phi\in C_c^\infty(\mathbb{R}^d):\  \int_{\mathbb{R}^{d}}\bigg{(}\Delta_p\phi - 
\frac{p}{m}\cdot\nabla_p \phi\bigg) \, \nu_q(dp) = 0,
\]
which is the weak form of an elliptic equation on $\R^d$ with unique solution (see e.g. \cite[Theorem 4.1.11]{BKRS15})
\[
\nu_q(dp) = \frac1Z \exp\left(-\frac{p^2}{2m}\right)dp.
\]
This proves the lemma.
\end{proof}

In the following theorem we give the central \emph{a priori} estimate, in which free energy and relative Fisher Information are bounded from above by the rate functional and the relative entropy at initial time. 
\begin{theorem}[\emph{A priori} bounds]
\label{DOG-thm:VFP-Ent-Fisher-Inf-Bounds}
Fix $\gamma>0$ and let $\rho\in C([0,T];\mathcal{P}(\mathbb{R}^{2d}))$ with $\rho_t|_{t=0}=:\rho_0$  
satisfy 
\begin{align}
I^\gamma(\rho)<\infty, \ \mathcal F(\rho_0)<\infty.
\end{align} 
Then for any  $t\in [0,T]$ we have
\begin{align}\label{DOG-eq:VFP-Ent+FI-Bounds}
\M{F}(\rho_t)+\frac{\gamma^2}{2}\int_0^t \RelFI(\rho_s|\mu)\,ds\leq I^\gamma(\rho)+\M{F}(\rho_0).
\end{align}
From~\eqref{DOG-eq:VFP-Ent+FI-Bounds} we obtain the separate inequality 
\begin{equation}
\label{ineq:bound-Hrho}
\frac{1}{2}\int_{\R^{2d}}  H \, d\rho_t \leq \mathcal F(\rho_0) + I^\g(\rho) + \log \frac{\int_{\R^{2d}} e^{-H/2}}{\int_{\R^{2d}} e^{-H}}.
\end{equation}

\end{theorem}

This estimate will lead to a priori bounds in two ways. First, the bound~\eqref{ineq:bound-Hrho} gives tightness estimates, and therefore compactness in space and time (Theorem~\ref{DOG-thm:VFP-Compactness}); secondly, by \eqref{DOG-eq:VFP-Ent+FI-Bounds}, the relative Fisher Information is bounded by $C/\g^2$ and therefore vanishes in the limit $\gamma\rightarrow \infty$. This fact is used to prove that the limiting measure is Maxwellian (Lemma~\ref{DOG-lem:VFP-Local-Eq}).

\begin{proof}
We give a heuristic motivation here; Appendix~\ref{DOG-sec:App-PDE} contains a full proof. 
Given a trajectory $\rho$ as in the theorem, note that by~\eqref{DOG-eq:VFP-Alternate-LDRF} $\rho$ satisfies 
\[
\partial_t \rho_t  = -\gamma \div\rho_t J\nabla (H+\psi*\rho_t) + \gamma^2 \Bigl(\div_p \rho_t\frac pm + \Delta_p \rho_t\Bigr) 
-\gamma \div_p\rho_t h_t,
\qquad\text{with }h\in L^2(0,T;L^2_\nabla(\rho)).
\]
We then formally calculate 
\begin{alignat*}2
\frac d{dt} \mathcal F(\rho_t) &= \int_{\Rd}\bigl[ \log\rho_t + 1 + H + \psi*\rho_t\bigr]
\Bigl( -\gamma \div \rho_t J\nabla (H + \psi*\rho_t) + \gamma^2 \bigl(\div_p \rho_t\frac pm &&{}+ \Delta_p \rho_t\bigr) 
  {}-\gamma \div_p\rho_t h_t\Bigr)\\
&= -\gamma^2 \int_{\Rd} \frac1{\rho_t} \left|\nabla_p \rho_t+ \rho _t\frac pm \right|^2 + \gamma \int_{\Rd} h_t \Bigl(\nabla_p \rho_t + \rho_t\frac pm\Bigr)\\
&\leq -\frac{\gamma^2}2 \int_{\Rd} \frac1{\rho_t} \left|\nabla_p \rho_t+ \rho _t\frac pm \right|^2 + \frac12 \int_{\Rd} \rho_t h_t^2,
 \end{alignat*}
where the first $O(\gamma)$ term cancels because of the anti-symmetry of $J$. After integration in time this latter expression yields~\eqref{DOG-eq:VFP-Ent+FI-Bounds}. 

For exact solutions of the VFP equation, i.e.\ when $I^\gamma(\rho) = 0$, this argument can be made rigorous following e.g.~\cite{BonillaCarrilloSoler97}. However, the fairly low regularity of the right-hand side in~\eqref{DOG-eq:VFP-Alternate-LDRF} prevents these techniques from working. `Mild' solutions, defined using the variation-of-constants formula and the Green function for the hypoelliptic operator, are not well-defined either, for the same reason: the term $\iint \nabla_p G\cdot h \, d\rho$ that appears in such an expression is generally not integrable. In the appendix we give a different proof, using the method of dual equations. 

Equation~\eqref{ineq:bound-Hrho} follows by substituting
\begin{equation*}
\mathcal{F}(\rho_t)=\RelEnt\left(\rho_t\Big|Z_{H/2}^{-1} e^{-H/2}dx\right)+\frac{1}{2}\int_{\mathbb{R}^{2d}}H\,d\rho_t +\frac{1}{2}\int_{\mathbb{R}^{2d}}\psi*\rho_t \, d\rho_t + \log \frac{\int_{\R^{2d}} e^{-H}}{\int_{\R^{2d}} e^{-H/2}},
\end{equation*}
in~\eqref{DOG-eq:VFP-Ent+FI-Bounds}, where $Z_{H/2}:=\int_{\R^{2d}} e^{-H/2}$. 
\end{proof}

\subsection{Coarse-graining and compactness}
As we described in the introduction, in the overdamped limit $\gamma\to\infty$ we expect that $\rho$ will resemble a Maxwellian distribution $Z^{-1}\exp\bigl(-{p^2}/{2m}\bigr)\sigma_t(dq)$, and that the $q$-dependent part $\sigma$ will solve equation~\eqref{DOG-eq:VFP-Intro-Limiting-Eq}. We will prove this statement using the method described in Section~\ref{sec:coarse-graining-intro}. 

It would be natural to define `coarse-graining' in this context as the projection $\xi(q,p) := q$, since that should eliminate the fast dynamics of $p$ and focus on the slower dynamics of $q$. However, this choice fails: it completely decouples the dynamics of $q$ from that of $p$, thereby preventing the noise in $p$ from transferring to $q$. Following the lead of Kramers~\cite{Kramers1940}, therefore, we define a slightly different coarse-graining map 
\begin{equation}
\label{DOG-def:xi1}
\xi^\gamma:\mathbb{R}^{2d}\rightarrow\mathbb{R}^d, \qquad \xi^\gamma(q,p):=q+\frac p\gamma.
\end{equation}
In the limit $\gamma\to\infty$,  $\xi^\g\rightarrow\xi$ locally uniformly, recovering the projection onto the $q$-coordinate.

The theorem below gives the compactness properties of the solutions $\rho^\g$ of the rescaled VFP equation that allow us to pass to the limit. There are two levels of compactness, a weaker one in the original space $\R^{2d}$, and a stronger one in the coarse-grained space $\R^d = \xi^\g(\R^{2d})$. This is similar to other multilevel compactness results as in e.g.~\cite{GrunewaldOttoVillaniWestdickenberg09}.

\begin{theorem}[Compactness]\label{DOG-thm:VFP-Compactness}
Let a sequence $\rho^\gamma\in C([0,T];\mathcal{P}(\mathbb{R}^{2d}))$
satisfy for a suitable constant $C>0$ and every $\gamma$ the estimate
\begin{align}\label{DOG-eq:Basic-Ass}
I^\gamma(\rho^\gamma)+ \mathcal{F}(\rho_t^\gamma |_{t=0})\leq C.
\end{align} 
Then there exist a subsequence (not relabeled) such that 
\begin{enumerate}
\item\label{thm:compactness-gamma:item:R2d}
$\rho^\gamma\rightarrow\rho$ in $\mathcal{M}([0,T]\times\mathbb{R}^{2d})$  with respect to the narrow topology. 
\item\label{thm:compactness-gamma:item:Rd}$\xi^\gamma_{\#}\rho^\gamma\rightarrow \xi_\#\rho$ in $C([0,T];\mathcal{P}(\mathbb{R}^d))$ with respect to the uniform topology in time and narrow topology on $\mathcal{P}(\mathbb{R}^d)$. 
\end{enumerate}
For a.e. $t\in [0,T]$ the limit $\rho_t$ satisfies
\begin{align}\label{DOG-eq:VFP-Ent+FI-Bounds2}
\RelFI(\rho_t|\mu)=0
\end{align}
\end{theorem}

\begin{proof}
To prove part~\ref{thm:compactness-gamma:item:R2d}, note that the positivity of the convolution integral involving $\psi$ and the free-energy-dissipation inequality~\eqref{DOG-eq:VFP-Ent+FI-Bounds} imply that 
$\RelEnt(\rho_t^\gamma|Z^{-1}_He^{-H}dx)$ is bounded uniformly in $t$ and $\g$. By an argument as in~\cite[Prop.~4.2]{AmbrosioSavareZambotti09} this implies that the set of space-time measures $\{\rho^\g: \g>1\}$ is tight, from which compactness in $\mathcal{M}([0,T]\times\mathbb{R}^{2d})$ follows. 

To prove~\eqref{DOG-eq:VFP-Ent+FI-Bounds2} we remark that 
\[
0\leq\sup_{\varphi\in C_c^\infty(\R\times\Rd)} 
2\int_0^T\!\!\int_\Rd \Bigl[ \Delta_p \varphi - \frac pm \nabla_p \varphi - \frac12 |\nabla_p\varphi|^2\Bigr]\, d\rho_t^\g dt 
\leq \int_0^T \RelFI(\rho^\g_t|\mu)\, dt \leq \frac{C}{\g^2} \stackrel{\g\to\infty}\longrightarrow 0,
\]
and by passing to the limit on the left-hand side we find
\[
\sup_{\varphi\in C_c^\infty(\R\times\Rd)} 
2\int_0^T\!\!\int_\Rd \Bigl[ \Delta_p \varphi - \frac pm \nabla_p \varphi - \frac12 |\nabla_p\varphi|^2\Bigr]\, d\rho_t dt = 0.
\]
By disintegrating $\rho$ in time as $\rho(dtdqdp) = \rho_t(dqdp)dt$, we find that $\RelFI(\rho_t|\mu) = 0$ for (Lebesgue-) almost all $t$.

We prove part~\ref{thm:compactness-gamma:item:Rd} with the Arzel\`a-Ascoli theorem. For any $t\in [0,T]$ the sequence $\xi^\g_\#\rho^\g_t$ is tight, which follows from the tightness of $\rho_t^\g$ proved above and the local uniform convergence $\xi^\g\rightarrow\xi$ (see e.g.~\cite[Lemma 5.2.1]{AmbrosioGigliSavare08}).

To prove equicontinuity we will show that  
\begin{align}\label{DOG-eq:Intermediate-Compactness-Step}
\sup _{\g> 1} \;\sup\limits_{t\in[0,T-h]}\sup_{\substack{\varphi\in C_c^2(\rd) \\ \|\varphi\|_{C^2(\rd)}\leq 1}} \int_{\rd} \varphi(\xi^\g_\#\rho^\g_{t+h}-\xi^\g_\#\rho^\g_{t})\xrightarrow{h\rightarrow 0}0.
\end{align}
In fact, \eqref{DOG-eq:Intermediate-Compactness-Step} is a direct consequence of the following stronger statement
\begin{align}
\label{DOG-eq:HoelderWasserstein}
\int_{\rd} \varphi(\xi^\g_\#\rho^\g_{t+h}-\xi^\g_\#\rho^\g_{t})\leq C\|\nabla\varphi\|_{\infty}\sqrt{h}
\end{align}
with $C$ independent of $t,\g$ and $\varphi$. Note that \eqref{DOG-eq:HoelderWasserstein} in particular implies a uniform $1/2$-H\"older estimate with respect to the $L^1$-Wasserstein distance. 

Let us now give the proof of \eqref{DOG-eq:HoelderWasserstein}. Indeed, the boundedness of the rate functional, definition \eqref{DOG-eq:VFP-Alternate-LDRF}, and tightness of $\rho^\g$ imply that there exists some $h^\g\in L^2(0,T;L^2_\nabla(\rho^\g_t))$ with
\begin{align}\label{DOG-eq:Bounded-Rate-Fn-PDE-Desc}
 \partial_t\rho^\g_t=(\mathcal{L}_{\rho^\g_t})^*\rho^\g_t-\g\div_p(\rho^\g_t h^\g_t).
\end{align}
in duality with $C_b^2(\R^{2d})$, pointwise almost everywhere in $t\in[0,T]$. 
Therefore for any $f\in C^2_b(\Rd)$ we have in the sense of distributions on $[0,T]$,
\begin{align*}
\frac{d}{dt}\int_{\Rd}f\rho^\g_t=\int_{\Rd}\bigg{(}\gamma\frac pm \cdot \nabla _q f - \gamma\nabla_q V\cdot\nabla_pf 
-\gamma\nabla_pf\cdot(\nabla_q\psi\ast \rho^\g) 
- \gamma^2\frac pm\cdot \nabla_p f 
+\gamma^2 \Delta_p f +\g\nabla_pf\cdot h_t^\g)
\bigg{)}d\rho^\g_t.
\end{align*} 
To prove \eqref{DOG-eq:HoelderWasserstein}, make the choice $f=\varphi\circ\xi^\g$  for $\varphi\in C^2_c(\rd)$ and integrate over $[t,t+h]$. 
Note that due to the specific form of $\xi^\g=q+p/\g$ the terms $\gamma\frac pm \cdot \nabla _q f$ and $\gamma^2\frac pm\cdot \nabla_p f$ cancel and therefore
\begin{multline*}
\int_{\rd}\varphi(\xi^\g_\#\rho^\g_{t+h}-\xi^\g_\#\rho^\g_{t})=\int_t^{t+h}\int_{\Rd}\bigg{(}-\nabla V(q)\cdot\nabla\varphi\left(q+\frac{p}{\g}\right)-(\nabla_q\psi\ast\rho_s^\g)(q)\cdot\nabla\varphi\left(q+\frac{p}{\g}
\right)\\
+\Delta\varphi\left(q+\frac{p}{\g}\right)+\nabla\varphi\left(q+\frac{p}{\g}\right)\cdot h_s^\g(q,p)
\bigg{)}d\rho^\g_s\,ds.
\end{multline*}
We estimate the first term on the right hand side by using H\"older's inequality and growth condition (V1),
\begin{align*}
\left|\int_t^{t+h}\int_{\Rd}\nabla V(q)\cdot\nabla\varphi\left(q+\frac{p}{\g}\right)d\rho^\g_s\,ds\right|
&\leq\|\nabla\varphi\|_{\infty}\sqrt{h}\left(\int_t^{t+h}\int_{\Rd}|\nabla V(q)|^2d\rho^\g_s\,ds\right)^{1/2}\\
&\leq\|\nabla\varphi\|_{\infty}\sqrt{h}\left(\int_t^{t+h}\int_{\Rd}C(1+V(q))\rho^\g_s\,ds\right)^{1/2}
\leq \tilde C\|\nabla\varphi\|_{\infty}\sqrt{h},
\end{align*}
where the last inequality follows from the free-energy-dissipation inequality~\eqref{DOG-eq:VFP-Ent+FI-Bounds}. 
For the second term we use $|\nabla_q\psi\ast\rho^\g_s|\leq \|\nabla_q\psi\|_\infty$ and the last term is estimated by H\"older's 
inequality,
\begin{align*}
\left|\int_t^{t+h}\int_{\Rd}\nabla\varphi\left(q+\frac{p}{\g}\right)h_s^\g(q,p)d\rho^\g_sds\right|
&\leq\|\nabla\varphi\|_\infty \sqrt{h}\bigg{(}\int_t^{t+h}\int_{\Rd}|h_s^\g|^2d\rho^\g_sds\bigg{)}^\frac{1}{2}
\\
&\leq \|\nabla\varphi\|_\infty \sqrt{h}\,\left(2I^\g(\rho^\g)\right)^{\frac12}\leq C \|\nabla\varphi\|_\infty\sqrt{h}.
\end{align*}
To sum up we have \begin{align*}
\bigg{|}\int_{\rd}\varphi(\xi^\g_\#\rho^\g_{t+h}-\xi^\g_\#\rho^\g_{t})\bigg{|}\leq C\|\nabla\varphi\|_\infty\sqrt{h}\xrightarrow{h\rightarrow 0}0,
\end{align*}
where $C$ is independent of $t,\g$ and $\varphi$.


Thus by the Arzel\`a-Ascoli theorem there exists a $\nu \in C([0,T];\mathcal{P}(\mathbb{R}^d))$ such that $\xi^\gamma_{\#}\rho^\gamma\rightarrow \nu$ with respect to uniform topology in time and narrow topology on $\mathcal{P}(\mathbb{R}^d)$. Since $\rho^\g\rightarrow\rho$ in $\M{M}([0,T]\times\Rd)$ and 
$\xi^\g\rightarrow\xi$ locally uniformly, we have $\xi^\g_\#\rho^\g\rightarrow\xi_\#\rho$ in $\M{M}([0,T]\times\rd)$ (again using~\cite[Lemma 5.2.1]{AmbrosioGigliSavare08}), implying that $\nu=\xi_\#\rho$.
This concludes the proof of Theorem~\ref{DOG-thm:VFP-Compactness}.
\end{proof}

\subsection{Local equilibrium}

A central step in any coarse-graining method is the treatment of the information that is `lost' upon coarse-graining. The lemma below uses the a priori estimate~\eqref{DOG-eq:VFP-Ent+FI-Bounds} to reconstruct this information, which for this system means showing that $\rho^\g$ becomes Maxwellian in $p$ as $\g\to\infty$.

\begin{lemma}[Local equilibrium]\label{DOG-lem:VFP-Local-Eq}
Under the assumptions of Theorem~\ref{DOG-thm:VFP-Compactness}, let $\rho^\gamma\rightarrow\rho$ in 
$\mathcal{M}([0,T]\times\mathbb{R}^{2d})$ with respect to the narrow topology and $\xi^\g_\#\rho^\g\rightarrow \xi_\#\rho$ in $C([0,T];\mathcal{P}(\mathbb{R}^d))$ with respect to 
the uniform topology in time and narrow topology on $\mathcal{P}(\mathbb{R}^d)$. Then there exists 
$\sigma\in C([0,T];\mathcal{P}(\mathbb{R}^d))$, $\sigma(dtdq) = \sigma_t(dq)dt$, such that
for almost all $t\in[0,T]$, 
\begin{align}
\label{char:Maxwellian}
\rho_t(dqdp)=Z^{-1}\exp\left(-\frac{p^2}{2m}\right)\sigma_t(dq)dp,
\end{align} 
where $Z = \int_{\R^d} e^{-p^2/2m}dp$ is the normalization constant for the Maxwellian distribution. 
Furthermore $\xi^\gamma_\#\rho^\gamma\rightarrow \sigma$ uniformly in time and narrowly on $\mathcal{P}(\mathbb{R}^d)$.
\end{lemma}
\begin{proof}
Since $\rho^\g\rightarrow \rho$ narrowly in $\mathcal{M}([0,T]\times\mathbb{R}^{2d})$, the limit $\rho$ also has the disintegration structure $\rho(dtdpdq) = \rho_t(dpdq)dt$, with $\rho_t\in \M P(\R^{2d})$.  From the \emph{a priori} estimate~\eqref{DOG-eq:VFP-Ent+FI-Bounds} and the duality definition of $\RelFI$ we have $\RelFI(\rho_t|\mu)=0$ for almost all $t$, and the characterization~\eqref{char:Maxwellian} then follows from Lemma~\ref{DOG-lem:zero-RFI}.
The uniform in time convergence of $\xi^\g_\#\rho^\g$ implies $\xi^\gamma_\#\rho^\gamma\rightarrow\xi_\#\rho=\sigma$ 
uniformly in time and narrowly on $\mathcal{P}(\mathbb{R}^d)$ and the regularity $\sigma\in C([0,T];\mathcal{P}(\mathbb{R}^d))$.
\end{proof}

\subsection{Liminf inequality}

The final step in the variational technique is proving an appropriate liminf inequality which also provides the structure of the limiting coarse-grained evolution. The following theorem makes this step rigorous.

Define the (limiting) functional $I:C([0,T];\mathcal{P}(\mathbb{R}^d))\rightarrow\mathbb{R}$
by 
\begin{multline}
\label{def:gamma:I}
 I(\sigma):=\sup_{g\in C_b^{1,2}(\R\times\R^d)}\int_{\mathbb{R}^{d}}g_Td\sigma_T
-\int_{\mathbb{R}^{d}}g_0d\sigma_0
-\int_0^T\int_{\mathbb{R}^{d}}
\Bigl(\partial_t g-\nabla V\cdot\nabla g
-(\nabla\psi\ast\sigma)\cdot\nabla g+\Delta g\Bigr)d\sigma_tdt\\
-\frac{1}{2}\int_0^T\int
_{\mathbb{R}^{d}}\left|\nabla g\right|^2d\sigma_tdt.
\end{multline}

Note that $I\geq0$ (since $g=0$ is admissible); we have the equivalence
\[
I(\sigma)=0 \quad \Longleftrightarrow \quad \partial_t\sigma=\div \sigma\nabla V(q) +\div \sigma(\nabla\psi\ast\sigma)+\Delta\sigma \qquad \text{in }[0,T]\times \R^d.
\]

\begin{theorem}[Liminf inequality]
\label{DOG-thm:VFP-Liminf-Inequality}
Under the same conditions as in Theorem~\ref{DOG-thm:VFP-Compactness} we assume that $\rho^\gamma\rightarrow\rho$  narrowly in $\mathcal{M}([0,T]\times\mathbb{R}^{2d})$ and $\xi^\gamma_\#\rho^\gamma\rightarrow\xi_\#\rho\equiv\sigma$ in $C([0,T];\mathcal{P}(\rd))$.
Then 
\begin{align*}
  \liminf\limits_{\g\rightarrow \infty}I^\gamma(\rho^\gamma)\geq I(\sigma).
\end{align*}
\end{theorem}
\begin{proof}
Write the large deviation rate functional $I^\gamma:C([0,T];\mathcal{P}(\mathbb{R}^{2d}))\rightarrow\mathbb{R}$ in~\eqref{def:I-gamma} as
\begin{align}\label{DOG-eq:VFP-Rate-Fn-Abstract-Form}
I^\g(\rho)=\sup\limits_{f\in C^{1,2}_b(\mathbb{R}\times\Rd)} \LDJ^\g(\rho,f),
\end{align}
where
\begin{multline*}
\LDJ^\gamma(\rho,f)=
\int_{\mathbb{R}^{2d}}f_Td\rho_T
-\int_{\mathbb{R}^{2d}}f_0d\rho_0
-\int_0^T\int_{\mathbb{R}^{2d}}
\bigg{(}\partial_t f+\gamma\frac pm \cdot \nabla _q f - \gamma\nabla_q V\cdot\nabla_pf 
-\gamma\nabla_pf\cdot(\nabla_q\psi\ast \rho_t) \nonumber\\
- \gamma^2\frac pm\cdot \nabla_p f 
+\gamma^2 \Delta_p f\bigg{)}d\rho_tdt
-\frac{\gamma^2}{2}\int_0^T\int_{\mathbb{R}^{2d}}\left|\nabla_p f\right|^2d\rho_tdt.
\end{multline*}
Define $\M{A}:=\{f=g\circ\xi^\g \text{ with }g\in C_b^{1,2}(\mathbb{R}\times\rd)\}$. Then we have

\begin{align*}
I^\g(\rho^\g)\geq \sup\limits_{f\in\M{A}} \LDJ^\g(\rho^\g,f),
\end{align*}
and
\begin{multline}
\label{DOG-eq:VFP-Liminf-Intermediate-Var-Form}
\LDJ^\gamma(\rho^\g,g\circ\xi^\g)=
\int_{\mathbb{R}^{2d}}g_T\circ\xi^\g d\rho_T^\g
-\int_{\mathbb{R}^{2d}}g_0\circ\xi^\g d\rho_0^\g
-\int_0^T\int_{\mathbb{R}^{2d}}
\bigg{[}\partial_t (g\circ\xi^\g) -\nabla_q V(q)\cdot\nabla g\bigg{(}q+\frac{p}{\g}\bigg{)} \\
+ \Delta g\bigg{(}q+\frac{p}{\g}\bigg{)} -\nabla g\bigg{(}q+\frac{p}{\g}\bigg{)}\cdot(\nabla_q\psi\ast \rho_t^\g)(q) 
\bigg{]}d\rho_t^\g dt
-\frac{1}{2}\int_0^T\int
_{\mathbb{R}^{2d}}\left|\nabla (g\circ\xi^\g)\right|^2d\rho_t^\g dt.
\end{multline}
Note how the specific dependence of $\xi^\gamma(q,p) = q+p/\g$ on $\g$ has caused the coefficients $\g$ and $\g^2$ in the expression above to vanish. 
Adding and subtracting $\nabla V(q+p/\g)\cdot\nabla g(q+p/\g)$ in \eqref{DOG-eq:VFP-Liminf-Intermediate-Var-Form} and defining $\hrho^\g:=\xi^\g_\#\rho^\g$, $\LDJ^\g$ can be rewritten as
\begin{equation}\label{DOG-eq:VFP-LDRF-Push-Forward-Calculation}
\begin{aligned}
\LDJ^\gamma&(\rho,g\circ\xi^\g)=
\int_{\rd}g_Td\hrho_T^\g
-\int_{\rd}g_0d\hrho_0^\g
-\int_0^T\int_{\rd}
\left(\partial_t g -\nabla V\cdot\nabla g + \Delta g\right)(\zeta)\hrho^\g_t(d\zeta)dt-\frac{1}{2}\int_0^T\int_{\rd}\left|\nabla g\right|^2d\hrho_t^\g dt\\
&-\int_0^T\int_{\Rd}\bigg{(}
\nabla V\bigg{(}q+\frac{p}{\g}\bigg{)}-\nabla V(q)
\bigg{)}\cdot\nabla g\bigg{(}q+\frac{p}{\g}\bigg{)}d\rho^\g_t dt
+\int_0^T\int_{\Rd}\nabla g\bigg{(}q+\frac{p}{\g}\bigg{)}\cdot(\nabla_q\psi\ast \rho_t^\g)(q) d\rho_t^\g dt.
\end{aligned}
\end{equation}

We now show that~\eqref{DOG-eq:VFP-LDRF-Push-Forward-Calculation} converges to the right-hand side of~\eqref{def:gamma:I}, term by term.
Since $\xi^\g_\#\rho^\g\rightarrow\xi_\#\rho=\sigma$ narrowly in $\M{M}([0,T]\times\Rd)$ and $g\in C^{1,2}_b(\mathbb{R}\times\rd)$ we have
\begin{align*}
\int_0^T\int_{\rd}
\Bigl(\partial_t g -\nabla V\cdot\nabla g + \Delta g+\frac{1}{2}|\nabla g|^2\Bigr)d\hrho^\g_tdt\xrightarrow{\g\rightarrow \infty}\int_0^T\int_{\rd}
\Bigl(\partial_t g -\nabla V\cdot\nabla g + \Delta g+\frac{1}{2}|\nabla g|^2\Bigr)d\sigma_tdt.
\end{align*}
Taylor expansion of $\nabla V$ around $q$ and estimate~\eqref{ineq:bound-Hrho} give
\begin{multline*}
\left|\,\int_0^T\int_{\R^{2d}}\bigg{(}
\nabla V\bigg{(}q+\frac{p}{\g}\bigg{)}-\nabla V(q)
\bigg{)}\cdot\nabla g\bigg{(}q+\frac{p}{\g}\bigg{)}d\rho^\g_t dt\;\right|\leq \\
\leq \|D^2V\|_{\infty}\|\nabla g\|_{\infty}\sqrt{T}\left(\int_0^T\int_{\R^{2d}}\frac{p^2}{\gamma^2}d\rho^\g_t dt\right)^{1/2}
\leq\frac{C}{\gamma}\xrightarrow{\g\rightarrow\infty }0.
\end{multline*}

Adding and subtracting $\nabla g(q)\cdot(\nabla_q\psi\ast\rho_t^\g)(q)  $  in \eqref{DOG-eq:VFP-LDRF-Push-Forward-Calculation} we find
\begin{multline*}
\int_0^T\int_{\Rd}\nabla g\bigg{(}q+\frac{p}{\g}\bigg{)}\cdot(\nabla_q\psi\ast \rho_t^\g)(q) d\rho_t^\g dt=\int_0^T\int_{\Rd}\nabla g(q)\cdot(\nabla_q\psi\ast \rho_t^\g)(q) d\rho_t^\g dt\\
+\int_0^T\int_{\Rd}\bigg{[}\nabla g\bigg{(}q+\frac{p}{\g}\bigg{)}-\nabla g(q)\bigg{]}\cdot(\nabla_q\psi\ast \rho_t^\g)(q) d\rho_t^\g dt.
\end{multline*}
Since $\rho^\g\rightarrow\rho$ we have $\rho^\g\otimes\rho^\g\rightarrow\rho\otimes\rho$ and therefore passing to the limit in the first term  and using the local-equilibrium characterization of Lemma~\ref{DOG-lem:VFP-Local-Eq}, we obtain
\begin{align*}
\int_0^T\int_{\Rd}\nabla g(q)\cdot(\nabla_q\psi\ast \rho^\g)(q) \,d\rho_t^\g dt\xrightarrow{\g\rightarrow 0}
\int_0^T\int_{\rd}\nabla g\cdot(\nabla\psi\ast \sigma)\, d\sigma_t dt.
\end{align*}


For the second term we calculate
\begin{multline*}
 \left|\int_0^T\int_{\Rd}\bigg{[}\nabla g\bigg{(}q+\frac{p}{\g}\bigg{)}-\nabla g(q)\bigg{]}\cdot(\nabla_q\psi\ast \rho^\g)(q) d\rho_t^\g dt\right|\leq \\
\leq \|D^2 g\|_{\infty}\|\nabla_q \psi\|_\infty\sqrt{T}\left(\int_0^T\int_{\R^{2d}}\frac{p^2}{\gamma^2}d\rho^\g_t dt\right)^{1/2}
\leq\frac{C}{\gamma}\xrightarrow{\g\rightarrow\infty }0.
\end{multline*}

Therefore
\begin{align*}
\int_0^T\int_{\Rd}\nabla g\bigg{(}q+\frac{p}{\g}\bigg{)}\cdot(\nabla_q\psi\ast& \rho^\g)(q) d\rho_t^\g dt\xrightarrow{\g\rightarrow\infty}\int_0^T\int_{\rd}\nabla g\cdot(\nabla\psi\ast \sigma)\, d\sigma_t dt.
\end{align*} 
\end{proof}

\subsection{Discussion}

The ingredients of the convergence proof above are, as mentioned before, (a) a compactness result, (b) a local-equilibrium result, and (c) a liminf inequality. All three follow from the large-deviation structure, through the rate functional $I^\g$. We now comment on these. 

\medskip

\emph{Compactness.} Compactness in the sense of measures is, both for $\rho^\g$ and for $\xi^\g_\#\rho^\g$, a simple consequence of the confinement provided by the growth of $H$. In Theorem~\ref{DOG-thm:VFP-Compactness} we provide a stronger statement for $\xi^\g_\#\rho^\g$, by showing continuity in time, in order for the limiting functional $I(\sigma)$ in~\eqref{def:gamma:I} to be well defined. This continuity depends on the boundedness of $I^\g$.

\emph{Local equilibrium.}
The local-equilibrium statement depends crucially on the structure of $I^\g$, and more specifically on the large coefficient $\g^2$ multiplying the derivatives in $p$. This coefficient also ends up as a prefactor of the relative Fisher Information in the \emph{a priori} estimate~\eqref{DOG-eq:VFP-Ent+FI-Bounds}, and through this estimate it drives the local-equilibrium result. 

\emph{Liminf inequality.} As remarked in the introduction, the duality structure of $I^\g$ is the key to the liminf inequality, as it allows for relatively weak convergence of $\rho^\g$ and $\xi^\g_\#\rho^\g$. The role of the local equilibrium is to allow us to replace the $p$-dependence in some of the integrals by the Maxwellian dependence, and therefore to reduce all terms to dependence on the macroscopic information $\xi^\g_\#\rho^\g$ only.

\medskip

As we have shown, the choice of the coarse-graining map has the advantage that it has caused the (large) coefficients $\g$ and $\g^2$ in the expression of the rate functionals to vanish. In other words, it cancels out the inertial effects and transforms a Laplacian in $p$ variable to a Laplacian in the coarse-grained variable while rescaling it to be of order 1. The choice $\xi(q,p)=q$, on the other hand, would lose too much information by completely discarding the diffusion. 

\section{Diffusion on a graph in one dimension}
\label{DOG-sec:DOG}

In this section we derive the small-noise limit of a randomly perturbed Hamiltonian system, which corresponds to passing to the limit $\vep\rightarrow 0$ in \eqref{DOG-eq:Intro-DOG}. In terms of a rescaled time, in order to focus on the time scale of the noise, equation~\eqref{DOG-eq:Intro-DOG} becomes
\begin{align}
\partial_t\rho^\vep=-\frac{1}{\vep}\div(\rho^\vep J\nabla H)+\Delta_p\rho^\vep. 
\label{DOG-eq:Ran-Ham-Evo-R2}
\end{align}
Here $\rho^\e\in C([0,T],\mathcal{P}(\mathbb{R}^{2}))$, $J = \begin{psmallmatrix}0 &1\\-1& 0\end{psmallmatrix}$ is again the canonical symplectic matrix, $\Delta_p$ 
is the Laplacian in the $p$-direction, and the equation holds in the sense of distributions. The Hamiltonian $H\in C^2(\Rd;\mathbb{R})$ is again defined by $H(q,p)=p^2/2m+V(q)$ for some potential $V:\mathbb{R}^d\rightarrow\mathbb{R}$. We make the following assumptions (that we formulate on $H$ for convenience): 
\begin{enumerate}[label=(\text{A}\arabic*)]
\item{$H\geq 0$,  and $H$ is coercive, i.e. $H(x)\xrightarrow{|x|\rightarrow\infty}\infty$;}
\item\label{cond:DOG-H-Growth}{$|\nabla H|,|\Delta H|,|\nabla_p H|^2\leq C(1+H)$};
\item\label{cond:DOG-H-non-deg}
{$H$ has a finite number of non-degenerate (i.e.\ non-singular Hessian) saddle points $O_1,\ldots,O_n$ with $H(O_i)\neq H(O_j)$ for every $i,j\in\{1,\ldots,n\}$, $i\not=j$.}
\end{enumerate} 

As explained in the introduction, and in contrast to the VFP equation of the previous section, equation~\eqref{DOG-eq:Ran-Ham-Evo-R2} has two equally valid interpretations: as a PDE in its own right, or as the Fokker-Planck (forward Kolmogorov) equation of the stochastic process
\begin{equation}
\label{DOG-eq:DOGSDE}
X^\vep = \begin{pmatrix} Q^\vep\\P^\vep\end{pmatrix}, \qquad dX^\vep_t = \frac 1\vep J\nabla H(X^\vep_t)dt + \sqrt 2\, \begin{pmatrix}0\\1\end{pmatrix} dW_t.
\end{equation}
For the sequel we will think of $\rho^\vep$ as the law of the process $X_t^\vep$; although this is not strictly necessary, it helps in illustrating the ideas.

\subsection{Construction of the graph $\Gamma$}
\label{DOG-sec:Graph}
\begin{figure}[t]
\centering
\includegraphics[scale=0.55]{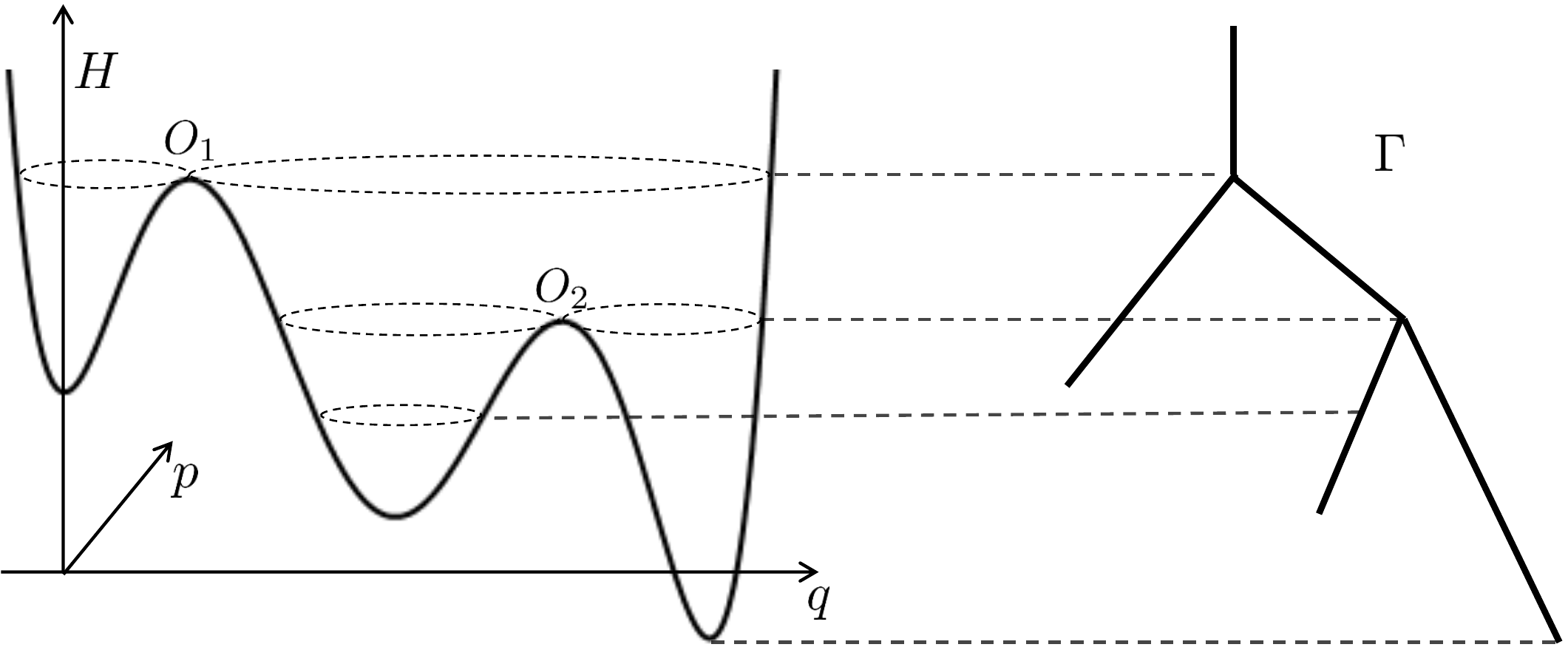}
\caption[Hamiltonian]{Left: Hamiltonian $\mathbb{R}^2\ni(q,p)\mapsto H(q,p)$, Right: Graph $\Gamma$}
\label{DOG-fig:Hamiltonian-LevSet}
\end{figure}

As mentioned in the introduction, the  dynamics of \eqref{DOG-eq:Ran-Ham-Evo-R2} has two time scales when $0<\vep\ll 1$, a fast and  a slow one. The fast time scale, of scale $\e$, is described by the (deterministic) equation 
\begin{equation}
\label{eq:Hamiltonian}
\dot x = \frac1\e J\nabla H(x)\qquad \text{in $\R^2$}, 
\end{equation}
whereas the slow time scale, of order $1$, is generated by the noise term.

The solutions of~\eqref{eq:Hamiltonian} follow level sets of~$H$. There exist three types of such solutions: stationary ones, periodic orbits, and homoclinic orbits. Stationary solutions of~\eqref{eq:Hamiltonian} correspond to stationary points of $H$ (where $\nabla H = 0$); periodic orbits to connected components of level sets along which $\nabla H \not=0$; and homoclinic orbits to components of level sets of $H$ that are terminated on each end by a stationary point. Since we have assumed in (A3) that there is at most one stationary point in each level sets, heteroclinic orbits do not exist, and the orbits necessarily connect a stationary point with itself. 

Looking ahead towards coarse-graining, we define $\Gamma$ to be the set of all connected components of level sets of $H$, and we identify $\Gamma$ with a union of one-dimensional line segments, as shown in Figure~\ref{DOG-fig:Hamiltonian-LevSet}. Each periodic orbit corresponds to an interior point of one of the edges of $\Gamma$; the vertices of $\Gamma$ correspond to connected components of level sets containing a stationary point of $H$. Each saddle point $O$ corresponds to a vertex connected by three edges.  

For practical purposes we also introduce a coordinate system on $\Gamma$. We represent the edges by closed intervals $I_k\subset \R$, and number them with numbers $k=1,2,
\ldots,n$; the pair $(h,k)$ is then a coordinate for a point $\g\in\Gamma$, if $k$ is the index of the edge containing $\g$, and $h$ the value of $H$ on the level set represented by~$\g$. For a vertex $O\in \Gamma$, we 
write $O\sim I_k$ if $O$ is at one end of edge $I_k$; we use the shorthand notation $\pm^{}_{kj}$ to mean $1$ if $O_j$ is at the upper end of $I_k$, and $-1$ in the other case. Note that if $O\sim I_{k_1}$, $O\sim 
I_{k_2}$ and $O\sim I_{k_3}$ and $h_0$ is the value of $H$ at the point corresponding to $O$, then 
the coordinates $(h_0,k_1)$, $(h_0,k_2)$ and $(h_0,k_3)$ correspond to the same point $O$. 
With a slight abuse of notation, we also define the function $k:\mathbb{R}^2\rightarrow \{1,\ldots,n\}$ as the index of the edge $I_k\subset\Gamma$ corresponding to the component containing $(q,p)$.

The rigorous construction of the graph $\Gamma$ and the topology on it has been done several times~\cite{FreidlinWentzell93,Freidlin1994,Barret2014}; for our purposes it suffices to note that (a) inside each edge, the usual topology and geometry of $\R^1$ apply, and (b) across the whole graph there is a natural concept of distance, and therefore of continuity. It will be practical to think of functions $f:\Gamma\to\R$ as defined on the disjoint union $\sqcup_k I_k$. 
A function $f:\Gamma\rightarrow \mathbb{R}$ is then called well-defined if it is a single-valued function on $\Gamma$ (i.e., it takes the same value on those vertices that are multiply represented). A well-defined function $f:\Gamma\rightarrow \mathbb{R}$ is \emph{continuous} if $f|_{I_k}\in C(I_k)$ for every $k$. 

We also define a concept of \emph{differentiability} of a function $f:\Gamma\to\R$. A \emph{subgraph} of $\Gamma$ is defined as any union of edges such that each interior vertex connects exactly two edges, one from above and one from below---i.e., a subtree without bifurcations.  A continuous function on $\Gamma$ is called {differentiable} on $\Gamma$ if it is differentiable on each of its subgraphs. 

Finally, in order to integrate over $\Gamma$, we write $d\gamma$ for the measure on $\Gamma$ which is defined on each $I_k$ as the local Lebesgue measure $dh$. Whenever we write $\int_\Gamma$, this should be interpreted as $\sum_k \int_{I_k}$.

\subsection{Adding noise: diffusion on the graph}

In the noisy evolution~\eqref{DOG-eq:DOGSDE}, for small but finite $\vep>0$, the evolution follows fast trajectories that nearly coincide with the level sets of $H$; the noise breaks the conservation of $H$, and causes a slower drift of $X_t$ across the levels of $H$. In order to remove the fast deterministic dynamics, we now define the coarse-graining map as 
\begin{equation}
\label{def:xi-DOG}
\xi: \R^2 \to \Gamma, \quad \xi(q,p):=(H(q,p),k(q,p)),
\end{equation}
where the mapping $k:\mathbb{R}^2\rightarrow\{1,\ldots,n\}$ indexes the edges of the graph, as above. 

We now consider the process $\xi(X_t^\vep)$, which contains no fast dynamics. For each finite $\vep>0$, $\xi(X^\vep_t)$ is not a Markov process; but as $\e\to0$, the fast movement should result in a form of averaging, such that the influence of the missing information vanishes; then the limit process is a diffusion on the graph $\Gamma$. 

The results of this section are  stated and proved in terms of the corresponding objects $\rho^\e$ and $\hrho^\e$, where  $\hrho^\e$ is the push-forward
\begin{equation}
\label{def:hrho}
\hrho^\vep := \xi_\# \rho^\vep,
\end{equation}
as explained in Section~\ref{sec:coarse-graining-intro}, and similar to Section~\ref{sec:overdamped-limit}. The corresponding statement about $\rho^\vep$ and $\hrho^\vep$ is that $\hrho^\e$ should converge to some $\hrho$, which in the limit satisfies a (convection-) diffusion equation on $\Gamma$. Theorems~\ref{DOG-thm:DOG-Compact} and~\ref{DOG-eq:DOG-Liminf-Theorem} make this statement precise.

\subsection{Compactness}
As in the case of the VFP equation, equation~\eqref{DOG-eq:Ran-Ham-Evo-R2} has a free energy, which in this case is simply the Boltzmann entropy
\begin{equation}
\label{def:FreeEnergyDOG}
\mathcal F(\rho) = \int_{\R^2}\rho\log\rho\,\Lebesgue^2,
\end{equation}
where $\Lebesgue^2$ denotes the two dimensional Lebesgue measure in $\R^2$.

The corresponding `relative' Fisher Information is the same as the Fisher Information in the $p$-variable, 
\begin{align}\label{DOG-def:DOG-RelFI}
\RelFI(\rho|\Lebesgue^2) = \sup_{\varphi\in C_c^\infty(\R^{2})}
2\int_{\R^{2}}\Bigl[ \Delta_p \varphi - \frac12 |\nabla_p\varphi|^2\Bigr]\, d\rho,
\end{align}
and satisfies for $\rho = f\Lebesgue^2$,
\[
\RelFI (f\Lebesgue^2|\Lebesgue^2) = \int_{\R^2} |\nabla_p \log f |^2 \,f \,dqdp,
\]
whenever this is finite.

\medskip
The large deviation functional $I^\vep:C([0,T];\mathcal{P}(\mathbb{R}^2))\rightarrow\mathbb{R}$ is given by
\begin{align}\label{DOG-eq:DOG-Large-Dev-Rate-Fn}
I^\vep(\rho)=\sup\limits_{f\in C_c^{1,2}(\mathbb{R}\times\mathbb{R}^{2})}\bigg{[} 
\int\limits_{\mathbb{R}^{2}}f_Td\rho_T
-\int\limits_{\mathbb{R}^{2}}f_0d\rho_0
-\int\limits_0^T\int\limits_{\mathbb{R}^{2}}
(\partial_t f+\frac{1}{\vep}J\nabla H\cdot\nabla f+\Delta_p f)d\rho_tdt
-\frac{1}{2}\int\limits_0^T\int\limits_{\mathbb{R}^{2}}\left|\nabla_p f\right|^2d\rho_tdt\bigg{]}.
\end{align}

For fixed $\vep>0$, $\rho^\vep$ solves \eqref{DOG-eq:Ran-Ham-Evo-R2} iff $I^\vep(\rho^\vep)=0$.

\medskip

The following theorem states the relevant \emph{a priori} estimates in this setting. 
 
\begin{theorem}[A priori estimates]\label{DOG-thm:DOG-Compactness}
Let $\vep>0$ and let $\rho\in C([0,T];\mathcal{P}(\mathbb{R}^{2}))$ with $\rho_t|_{t=0}=:\rho_0$
satisfy 
\begin{align*}
I^\e(\rho)+ \mathcal F(\rho_0) + \int _{\R^2}H\,d\rho_0\leq C.
\end{align*} 
Then for any  $t\in [0,T]$ we have 
\begin{align}\label{DOG-eq:Ham-Bounds-DOG}
\int_{\R^2}H\rho_t\,dt<C',
\end{align}
where $C'>0$ depends on $C$ but is independent of $\vep$.  Furthermore, for any $t\in [0,T]$ we have  
\begin{align}\label{est:DOG-F-I}
\M{F}(\rho_t)+\frac1{2}\int_0^t \RelFI(\rho_s|\Lebesgue^2)\,ds\leq I^\e(\rho)+\M{F}(\rho_0).
\end{align}
\end{theorem}
See Appendix~\ref{DOG-App-Sec:HIR} for a proof of Theorem~\ref{DOG-thm:DOG-Compactness}.

Note that the estimate~\eqref{est:DOG-F-I} implies that $\mathcal{F}(\rho_t)=\RelEnt(\rho_t|\mathcal L^2)$ is finite for all $t$, and therefore $\rho_t$ is Lebesgue absolutely continuous. We will often therefore write $\rho_t(x)$ for the Lebesgue density of $\rho_t$. In addition, the integral of the relative Fisher Information is also bounded: $0\leq \int_0^t\RelFI(\rho_s|\Lebesgue^2)\,ds\leq C$.

The next result summarizes the compactness properties for any sequence $\rho^\vep$ with $\sup_\e I^\vep(\rho^\vep)<\infty$.
\begin{theorem}[Compactness]\label{DOG-thm:DOG-Compact}
Let a sequence $\rho^\vep\in C([0,T];\mathcal P(\R^2))$ with $\rho^\vep|_{t=0}=:\rho^\vep_0$ satisfy for a constant $C>0$ and all $\vep>0$ the estimate
\begin{align*}
I^\vep(\rho^\vep) + \M{F}(\rho^\vep_0)+\int_{\R^2}Hd\rho^\vep_0\leq C.
\end{align*}
Then there exist subsequences (not relabelled) such that 
\begin{enumerate}
\item{$\rho^\vep\rightarrow\rho$ in $\mathcal{M}([0,T]\times\mathbb{R}^{2})$  in the narrow topology;}\label{DOG-eq:DOG-Full-Compact}
\item{$\hrho^\vep\rightarrow \hrho= \xi_\# \rho$ in $C([0,T];\mathcal{P}(\G))$ with respect to the uniform topology in time and narrow topology on $\mathcal{P}(\G)$. }\label{DOG-eq:DOG-CG-Compact}
\end{enumerate}
Finally, we have the estimate
\[
\M{F}(\rho_t)+\frac1{2}\int_0^t \RelFI(\rho_s|\Lebesgue^2)\,ds\leq C \qquad
\text{for all }t\in[0,T].
\]
\end{theorem}
The sequence $\rho^\vep$ is tight in $\mathcal{M}([0,T]\times\mathbb{R}^{2})$ by estimate~\eqref{DOG-eq:Ham-Bounds-DOG}, which implies Part~\ref{DOG-eq:DOG-Full-Compact}. The proof of part~\ref{DOG-eq:DOG-CG-Compact} is similar to Part~\ref{thm:compactness-gamma:item:Rd} in Theorem~\ref{DOG-thm:VFP-Compactness}, and the final estimate is a direct consequence of \eqref{est:DOG-F-I}.

\subsection{Local equilibrium}

Theorem~\ref{DOG-thm:DOG-Compact} states that $\rho^\e$ converges narrowly on $[0,T]\times\R^2$ to some $\rho$. In fact we need a stronger statement, in which the behaviour of $\rho$ on each connected component of $H$ is fully determined by the limit $\hrho$.

Lemma~\ref{DOG-lem:DOG-Local-Eq-Const-Level-Sets} below makes this statement precise.    
Before proceeding we define $T:\Gamma\rightarrow\mathbb{R}$ as
\begin{align}\label{DOG-eq:DOG-Time-Period}
T(\gamma):=\int_{\xi^{-1}(\gamma)}\frac{\Hausdorff^1(dx)}{|\nabla H(x)|},
\end{align}
where $\Hausdorff^1$ is the the one-dimensional Hausdorff measure. $T$ has a natural interpretation as the period of the periodic orbit of the deterministic equation~\eqref{eq:Hamiltonian} corresponding to $\g$. When $\g$ is an interior vertex, such that the orbit is homoclinic, not periodic, $T(\g)=+\infty$.
$T$ also has a second natural interpretation: the measure $T(\gamma)d\gamma = T(h,k) dh$ on $\Gamma$ is the push-forward under $\xi$ of the Lebesgue measure on~$\R^2$, and the measure $T(\gamma)d\g$ therefore appears in various places.

\begin{lemma}[Local Equilibrium]\label{DOG-lem:DOG-Local-Eq-Const-Level-Sets}
Under the assumptions of Theorem~\ref{DOG-thm:DOG-Compact},  
let  $\rho^\vep\rightarrow\rho$ in $\M{M}([0,T]\times\mathbb{R}^2)$ with respect to the narrow topology. Let $\hrho$ be the push-forward $\xi_\#\rho$ of the limit $\rho$, as above. 

Then for a.e. $t$, the limit $\rho_t$ is absolutely continuous with respect to the Lebesgue measure, $\hrho_t$ is absolutely continuous with respect to the measure $T(\g)d\g$, where $T(\gamma)$ is defined in \eqref{DOG-eq:DOG-Time-Period}. Writing 
\[
\rho_t(dx) = \rho_t(x)dx  \qquad\text{and}\qquad 
\hrho_t(d\g) = \alpha_t(\g) T(\g)d\g,
\]
we have
\begin{align}
\label{eq:lem:DOG-locEq}
\rho_t(x)=\alpha_t(\xi(x)) 
\qquad \text{for almost all }x\in \R^{2} \text{ and } t\in[0,T].
\end{align}
\end{lemma}

\begin{proof}
From the boundedness of $I^\vep(\rho^\vep)$ and the narrow convergence $\rho^\vep\rightarrow\rho$ we find, passing to the limit in the rate functional \eqref{DOG-eq:DOG-Large-Dev-Rate-Fn}, for any $f\in C^{1,2}_c(\R\times \mathbb{R}^2)$
\begin{align}\label{DOG-eq:DOG-Limit-LDRF-Local-Eq}
\int_0^T\int_{\R^2}J\nabla H\cdot \nabla f \,d\rho_t dt=0.
\end{align}
Now choose any $\varphi\in C^2_c([0,T]\times\mathbb{R}^2)$ and any $\zeta\in C^2_b(\Gamma)$ such that $\zeta$ is constant in a neighbourhood of each vertex; then the function $f(t,x)=\zeta(\xi(x))\varphi(t,x)$ is well-defined and in $C^2_c([0,T]\times \mathbb{R}^2)$. 
We substitute this special function in \eqref{DOG-eq:DOG-Limit-LDRF-Local-Eq}; since $J\nabla H \nabla(\zeta\circ \xi) = 0$, we have $J\nabla H\nabla f = (\zeta\circ\xi) J\nabla H \nabla \varphi$. Applying the disintegration theorem to $\rho$, writing $\rho_t(dx) = \hrho_t(d\g)\trho_t(dx|\g)$
with $\supp\tilde{\rho}_t(\cdot|\g)\subset \xi^{-1}(\g)$, we obtain
\begin{align*}
0=\int_0^T\!\!\int_{\G}\zeta(\g)\hrho_t(d\g)\int_{\xi^{-1}(\g)}\nabla\varphi\cdot \frac{J\nabla H}{|\nabla H|}|\nabla H|\tilde{\rho}(\cdot|\g) d\Hausdorff^1=\int_0^T\int_{\G}\zeta(\g)\hrho_t(d\g)\int_{\xi^{-1}(\g)}\partial_\tau\varphi|\nabla H|\tilde{\rho}(\cdot|\g) d\Hausdorff^1dt,
\end{align*}
where $\partial_\tau$ is the tangential derivative.
By varying $\zeta$ and $\varphi$ we conclude that for $\hrho$-almost every $(\g,t)$, $|\nabla H|\tilde{\rho}_t(\cdot |\g) =C_{\g,t}$ for some $\g,t$-dependent constant $C_{\g,t}>0$, and since $\trho$ is normalized, we find that 
\begin{equation}
\label{DOG-prop:loc-eq-DOG}
\text{for }\hrho \text{-a.e. }  (\g,t): \ \trho_t(dx|\g)=\frac1{T(\g)|\nabla H(x)|}
{\Hausdorff^1\lfloor_{\xi^{-1}(\g)}(dx)}.
\end{equation} 
This also implies  that $\trho_t(\cdot|\g)$ is in fact $t$-independent.

For measurable $f$ we now compare the two relations
\begin{alignat*}2
&\int_{\mathbb{R}^2}fd\rho_t=\int_{\mathbb R^2} f(y) \rho_t(y) \, dy &&= \int_\G d\g\int_{\xi^{-1}(\g)}\frac{f(y)}{|\nabla H(y)|}\rho_t(y)\Hausdorff^1(dy)\\
&\int_{\mathbb{R}^2}fd\rho_t
=\int_\G\hrho_t(d\g)\int_{\xi^{-1}(\g)}f(y)\tilde{\rho}(dy|\g)
&&=\int_\G \frac{\hrho_t(d\g)}{T(\g)} \int_{\xi^{-1}(\g)}\frac{f(y)}{|\nabla H(y)|}\Hausdorff^1(dy)
\end{alignat*} 
where we have used the co-area formula in the first line and~\eqref{DOG-prop:loc-eq-DOG} in the second one. Since $f$ was arbitrary, \eqref{eq:lem:DOG-locEq} follows for almost all $t$.
\end{proof}

\subsection{Continuity of $\rho$ and $\hrho$}

As a consequence of the local-equilibrium property~\eqref{eq:lem:DOG-locEq} and the boundedness of the Fisher Information, we will show in the following that $\rho$ and its push-forward $\hrho$ satisfy an important continuity property. We first motivate this property heuristically.

The local-equilibrium result Lemma~\ref{DOG-lem:DOG-Local-Eq-Const-Level-Sets} states that the limit measure $\rho$ depends on $x$ only through $\xi(x)$. Take any measure $\rho\in\M P(\R^{2})$ of that form, i.e. $\rho(dx) = f(\xi(x))dx$, with finite free energy and finite relative Fisher Information. Setting $\tilde f = f\circ \xi$, by Lemma~\ref{DOG-lem:two-defs-RFI},  $\nabla _p \tilde f$ is well-defined and locally integrable.
\begin{figure}[h]
\centering
\includegraphics[scale=0.55]{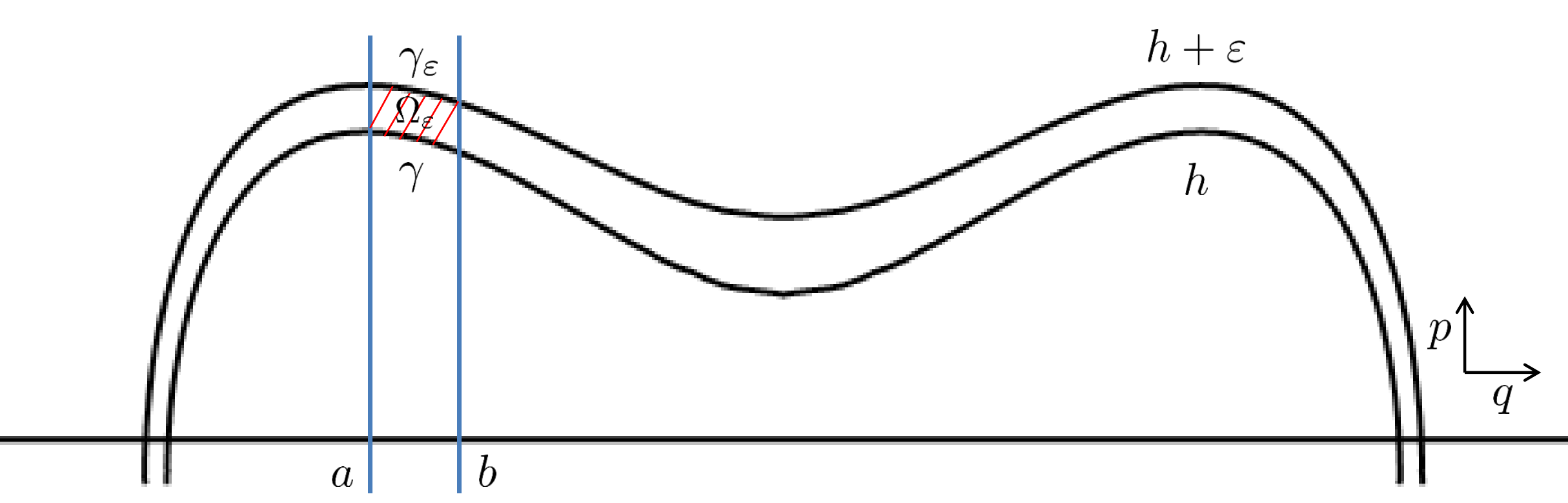}
\caption[Hamiltonian]{Section $\Omega$ in which $H^{-1}(h)$ is transverse to $p$. }
\label{fig:ProofPlot.pdf}
\end{figure}

Consider a section $\Omega_\e$ of the $(q,p)$-plane as shown in Figure~\ref{fig:ProofPlot.pdf}, bounded by $q=a$ and $q=b$ and level sets $H=h$ and $H=h+\e$. The top and bottom boundaries $\g$ and $\g_\e$ correspond to elements of $\Gamma$ that we also call $\g$ and~$\g_\e$; they might be part of the same edge $k$ of the graph, or they might belong to different edges. As $\e\to0$, $\g_\e$ converges to $\g$.

By simple integration we find that
\begin{align*}
\int_{\Omega_\e}\nabla_p\tilde f=\int_{\gamma_\e\cup\gamma}\tilde f n_p\,dr=(f(\g_\e)-f(\g))(b-a),
\end{align*}
where $dr$ is the scalar line element and $n_p$ the $p$-component of the normal $n$.
Applying H\"older's inequality we find
\begin{align*}
|b-a| \, |f(\g_\e)-f(\g)|&= \bigg| \int_{\Omega_\e} \nabla_p \rho\;\bigg|
 \leq \bigg{(}\int_{\Omega_\e}\frac{1}{\rho}\bigl|\nabla_p \rho \bigr|^2\bigg{)}^{\frac{1}{2}}\bigg{(}\int_{\Omega_\e}\rho\bigg{)}^{\frac{1}{2}}
 \xrightarrow{\vep\rightarrow 0}0.
\end{align*}
This argument shows that $f$ is continuous from the right at the point $\g\in \G$.

\medskip
The following lemma generalizes this argument to the case at hand, in which $\rho$ also depends on time. Note that $\Int \Gamma$ is the interior of the graph $\Gamma$, which is $\Gamma$ without the lower exterior vertices. 



\begin{lemma}[Continuity of $\rho$]
\label{lem:Continuity-Rho/T}
Let $\rho\in \M{P}([0,T]\times\R^2)$, $\rho(dtdx)=f(t,\xi(x))dtdx$ for a Borel measurable $f:[0,T]\times\Gamma\rightarrow\mathbb{R}$, and assume that 
\[
\int_0^T \RelFI(\rho_t|\Lebesgue^2)\, dt + \sup_{t\in[0,T]} \M F(\rho_t)<\infty.
\]
Then for almost all $t\in[0,T]$, $\g\mapsto f(t,\g)$ is continuous on $\Int \Gamma$.
\end{lemma}

\begin{proof}
The argument is essentially the same as the one above. For almost all $t$, $\rho_t$ is Lebesgue-absolutely-continuous and $\RelFI(\rho_t|\Lebesgue)$ is finite, and the argument above can be applied to the neighbourhood of any point~$x$ with $\nabla H(x)\not=0$, and to both right and left limits. The only elements of $\Gamma$ that have no representative $x\in \R^2$ with $\nabla H(x)\not=0$ are the lower ends of the graph, corresponding to the bottoms of the wells of $H$. At all other points of $\Gamma$ we obtain continuity.
\end{proof}

\begin{corollary}[Continuity of $\hrho$]
Let $\rho$ be the limit given by Theorem~\ref{DOG-thm:DOG-Compact}, and  $\hrho := \xi_\#\rho$ its push-forward.
For almost all $t$, $\hrho_t  \ll T(\g)d\g$, and $d\hrho_t/T(\g)d\g$ is continuous on $\Int \Gamma$.
\end{corollary}

\noindent
This corollary follows by combining Lemma~\ref{lem:Continuity-Rho/T} with Lemma~\ref{DOG-lem:DOG-Local-Eq-Const-Level-Sets}.

\subsection{Liminf inequality}

We now derive the final ingredient of the proof, the liminf inequality. 
Define 
\begin{equation}
\label{def:DOG-hatI}
\hat I(\hrho):=
\begin{cases}
\sup\limits_{g\in C^{1,2}_c(\R\times \Gamma)} \hat {\mathcal J}(\hrho,g)\quad
& \text{if }\hrho_t\ll T(\g)d\g, \ \hrho_t(d\g) = f_t(\g)T(\g)d\g\text{ with } f\text{ continuous on }
\Int \Gamma,\\[-3\jot]
&\qquad\qquad \text{for almost all $t\in[0,T]$},\\
+\infty &\text{otherwise},
\end{cases}
\end{equation}
where
\begin{multline}
\hat{\mathcal J}(\hrho,g) := 
\int_{\Gamma}g_Td\hrho_T
-\int_{\Gamma}g_0d\hrho_0
-\int_0^T\int_{\Gamma}
\big{(}\partial_t g_t(\g)+A(\gamma)g_t''(\gamma)+B(\gamma)g_t'(\gamma)
\big{)}\hrho_t(d\gamma)dt\\
-\frac{1}{2}\int_0^T\int
_{\Gamma}A(\gamma)(g_t'(\gamma))^2\hrho_t(d\gamma)dt,
\end{multline}
and we use $g'$ and $g''$ to indicate derivatives with respect to $h$.
For $\gamma\in \Gamma$, the coefficients are defined by
\begin{align}
\label{def:ABT}
A(\gamma):=\frac{1}{T(\gamma)}
\int_{\xi^{-1}(\gamma)}\frac{(\nabla_p H)^2}
{|\nabla H|}d\Hausdorff^{1}, 
\quad 
B(\gamma):=\frac{1}
{T(\gamma)}\int_{\xi^{-1}(\gamma)}
\frac{\Delta_p H}{|\nabla H|}d\Hausdorff^1,
\quad 
T(\gamma):=\int_{\xi^{-1}(\gamma)}\frac{1}{|\nabla H|} d\Hausdorff^1. 
\end{align}
Note that for our particular choice of $H(q,p)=p^2/2m+V(q)$, we have $B(\gamma)=1/m$.

The class of test functions in~\eqref{def:DOG-hatI} is $C^{1,2}_c(\R\times \Gamma)$; recall that differentiability of a function $f:\Gamma\to\R$ is defined by restriction to one-dimensional subgraphs, and $C^{1,2}_c(\R\times \Gamma)$ therefore consists of functions $g:\Gamma\to\R$ that are twice continuously differentiable in $h$ in this sense. The subscript $c$ indicates that we restrict to functions that vanish for sufficiently large $h$ (i.e. somewhere along the top edge of $\Gamma$).

Note that again $\hat I\geq0$; formally, $\hat I(\hrho)=0$ iff $\hrho$ satisfies the diffusion equation 
\begin{align*}
\partial_t\hrho=(A\hrho)''- (B\hrho)',
\end{align*}
and we will investigate this equation in more detail in the next section.

\medskip
\begin{theorem}[Liminf inequality]\label{DOG-eq:DOG-Liminf-Theorem}
Under the same assumptions as in Theorem~\ref{DOG-thm:DOG-Compact}, let  $\rho^\vep\rightarrow\rho$ in $\mathcal M([0,T];\mathbb{R}^2)$ and $\hrho^\vep:= \xi_\#\rho^\vep\rightarrow\xi_\#\rho=:\hrho$ in $C([0,T];\M{P}(\G))$. Then \begin{align*}
 \liminf\limits_{\vep\rightarrow 0}I^\vep(\rho^\vep)\geq \hat I(\hrho).
\end{align*}
\end{theorem}
\begin{proof}
Recall the rate functional from \eqref{DOG-eq:DOG-Large-Dev-Rate-Fn}
\begin{align}
I^\vep(\rho^\vep)=\sup\limits_{f\in C_c^{1,2}(\mathbb{R}\times\mathbb{R}^{2})}\mathcal{J}^\vep(\rho^\vep,f),&  \ \ \text{ where}
\label{eq:thm-lim-inf-rate-fn}
\\
\mathcal{J}^\vep(\rho^\vep,f):= 
\int_{\mathbb{R}^{2}}f_Td\rho_T^\vep
-\int_{\mathbb{R}^{2}}f_0d\rho_0^\vep
-\int_0^T\int_{\mathbb{R}^{2}}
\Bigl(\partial_t f+\frac{1}{\vep}J\nabla H\cdot\nabla f+&\Delta_pf\Bigr)d\rho_t^\vep dt
-\frac{1}{2}\int_0^T\int
_{\mathbb{R}
^{2}}\left|\nabla_p f\right|^2d\rho_t^\vep dt.\nonumber
\end{align}
Define $\hat{\mathcal{A}}:=\left\{ f=g\circ \xi: \ g\in C_c^{1,2}(\mathbb{R}\times \Gamma)\right\}$. Then we have 
\begin{align*}
I^\vep(\rho^\vep)\geq\sup\limits_{f\in \hat{\mathcal{A}}}\mathcal{J}^\vep(\rho^\vep,f).
\end{align*}
Since $J\nabla H\nabla (g\circ\xi)= 0$, upon substituing $f=g\circ \xi$ into $\mathcal \LDJ^\vep$ the $O(1/\vep)$ term  vanishes. 
Using the notation $g'$ for the partial derivative with respect to $h$, $\partial_t g$ for the time derivative, and suppressing the dependence of $g$ on time, we find
\begin{multline}\label{DOG-eq:DOG-Liminf-Rate-Fn-Substitute}
\mathcal{J}^\vep(\rho^\vep,g\circ \xi):=
\int_{\G}g_Td\hrho_T^\vep
-\int_{\G}g_0d\hrho_0^\vep
-\int_0^T\int_{\mathbb{R}^{2}}\bigg{(}\partial_t g(\xi(x))+
g''(\xi(x))(\nabla_p H(x))^2+g'(\xi(x))\Delta_p H(x)\bigg{)}\rho_t^\vep(dx) dt\\
-\frac{1}{2}\int_0^T\int
_{\mathbb{R}
^{2}}|g'(\xi(x))\nabla_p H(x)|^2\rho_t^\vep(dx) dt.
\end{multline}
The limit of \eqref{DOG-eq:DOG-Liminf-Rate-Fn-Substitute} is determined term by term. Taking the fourth term as an example, using the co-area formula and the local-equilibrium result of Lemma~\ref{DOG-lem:DOG-Local-Eq-Const-Level-Sets}, 
the fourth term on the right-hand side of \eqref{DOG-eq:DOG-Liminf-Rate-Fn-Substitute} gives
\begin{multline*}
\int_0^T\!\!\int_{\mathbb{R}^2}g''(\xi(x))(\nabla_p H(x))^2\rho^\vep_t(dx)dt\xrightarrow{\vep\rightarrow 0}
\int_0^T\!\!\int_{\mathbb{R}^2}g''(\xi(x))(\nabla_p H(x))^2\rho_t(dx)dt\\
=\int_0^T dt\int_{\Gamma}\frac{g''(\g)\hrho_t(d\gamma)}{T(\gamma)}\bigg{(}\int_{\xi^{-1}(\gamma)}\frac{(\nabla_p H(y))^2}{|\nabla H(y)|}
\Hausdorff^1(dy)\bigg{)}=\int_0^T \!\!\int_{\Gamma}A(\gamma)g''(\gamma)\hrho_t(d\gamma)dt,
\end{multline*}
where $A:\Gamma\rightarrow \mathbb{R}$ is defined in~\eqref{def:ABT}.
Proceeding similarly with the other terms we find
\begin{align}\label{DOG-eq:DOG-Liminf-First-Admissible-Class}
\liminf\limits_{\vep\rightarrow 0} I^\vep(\rho^\vep)\geq \sup\limits_{g\in C^{1,2}_c(\mathbb{R}\times\Gamma)} \mathcal{\hat J}(\hrho,g).
\end{align}
This concludes the proof of Theorem~\ref{DOG-eq:DOG-Liminf-Theorem}.
\end{proof}

\subsection{Study of the limit problem}

We now investigate the limiting functional $\hat I$ from \eqref{def:DOG-hatI} a little further. The two main results of this section are that $\hat {\mathcal J}$ can be written as 
\begin{equation}
\label{DOG-expr:J-alt}
\hat {\mathcal J}(\hrho,g) = \int_{\Gamma}g_Td\hrho_T
-\int_{\Gamma}g_0d\hrho_0
-\int_0^T\int_{\Gamma}
\Bigl[\partial_t g_t\,d\hrho_t+\Bigl((\TA \,g_t')' + \frac12 \TA\, {g_t'}^2\Bigr)\frac{d\hrho_t}{T} \Bigr]dt,
\end{equation}
and that $\hat I$ satisfies
\begin{equation}
\label{DOG-ineq:IhatItilde}
\hat I(\hrho) \geq \sup_{g\in \mathcal A} \hat{\mathcal J}(\hrho,g)
\qquad\text{for all }\hrho\in C([0,T];\M P(\Gamma)),
\end{equation}
where $\M A$ is the larger class
\begin{align}\label{DOG-eq:DOG-Liminf-Final-Bounded-Test-Class}
\M{A}:=\bigg{\{}g:C^{1,0}(\R\times\Gamma): g\big{|}_{I_k}\in C^{1,2}_b(\R\times I_k), \quad \forall \text{ interior vertex }O_j \ \forall t:
\sum_{k:I_k\sim O_j}\pm_{kj} \,g_t'(O_j,k)
\,{\TA}(O_j,k)=0 
\bigg{\}}.
\end{align} 
The admissible set $\M A$ relaxes the conditions on $g$ at interior vertices: instead of requiring $g$ to have identical derivatives coming from each edge, only a single scalar combination of the derivatives has to vanish. 
(In fact it can be shown that equality holds in~\eqref{DOG-ineq:IhatItilde}, but that requires a further study of the limiting equation that takes us too far here.)

Both results use some special properties of $T$, $A$, and $B$, which are given by the following lemma. In this lemma and below we use $\TA$ and $TB$ for the functions obtained by multiplying $T$ with $A$ and $B$; these combinations play a special role, and we treat them as separate functions. 
\begin{lemma}[Properties of $\TA$ and $TB$] 
The functions $\TA$ and $TB$ have the following properties.
\label{DOG-lem:propsTATB}
\begin{enumerate}
\item \label{DOG-lem:propsTATB-part1}
$\TA \in C^1(I_k)$ for each $k$, and $(\TA)' = TB$;
\item \label{DOG-lem:propsTATB-part3}
$\TA$ is bounded on compact subsets of $\Gamma$;
\item \label{DOG-lem:propsTATB-part4}
At each interior vertex $O_j$, for each $k$ such that $I_k\sim O_j$, $\TA(O_j,k) := \lim\limits_{\substack{h\in I_k\\h\to O_j}} \TA(h,k)$ exists, and 
\begin{equation}
\label{DOG-prop:summation-TA}
\sum_{k:I_k\sim O_j} \pm_{kj} \,\TA(O_j,k) = 0.
\end{equation}
\end{enumerate}
\end{lemma}

From this lemma the expression~\eqref{DOG-expr:J-alt} follows by simple manipulation.

\medskip

With these two results, we can obtain a {differential-equation characterization} of those $\hrho$ with $\hat I(\hrho)=0$. 
Assume that a $\hrho$ with $\hat I(\hrho)=0$ is given. By rescaling we find that for all $g\in\M A$,
\begin{equation}
\label{DOG-eq:weak-DOG}
\int_{\Gamma}g_Td\hrho_T
-\int_{\Gamma}g_0d\hrho_0
=\int_0^T\int_{\Gamma}
\Bigl[\partial_t g_t\,d\hrho+(\TA \,g_t')' \frac{d\hrho_t}{T} \Bigr]dt.
\end{equation}
As already remarked we find a parabolic equation inside each edge of $\Gamma$,
\begin{equation}
\label{DOG-eq:PDE-on-graph}
\partial_t \hrho_t = \Bigl(\TA\, \bigl(\frac {\hrho_t}T\bigr)'\Bigr)'
= (A\hrho_t)'' - (B\hrho_t)'.
\end{equation}
We next determine the boundary and connection conditions at the vertices.

Consider a  single interior vertex $O_j$, and choose a function $g\in \M A$ such that $\supp g$ contains no other vertices. Writing $\hrho_t(d\g) = f_t(\g)T(\g)d\g$ we find first that $f_t$ is continuous at $O_j$, by the definition~\eqref{def:DOG-hatI}
of~$\hat I$. Then, assuming that $\hrho$ is smooth enough for the following expressions to make sense\footnote{This can actually be proved using the properties of $A$ and $B$ near the vertices and applying standard parabolic regularity theory on each of the edges.}, we perform two partial integrations in $\g$ and one in time on~\eqref{DOG-eq:weak-DOG} and substitute~\eqref{DOG-eq:PDE-on-graph} to find
\[
0 = \int_0^T f_t(O_j) \sum_{k: I_k\sim O_j} \pm_{kj} \,\TA(O_j,k) g'_t(O_j,k)\, dt
- \int_0^T g_t(O_j)\sum_{k: I_k\sim O_j} \pm_{kj} \,\TA(O_j,k) f'_t(O_j,k)\, dt.
\]
The first term vanishes since $g\in\M A$, while the second term leads to the connection condition
\[
\text{at each interior vertex }O_j: \quad \sum_{k: I_k\sim O_j} \pm_{kj} \,\TA(O_j,k) f'_t(O_j,k) =0.
\]

The lower exterior vertices and the top vertex are \emph{inaccessible}, in the language of~\cite{Feller1952,Mandl1968}, and therefore require no boundary condition. Summarizing, we find that if $\hat I(\hrho)=0$, then $\hrho =: fTd\g$ satisfies a weak version of equation~\eqref{DOG-eq:PDE-on-graph} with connection conditions
\[
\text{at each interior vertex }O_j: \quad \text{$f$ \ is continuous and } \quad \sum_{k: I_k\sim O_j} \pm_{kj} \,\TA(O_j,k) f'_t(O_j,k) =0.
\]
This combination of equation and boundary conditions can be proved to characterize a well-defined semigroup  using e.g.\ the Hille-Yosida theorem and the characterization of one-dimensional diffusion processes by Feller (e.g.~\cite{Feller1952}). 
\bigskip

We now prove the inequality~\eqref{DOG-ineq:IhatItilde}.
\begin{lemma}[Comparison of $\hat I$ and $\tilde I$]
We have 
\[
\hat I(\hrho) \geq \tilde I(\hrho) := \sup_{g\in\M A} \hat{\M J}(\hrho,g).
\]
\end{lemma}

\begin{proof}
Take $\hrho$ such that $\hat I(\hrho)<\infty$, implying that $\hrho_t(d\g) = f_t(\g)T(\g)d\g$ with $f_t$ continuous on $\Int \Gamma$ for almost all $t$.
Choose $g\in \M A$; we will show that $\hat I(\hrho)\geq \hat {\M J}(\hrho,g)$, thus proving the lemma. For simplicity we only treat the case of a single interior vertex, called $O$; the case of multiple vertices is a simple generalization. For convenience we also assume that $O$ corresponds to $h=0$.

Define
\begin{align}
g_{\d,t}(h,k)=g_t(h,k)\zeta_\d(h)+(1-\zeta_\d(h))g_t(0),  
\end{align}
where $\zeta_\d$ is a sequence of smooth functions such that 
\begin{itemize}
\item $\zeta_\d$ is identically zero in a $\d$-neighbourhood of $O$, and identically $1$ away from a $2\d$-neighbourhood of $O$; 
\item $\zeta_\d$ satisfies the growth conditions $|\zeta'_\d|\leq 2/\d$ and $|\zeta''_\d|\leq 4/\d^2$.
\end{itemize}

We calculate $\hat {\M J}(\hrho,g_\d)$.
The limit of the first three terms is straightforward: by dominated convergence we obtain
\[
\int_{\Gamma}g_{\d,T}d\hrho_T
-\int_{\Gamma}g_{\d,0}d\hrho_0
-\int_0^T\!\!\int_{\Gamma}
\partial_t g_{\d,t}\,d\hrho_t
\xrightarrow{{\d\to0}} \int_{\Gamma}g_{T}d\hrho_T
-\int_{\Gamma}g_{0}d\hrho_0
-\int_0^T\!\!\int_{\Gamma}
\partial_t g_{t}\,d\hrho_t.
\]
Next consider the term
\begin{align}\label{DOG-eq:DOG-Liming-C0-A-Term}
\int_0^T\!\!\int_\Gamma A(\gamma)g''_\d(\gamma)\hrho_t(d\gamma)dt=
\int_0^T\!\!\int_{\Gamma}\bigg{[}
g''(h,k)\zeta_\d(h)+2\zeta_\d'(h)g'(h,k)
+\zeta''_\d(h)\bigl[hg'(0,k)+O(h^2)\bigr]
\bigg{]}A(\g)\hrho_t(d\gamma) dt.
\end{align}

Since the function $(\g,t)\mapsto A(\g)g_t''(\g)\in L^\infty(\hat\rho_t)$ the first term in \eqref{DOG-eq:DOG-Liming-C0-A-Term} again converges by  dominated convergence :
\begin{align*}
\int_0^T\!\!\int_\Gamma g_t''(h,k)\zeta_\d(h)A(h,k)\hrho_t(d\gamma)dt
\xrightarrow{\d\rightarrow 0} \int_0^T\!\!\int_\Gamma g_t''(h,k)A(h,k)\hrho_t(d\gamma)dt.
\end{align*}
Abbreviate $f_t(\g)\TA(\g)$ as $a(\g)$; note that $a$ is continuous and bounded in a neighbourhood of $O$. 
Write the second term on the right-hand side in~\eqref{DOG-eq:DOG-Liming-C0-A-Term} as (supressing the time integral for the moment)
\begin{align*}
2\int_\Gamma\zeta_\d'(h)g'(h,k)a(h,k)dh
&=2\int_\Gamma \zeta_\d'(h)g'(h,k)\bigl(a(h,k)-a(0,k)\bigr)d\gamma 
+ 2\sum_k a(0,k)\int_{I_k}\zeta_\d'(h)\bigl(g'(h,k)-g'(0,k)\bigr)dh \\
&+2\sum_ka(0,k)g'(0,k)\int_{\Gamma_k}\zeta_\d'(h)dh\\
&\xrightarrow{\d\rightarrow 0}0+0-2\sum_{k:I_k\sim O} \pm_{kO}\,g'(0,k)\,a(0,k)
= 2\sum_{k:I_k\sim O} \pm_{kO}\,g'(0,k)\,f(0,k)\,\TA(0,k).
\end{align*}
The limit above holds since $-\zeta_\d'(\cdot,k)$ converges weakly to a signed Dirac, $\pm_{kO} \delta_0$, as $\d\to0$.
Proceeding similarly with the remaining terms we have 
\begin{align*}
\hat I(\hrho)\geq\hat\LDJ(\hrho,g_\d)\xrightarrow{\d\rightarrow 0} &\int_{\Gamma}g_Td\hrho_T
-\int_{\Gamma}g_0d\hrho_0
-\int_0^T\!\!\int_{\Gamma}
\big{(}\partial_t g_t+A(\gamma)g_t''(\gamma)+B(\gamma)g_t'(\gamma)
\big{)}\hrho_t(d\gamma)dt\\
&-\frac{1}{2}\int_0^T\!\!\int
_{\Gamma}A(\gamma){g_t'(\gamma)}^2\hrho_t(d\gamma)dt-\int_0^T f_t(0,k)\biggl[\sum_{k:I_k\sim O}\pm_{kO}\TA(0,k)g_t'(0,k)\biggr] dt.
\end{align*}
Note that the final term vanishes by the requirement that $g\in \M A$, and therefore the right-hand side above equals $\hat \LDJ(\hrho,g)$.
This concludes the proof of the lemma. 
\end{proof}

We still owe the reader the proof of Lemma~\ref{DOG-lem:propsTATB}.

\begin{proof}[Proof of Lemma~\ref{DOG-lem:propsTATB}]
We first prove part~\ref{DOG-lem:propsTATB-part1}. For simplicity, assume first that $H$ has a single well, and therefore $\Gamma$ has only one edge, $k=1$. Since
\[
\div \begin{pmatrix} 0\\\nabla_p H\end{pmatrix} = \Delta_p H,
\]
and remarking that the exterior normal $n$ to the set $H\leq h$ equals $(0,\nabla_p H/|\nabla H|)^T$, we calculate that
\begin{equation}
\label{DOG-eq:alt-def-TA}
\int\limits_{\{H\leq h\}} \Delta_p H  = \int\limits_{\{H =h\}} \frac {(\nabla_p H)^2}{|\nabla H|}\, d\Hausdorff^1 = \TA(h).
\end{equation}
By the smoothness of $H$, the derivative of the left-hand integral is well-defined for all $h$ such that $\nabla H\not=0$ at that level. At such $h$ we then have
\[
TB(h) = \int\limits_{\{H=h\}} \frac{\Delta_pH}{|\nabla H|} \,d\Hausdorff^1 = 
 \partial_h \int\limits_{\{H\leq h\}} \Delta_p H = \partial_h \TA(h).
\]
For the multi-well case, this argument can simply be applied to each branch of $\Gamma$. 

For part~\ref{DOG-lem:propsTATB-part3}, since $H$ is coercive, $\{H\leq h\}$ is bounded for each $h$; since $H$ is smooth, therefore $\Delta_pH$ is bounded on bounded sets. From~\eqref{DOG-eq:alt-def-TA} it follows that $\TA$ also is bounded on bounded sets of $\Gamma$. 

Finally, for part~\ref{DOG-lem:propsTATB-part4}, note first that $TB$ is bounded near each interior vertex. This follows by an explicit calculation and our assumption that each interior vertex corresponds to exactly one, non-degenerate, saddle point. Since $(\TA)'=TB$, $\TA$ has a well-defined and finite limit at each interior saddle. The summation property~\eqref{DOG-prop:summation-TA} follows from comparing~\eqref{DOG-eq:alt-def-TA} for values of $h$ just above and below the critical value. For instance, in the case of a single saddle at value $h=0$, with two lower edges $k=1,2$ and upper edge $k=0$, we have
\begin{align*}
\lim_{h\uparrow 0} \ \TA(h,1) + \TA(h,2) \;&=\ \lim_{h\uparrow 0} \int\limits_{\xi^{-1}\bigl((-\infty,h]\times \{1\}\bigr)} \Delta_p H
\ + \int\limits_{\xi^{-1}\bigl((-\infty,h]\times \{2\}\bigr)} \Delta_p H\\
&= \ \lim_{h\uparrow 0} \int\limits_{\{H\leq h\}} \Delta_p H  \ 
=\  \lim_{h\downarrow 0} \int\limits_{\{H\leq h\}} \Delta_p H  = \lim_{h\downarrow0}\; \TA(h,0).
\end{align*}
This concludes the proof of Lemma~\ref{DOG-lem:propsTATB}.
\end{proof}

\subsection{Conclusion and discussion}

The combination of Theorems~\ref{DOG-thm:DOG-Compact} and~\ref{DOG-eq:DOG-Liminf-Theorem} give us that along subsequences $\hrho^\e := \xi_\# \rho^\e$ converges in an appropriate manner to some $\hrho$, and that
\[
\hat I(\hrho) \leq \liminf_{\e\to0} I^\e(\rho^\vep).
\]
In addition,  any $\hrho$ satisfying $I(\hrho)=0$ is a weak solution of the PDE 
\[
\partial_t \hrho = (A\hrho)''- (B\hrho)'
\]
on the graph $\Gamma$. This is the central coarse-graining statement of this section. We also obtain the boundary conditions, similarly as in the conventional weak-formulation method, by expanding the admissible set of test functions.

\medskip

In switching from the VFP equation~\eqref{DOG-eq:Intro-VFPa} to equation~\eqref{DOG-eq:Ran-Ham-Evo-R2} we removed two terms, representing the friction with the environment and the interaction between particles. Mathematically, it is straightforward to treat the case with friction, which leads to an additional drift term in the limit equation in the direction of decreasing~$h$. We left this out simply for the convenience of shorter expressions.

As for the interaction, represented by the interaction potential $\psi$, again there is no mathematical necessity for setting $\psi=0$ in this section; the analysis continues rather similarly. However, the limiting equation will now be non-local, since the particles at some $\gamma\in\Gamma$, which can be thought of as `living' on a full connected level set of $H$, will feel a force exerted by particles at a different $\gamma'\in \Gamma$, i.e. at a different level set component. This makes the interpretation of the limiting equation somewhat convoluted. 

\medskip

The results of the current and the next sections were proved by Freidlin and co-authors in a series of papers~\cite{FreidlinWentzell93,Freidlin1994,Freidlin1998,Freidlin2001,
Freidlin2004},  using probabilistic techniques. Recently, Barret and Von Renesse \cite{Barret2014} provided an alternative proof using Dirichlet forms and their convergence. The latter approach is closer to ours in the sense that it is mainly PDE-based method and of variational type. However, in \cite{Barret2014} the authors consider a perturbation of the Hamiltonian by a friction term and a non-degenerate noise, i.e. the noise is present in both space and momentum variables; this non-degeneracy appears to be essential in their method. Moreover, their approach invokes a reference measure which is required to satisfy certain non-trivial conditions. In contrast, the approach of this paper is applicable to degenerate noise and does not require such a reference measure. In addition, certain non-linear evolutions can be treated, such as the example of the VFP equation.

\section{Diffusion on a graph, $d>1$}\label{DOG-sec:DOG-d>1}

We now switch to our final example. As described in the introduction, the higher-dimensional analogue of the diffusion-on-graph system has an additional twist: in order to obtain unique stationary measures on level sets of $\xi$, we need to add an additional noise in the SDE, or equivalently, an additional diffusion term in the PDE. This leads to the equation
\begin{align}
\label{eq:diffusionmultidim}
\partial_t\rho=-\frac{1}{\vep}\div(\rho J\nabla H)+\frac{\kappa}{\vep}\div(a\nabla \rho)+\Delta_p\rho,
\end{align}
where $a:\Rd\rightarrow \R^{2d\times 2d}$ with $a\nabla H=0$, $\text{dim}(\text{Ker}(a))=1$ and $\kappa,\vep>0$ with $\kappa\gg\vep$. The spatial domain is $\mathbb{R}^{2d}$, $d>1$,  with coordinates $(q,p)\in\mathbb{R}^d\times \mathbb{R}^d$. Here the unknown is trajectory in the space of probability measures $\rho:[0,T]\rightarrow\mathcal{P}(\mathbb{R}^{2d})$; the Hamiltonian is the same as in the previous section, $H:\Rd\rightarrow\mathbb{R}$ given by $H(q,p)=p^2/2m+V(q)$.

The results for the limit $\vep\rightarrow 0$ in \eqref{eq:diffusionmultidim} closely mirror the one-degree-of-freedom diffusion-on-graph problem of the previous section; the only real difference lies in the proof of local equilibrium (Lemma~\ref{DOG-lem:DOG-Local-Eq-Const-Level-Sets}). For a rigorous proof of this lemma in this case, based on probabilistic techniques, we refer to~\cite[Lemma 3.2]{Freidlin2001}; here we only outline a possible analytic proof. 

Along the lines of Theorem~\ref{DOG-thm:DOG-Compactness}, and using boundedness of the rate functional $I^\e(\rho^\vep)$, one can show that
\begin{align*}
\frac{1}{2}\int_0^T\int_{\mathbb{R}^2}\frac{|\nabla_p\rho^\vep|^2}{\rho^\vep}+\frac{\k}{\vep}\int_0^T\int_{\mathbb{R}^2}\frac{a\nabla\rho^\vep\cdot\nabla\rho^\vep}{\rho^\vep}\leq C.
\end{align*}
Multiplying this inequality by $\vep/\k$ and using the weak convergence $\rho^\vep\rightharpoonup\rho$  along with the lower-semicontinuity of the Fisher information \cite[Theorem D.45]{FengKurtz06} we find
\begin{align*}
\int_0^T\int_{\mathbb{R}^2}\frac{a\nabla\rho\cdot\nabla\rho}{\rho}=0,
\end{align*}
or in variational form, for almost all $t\in[0,T]$,
\begin{align*}
&0=\sup\limits_{\varphi\in C_c^\infty(\Rd)}\int_{\Rd}\mathrm{div}(a\nabla\varphi)\rho_t-\frac{1}{2}\int_{\Rd}a\nabla\varphi\cdot\nabla\varphi\rho_t \\
&\qquad \Longleftrightarrow \ \ 0= \int_{\Rd}\mathrm{div}(a\nabla\varphi)\rho_t,  \qquad \forall\varphi\in C_c^\infty(\Rd).
\end{align*} 
Applying the co-area formula we find 
\begin{align}\label{DOG-eq:DOG-d>1-Loc-Eq-Co-area}
\int_{\xi^{-1}(\g)}\frac{\rho(x)}{|\nabla H(x)|}\mathrm{div}(a(x)\nabla\varphi(x))\,\Hausdorff^{2d-1}(dx)=0,
\end{align}
where $\Hausdorff^{2d-1}$ is the $(2d-1)$ dimensional Haursdoff measure. Let $\mathcal{M}_\g$ be the $(2d-1)$ dimensional manifold $\xi^{-1}(\g)$ with volume element $|\nabla H|^{-1}\Hausdorff^{2d-1}$. Then \eqref{DOG-eq:DOG-d>1-Loc-Eq-Co-area}  becomes
\begin{align*}
\int_{\mathcal{M}_\g}\rho(x)\,\mathrm{div}_{\mathcal{M}}(a(x)\nabla_{\mathcal{M}}\varphi(x))\, \mathrm{vol}_{\mathcal{M}}(dx)=0,
\end{align*}
where  $\mathrm{div}_\mathcal{M}$ and $\nabla_{\mathcal{M}}$ are the corresponding differential operators on $\mathcal{M}_\g$, and $\mathrm{vol}_{\mathcal{M}}$ is the induced volume measure. Since  $a\nabla H=0$, $\text{dim}(\text{Ker}(a))=1$, $a$ is non-degenerate on the tangent space of $\mathcal{M}_\g$. Therefore, given $\psi\in C^\infty(\mathcal{M}_\g)$ with  $\int_{\mathcal{M}_\g}\psi\,d\,\mathrm{vol_{\mathcal{M}}}=0$,
 we can solve the corresponding Laplace-Beltrami-Poisson equation for $\varphi$,
\begin{align*}
\mathrm{div}_\mathcal{M}(a\nabla_{\mathcal{M}}\varphi)=\psi,
\end{align*}
and therefore 
\begin{align*}
\int_{\mathcal{M}_\g}\rho\,\psi\,\mathrm{dvol_{\mathcal{M}}}=0, \  \forall \psi\in C^\infty(\mathcal{M}_\g) \text{ with } \int_{\mathcal{M}_\g}\psi\,d\,\mathrm{vol_{\mathcal{M}}}=0.
\end{align*} 
Since $\mathcal M_\g$ is connected by definition, it follows that $\rho$ constant on $\mathcal{M}_\g$; this is the statement of Lemma~\ref{DOG-lem:DOG-Local-Eq-Const-Level-Sets}.

\section{Conclusion and discussion}
\label{sec: discussion}

In this paper we have presented a structure in which coarse-graining and `passing to a limit' combine in a natural way, and which extends also naturally to a class of approximate solutions. The central object is the rate function $I$, which is minimal and vanishes at solutions; in the dual formulation of this rate function, coarse-graining has a natural interpretation, and the inequalities of the dual formulation and of the coarse-graining combine in a convenient way.

We now comment on a number of issues related with this method.
\medskip

\emph{Why does this method work?}
One can wonder why the different pieces of the arguments of this paper fit together. Why do the relative entropy and the relative Fisher information appear? To some extent this can be recognized in the similarity between the duality definition of the rate function $I$ and the duality characterization of relative entropy and relative Fisher Information. The details of Appendix~\ref{DOG-sec:App-PDE} show this most clearly, but the similarity between the duality definition of the relative Fisher information and the duality structure of $I$ can readily be recognized: in~\eqref{def:I-gamma} combined with~\eqref{def:L-VFP-rescaled} we collect the $O(\g^2)$  terms
\[
\int_0^T \int_{\Rd} \bigg{[}\Delta_p f_t -  \frac pm \nabla_p f_t - \frac12 \left|\nabla_p f_t\right|^2\bigg{]}d\rho_tdt,
\]
and these match one-to-one to the definition~\eqref{DOG-eq:Relative-Fisher-Information}. This shows how the structure of the relative Fisher Information is to some extent `built-in' in this system.

\medskip

\emph{Relation with  other variational formulations}. Our variational formulation \eqref{eq: variational formulation} to `passing to a limit' is closely related to other variational formulations in the literature, notably the $\Psi$-$\Psi^*$ formulation and the method in \cite{PennacchioSavareColli05,AmbrosioSavareZambotti09}. In the $\Psi$-$\Psi^*$ formulation, a gradient flow of the energy $\mathcal{E_\varepsilon}:\mathcal{Z}\rightarrow \mathbb{R}$ with respect to the dissipation $\Psi_\varepsilon^*$ is defined to be a curve $\rho^\varepsilon\in C([0,T],\mathcal{Z})$ such that 
\begin{equation}
\label{psi-psi^* formulation}
\mathcal{A}^\varepsilon(\rho):=\mathcal{E}_\varepsilon(\rho_T)-\mathcal{E}_\varepsilon(\rho_0)+\int_0^T [\Psi_\varepsilon(\dot{\rho_t},\rho_t)+\Psi_\varepsilon^*(-\mathsf{D}\mathcal{E}_\varepsilon(\rho_t),\rho_t)]\,dt=0.
\end{equation}
`Passing to a limit' in a $\Psi$-$\Psi^*$ structure is then accomplished by studying (Gamma-) limits of the functionals~$\mathcal{A}^\varepsilon$.  The method introduced in \cite{PennacchioSavareColli05,AmbrosioSavareZambotti09} is slightly different. Therein `passing to a limit' in 
the evolution equation is executed by studying (Gamma-)limits of the functionals that appear in the approximating discrete minimizing-movement schemes. 

The similarities between these two  approaches and ours is that all the methods hinge on duality structure of the relevant functionals, allow one to obtain both compactness and limiting results, and can work with approximate solutions, see e.g.  \cite{ArnrichMielkePeletierSavareVeneroni12} and the papers above for details. In addition, all methods assume some sort of well-prepared initial data, such as bounded initial free energy and boundedness of the functionals. Our assumptions on the boundedness of the rate functionals arise naturally in the context of large-deviation principle since this assumption describes events of a certain degree of `improbability'.

The main difference is that the method of this paper makes no use of the gradient-flow structure, and therefore also applies to non-gradient-flow systems as in this paper. The first example, of the overdamped limit of the VFP equation, also is interesting in the sense that it derives a dissipative system from a non-dissipative one. Since the GENERIC framework unifies both dissipative and non-dissipative systems, we expect that the method of this paper could be used to derive evolutionary convergence for GENERIC systems (see the next point). Finally, we emphasize that using the duality of the rate functional is mathematically convenient because we do not need to treat the three terms in the right-hand side of \eqref{psi-psi^* formulation} separately. Note that although the entropy and energy functionals as well as the dissipation mechanism are not explictly present in this formulation, we are still able to derive an energy-dissipation inequality in~\eqref{eq: bound of energy and fisher information}.

\medskip

\emph{Relation with GENERIC.} As mentioned in the introduction, the Vlasov-Fokker-Planck system \eqref{DOG-eq:Intro-VFP} combines both conservative and dissipative effects. In fact it can be cast into the GENERIC form by introducing an excess-energy variable $e$, depending only on time, that captures the fluctuation of energy due to dissipative effects (but does not change the evolution of the system). The building blocks of the GENERIC for the augmented system for $(\rho,e)$ can be easily deduced from the conservative and dissipative effects of the original Vlasov-Fokker-Planck equation. Moreover, this GENERIC structure can be derived from the large-deviation rate functional of the empirical process \eqref{DOG-eq:Emperical-Measure}. We refer to \cite{DuongPeletierZimmer13} for more information.  This suggests that our method could be applied to other GENERIC systems.

\medskip

\emph{Gradient flows and large-deviation principles.} As mentioned in the introduction, this approach using the duality formulation of the rate functionals is motivated by our recent results on the connection between generalised gradient flows and large-deviation principles~ \cite{AdamsDirrPeletierZimmer11,AdamsDirrPeletierZimmer13,DuongPeletierZimmer14,DuongPeletierZimmer13,DirrLaschosZimmer12,
MielkePeletierRenger14}. We want to discuss here how the two overlap but are not the same. In \cite{MielkePeletierRenger14}, the authors show that if~$\opN^\ve$ is the adjoint operator of a generator of a Markov process that satisfies a \textit{detailed balance condition}, then the evolution \eqref{DOG-eq:Formal-Evolution} is the same as the generalised gradient flow induced from a large-deviation rate functional, which is of the form $\int_0^T \mathscr{L}^\varepsilon(\rho_t,\dot{\rho}_t)\,dt$, of the underlying empirical process. The generalised gradient flow is described via the $\Psi$-$\Psi^*$ structure as in \eqref{psi-psi^* formulation} with $\mathscr{L}^\varepsilon(z,\dot{z})=\Psi_\varepsilon(z,\dot{z})+\Psi_\varepsilon^*(z,-\mathsf{D}\mathcal{E}_\varepsilon(z))+\langle \mathsf{D}\mathcal{E}_\varepsilon(z),\dot{z}\rangle$. Moreover, $\mathcal{E}_\varepsilon$ and $\Psi_\varepsilon$ can be determined from $\mathscr{L}^\varepsilon$~\cite[Theorem 3.3]{MielkePeletierRenger14}. However, it is not clear if such characterisation holds true for systems that do not satisfy detailed balance. In addition, there exist (generalised) gradient flows for which we currently do not know of any corresponding microscopic particle systems, such as the Allen-Cahn and Cahn-Hilliard equations. 

\medskip

\emph{Quantification of coarse-graining error}. The use of the rate functional in a central role allows us not only to derive the limiting coarse-grained system but also to obtain quantitative estimates of the coarse-graining error. Existing quantitative methods such as \cite{Legoll2010} and \cite{GrunewaldOttoVillaniWestdickenberg09} only work for gradient flows systems since they use crucially the gradient flow structures. The essential estimate that they need is the energy-dissipation inequality, which is similar to \eqref{eq: bound of energy and fisher information}. Since we are able to obtain this inequality from the duality formulation of the rate functionals, our method would offer an alternative technique for obtaining quantitative estimate of the coarse-graining error for both dissipative and non-dissipative systems. We address this issue in detail in a companion article \cite{DLPSS-TMP}.

\medskip

\emph{Other stochastic processes.} The key ingredient of the method is the duality structure of the rate functional \eqref{DOG-eq:Abstract-Duality-Form-Rate-Fn} and \eqref{DOG-eq:Large-Dev-Rate-Fn-Generator-Form}. This duality formulation holds true for many other stochastic processes; indeed, the `Feng-Kurtz' algorithm~(see chapter 1 of~\cite{FengKurtz06}) suggests  that the large-deviation rate functional for a very wide class of Markov processes can  be written as 
\begin{equation*}
I(\rho)=\sup_{f}\left\{\langle f_T,\rho_T\rangle-\langle f_0,\rho_0\rangle -\int_0^T \langle \dot{f}_t,\rho_t\rangle\,dt-\int_0^T \mathcal{H}(\rho_t,f_t)\,dt\right\},
\end{equation*}
where $\mathcal{H}$ is an appropriate limit of `non-linear' generators. The formula \eqref{DOG-eq:Large-Dev-Rate-Fn-Generator-Form} is a special case. As a result, we expect that the method can be extended to this same wide class of Markov processes.

\begin{appendices}

\section{Proof of Lemma~\ref{DOG-lem:two-defs-RFI}}
\label{app:RFI-lemma-proof}
Define $\tilde \RelFI(f)$ to be the right-hand side in~\eqref{def:RelFI-ac},
\[
\tilde \RelFI(f):=
\begin{cases}
\displaystyle
\int_{\R^{2d}}\Bigl|\frac{\nabla_p f}{f}\mathds{1}_{\{f>0\}} + \frac pm \Bigr|^2 f\,dqdp,\qquad&\text{if}\quad \nabla_p f\in L^1_{\mathrm{loc}}(dqdp),\\
\infty&\text{otherwise}.
\end{cases}
\]
for $f\in L^1(\R^{2d})$.
We need to show that $\tilde \RelFI(f) = \RelFI(f\,dqdp|\mu)$. 

First assume that $\tilde \RelFI$ is finite. 
Then $\frac{\nabla_p f}{f}\mathds{1}_{\{f>0\}} + \frac pm\in L^2(fdqdp)$, which implies the following stronger statement.

\begin{lemma}
\label{lem:L2nabla}
One has
$$\frac{\nabla_p f}{f}\mathds{1}_{\{f>0\}} + \frac pm\in L^2_\nabla(fdqdp),$$
where the space $L^2_\nabla(fdqdp)$ is defined as the closure of $\left\{\nabla_p \varphi\,:\,\varphi\in C_c^\infty(\R^{2d})\right\}$ with 
respect to the norm $\|\cdot\|^2_{fdqdp}:= \int_{\R^{2d}}|\cdot|^2 \,fdqdp$.
\end{lemma}

Assuming Lemma \ref{lem:L2nabla} for the moment we rewrite $\tilde \RelFI(f)$ as
\begin{align*}
\tilde \RelFI(f) &= \int_{\R^{2d}}\Bigl|\tfrac{\nabla_p f}{f}\mathds{1}_{\{f>0\}} + \tfrac pm  \Bigr|^2\, f \,dqdp
=\left\|-\nabla_p\cdot\left(f\left(\tfrac{\nabla_p f}{f}\mathds{1}_{\{f>0\}}+\tfrac{p}{m}\right)\right)\right\|^2_{-1,(fdqdp)}\\
&=\|-\nabla_p\cdot\left(\mathds{1}_{\{f>0\}}\,\nabla_p f+f\tfrac{p}{m})\right)\|^2_{-1,(fdqdp)}\\
&=\|-\nabla_p\cdot\left(\nabla_p f+f\tfrac{p}{m})\right)\|^2_{-1,(fdqdp)},
\end{align*}
where $\|\cdot\|_{-1,fdqdp}$ is the dual norm (in duality with $L^2_\nabla(fdqdp)$) from \cite{DuongPeletierZimmer13} and 
$\mathds{1}_{\{f>0\}}\,\nabla_p f=\nabla_p f$ holds due to Stampacchia's Lemma \cite[Theorem A.1]{KinderlehrerStampacchia00}. 
Following the variational 
characterization of $\|\cdot\|_{-1,(fdqdp)}$ from \cite[(11)]{DuongPeletierZimmer13} we finally obtain
\begin{align*}
\tilde \RelFI(f)&=\sup_{\varphi\in C_c^\infty(\R^{2d})}2\int_{\R^{2d}}\left(\nabla_p\varphi\cdot\frac{p}{m}-\mathds{1}_{\{f>0\}}\,\Delta_p\varphi-\tfrac{1}{2}|\nabla_p\varphi|^2\right)\,f\,dqdp\\
&=\sup_{\varphi\in C_c^\infty(\R^{2d})}2\int_{\R^{2d}}\left(\nabla_p\varphi\cdot\frac{p}{m}-\Delta_p\varphi-\tfrac{1}{2}|\nabla_p\varphi|^2\right)\,f\,dqdp,
\end{align*}
which is the claimed result. The same reference also provides that $\tilde \RelFI=\infty$ iff $\RelFI(f\,dqdp|\mu)=\infty$.

\begin{proof}[Proof of Lemma \ref{lem:L2nabla}]
We assume that $\tfrac{\nabla_p f}{f}\mathds{1}_{\{f>0\}}+\tfrac{p}{m}\in L^2(fdqdp)$ and show that the two individual terms 
$\tfrac{\nabla_p f}{f}\mathds{1}_{\{f>0\}}$ and $\tfrac{p}{m}$ are in $L^2_\nabla(fdqdp)$. 
Choose a smooth cut-off function $\eta_R = \eta(x/R)$ with $\eta:\R^{2d}\rightarrow \R$, 
$\eta=1$ on $B_1(0)$ and $\eta=0$ in $\R^{2d}\setminus B_{2}(0)$. Then
\begin{align*}
-\int_{\R^{2d}} \eta_R \,  \frac pm \cdot \frac{\nabla_p f}{f}\mathds{1}_{\{f>0\}}\,f
&=-\int_{\R^{2d}} \eta_R \, \frac pm \cdot\nabla_p f\,\mathds{1}_{\{f>0\}}
=-\int_{\R^{2d}} \eta_R \, \frac pm \cdot\nabla_p (\mathds{1}_{\{f>0\}}\,f)\\
&= +\frac 1m\int_{\R^{2d}} \Bigl[\eta_R\, d + p\cdot\nabla_p \eta_R\Bigr]\mathds{1}_{\{f>0\}}\,f
\leq \frac dm  + \int_{\R^{2d}} p\cdot \nabla _p \eta_R\,f =: b(R).  
\end{align*}
As $R\to\infty$, the bound $b(R)$ converges to $d/m$. 

Therefore we have
\begin{align*}
\int_{\R^{2d}}\eta_R\left[\left|\tfrac{\nabla_p f}{f}\mathds{1}_{\{f>0\}}\right|^2+\left|\tfrac{p}{m}\right|^2\right]f
&=\int_{\R^{2d}}\eta_R\left|\tfrac{\nabla_p f}{f}\mathds{1}_{\{f>0\}}+\tfrac{p}{m}\right|^2f-2\int_{\R^{2d}}\eta_R\,\nabla_p f\cdot \tfrac{p}{m}\mathds{1}_{\{f>0\}}\\
&\leq 2b(R) + \int_{\R^{2d}}\eta_R\left|\tfrac{\nabla_p f}{f}\mathds{1}_{\{f>0\}}+\tfrac{p}{m}\right|^2f.
\end{align*}
By passing to the limit $R\rightarrow\infty$ we obtain
\begin{align*}
\lim_{R\rightarrow\infty}\int_{\R^{2d}}\eta_R\left[\left|\tfrac{\nabla_p f}{f}\mathds{1}_{\{f>0\}}\right|^2+\left|\tfrac{p}{m}\right|^2\right]f
\leq\int_{\R^{2d}}\left|\tfrac{\nabla_p f}{f}\mathds{1}_{\{f>0\}}+\tfrac{p}{m}\right|^2f+\frac{2d}m<\infty
\end{align*}
and thus $\tfrac{\nabla_p f}{f}\mathds{1}_{\{f>0\}},\frac{p}{m}\in L^2(fdqdp)$. 
To conclude the proof of Lemma \ref{lem:L2nabla} it remains to show that $\tfrac{\nabla_p f}{f}\mathds{1}_{\{f>0\}},\frac{p}{m}$ can 
be approximated by gradients of $C_c^\infty$-functions. To this end we consider, for $\ve>0$, the smooth cut-off function $\eta_\ve:=\eta(x\ve)$ 
with $\eta$ as above and define
\begin{align*}
\varphi_\ve:=\left[\log\left(\frac{1}{\ve}\wedge\left(f\vee\ve\right)\right)-\log\ve\right]\eta_\ve.
\end{align*}
Then $\varphi_\e$ has compact support in $\R^{2d}$. Note that $\varphi_\e$ is 
not necessarily smooth, but by convolution with a mollifier we can also achieve smoothness.  For the gradient one obtains
\begin{align*}
\nabla_p\varphi_\e=
\begin{cases}
\mathds{1}_{B_{\frac{1}{\ve}}(0)}\frac{\nabla_p f}{f} + \mathds{1}_{B_{\frac{2}{\ve}}(0)\setminus B_{\frac{1}{\ve}}(0)}
\left(\eta_\ve\frac{\nabla_p f}{f} + \nabla_p\eta_\ve(\log f-\log\ve)\right)& \quad\text{ for }\{\ve\leq f\leq\frac{1}{\ve}\}\\
\mathds{1}_{B_{\frac{2}{\ve}}(0)\setminus B_{\frac{1}{\ve}}(0)}\nabla_p\eta_\ve\left(\log\frac{1}{\ve}-\log\ve\right)&\quad\text{ for }\{f>\frac{1}{\ve}\}\\
0&\quad\text{ for }\{f<\ve\}
\end{cases}
\end{align*}
Our aim is to show that $\left\|\frac{\nabla_{p}f}{f}\mathds{1}_{\{f>0\}}-\nabla_p\varphi_\varepsilon\right\|_{fdqdp}\rightarrow 0$ as $\ve\rightarrow 0$. Indeed,
\begin{align*}
&\int_{\R^{2d}}\left|\tfrac{\nabla_{p}f}{f}\mathds{1}_{\{f>0\}}-\nabla_p\varphi_\varepsilon\right|^2\,f=\\
&\quad\int_{\{f<\ve\}}\left|\tfrac{\nabla_p f}{f}\mathds{1}_{\{f>0\}}\right|^2f + 
\int_{\{f>\frac{1}{\ve}\}}\left|\tfrac{\nabla_{p}f}{f}\mathds{1}_{\{f>0\}}-\nabla_p\eta_\ve\left(\log\frac{1}{\ve}
-\log\ve\right)\mathds{1}_{B_{\frac{2}{\ve}}(0)\setminus B_{\frac{1}{\ve}}(0)}\right|^2\,f\\
&\quad+\int_{\{\ve\leq f\leq\frac{1}{\ve}\}}\left|(1-\eta_\ve)\tfrac{\nabla_p f}{f}\mathds{1}_{\{f>0\}}
-\nabla_p\eta_\ve(\log f-\log\ve)\right|^2
\mathds{1}_{B_{\frac{2}{\ve}}(0)\setminus B_{\frac{1}{\ve}}(0)}\,f\\
&\quad+\int_{\{\ve\leq f\leq\frac{1}{\ve}\}}\left|\tfrac{\nabla_p f}{f}\mathds{1}_{\{f>0\}}\right|^2\mathds{1}_{\mathbb{R}^{2d}
\setminus B_{\frac{2}{\ve}}(0)}\,f\\
&\quad=:\mathrm{I}_\ve+\mathrm{II}_\ve+\mathrm{III}_\ve+\mathrm{IV}_\ve.
\end{align*}
Since $\tfrac{\nabla_p f}{f}\mathds{1}_{\{f>0\}}\in L^2(f\,dqdp)$ we directly conclude that $\mathrm{I}_\ve$ and $\mathrm{IV}_\ve$ vanish in the limit as $\ve\rightarrow 0$. 
Concerning $\mathrm{II}_\ve$ and $\mathrm{III}_\ve$ we note that, for $\{\ve\leq f\leq\frac{1}{\ve}\}$, one has  
\begin{align*}
\left|\nabla_p\eta_\ve(\log f-\log\ve)\right|^2\leq|\nabla_p\eta_\ve|^2\left|\log{1/\ve}-\log\ve\right|^2 
 =|\nabla_p\eta_\ve|^2\left|2\log\frac{1}{\ve}\right|^2\leq C\ve, 
\end{align*}
where we exploited $|\nabla_p\eta_\ve|^2\leq C\ve^2$ and $\left(\log\frac{1}{\ve}\right)^2\leq C\frac{1}{\ve}$ for some $\ve$-independent constant $C$. 
This shows that also $\mathrm{II}_\ve$ and $\mathrm{III}_\ve$ vanish in the limit as $\ve\rightarrow 0$. 
To sum up, we conclude that $\frac{\nabla_{p}f}{f}\mathds{1}_{\{f>0\}}\in L^2_\nabla(fdqdp)$. 
The calculation for $\frac{p}{m}=\nabla_p\left(\frac{|p|^2}{2m}\right)$ is similar.
\end{proof}


\section{Proof of Theorem~\ref{DOG-thm:VFP-Ent-Fisher-Inf-Bounds}}
\label{DOG-sec:App-PDE}
In this appendix, we prove Theorem~\ref{DOG-thm:VFP-Ent-Fisher-Inf-Bounds} using the method of the duality equation; see e.g.~\cite{AronsonCrandallPeletier82,RosenauKamin82,BertschKersnerPeletier85,Eidus90} or~\cite[Ch.~9]{BKRS15} for examples. 
Throughout this appendix $\gamma$ is fixed.

We recall the functional 
$I^\gamma: C([0,T];\mathcal{P}(\R^{2d}))\rightarrow\mathbb{R}$ defined in \eqref{def:I-gamma} 
\begin{multline}
I^\gamma(\rho)= \sup\limits_{f\in C_b^{1,2}(\mathbb{R}\times\mathbb{R}^{2d})}\bigg{[} 
\int\limits_{\mathbb{R}^{2d}}f_T\,d\rho_T
-\int\limits_{\mathbb{R}^{2d}}f_0\,d\rho_0
-\int\limits_0^T\int\limits_{\mathbb{R}^{2d}}
\Bigl{(}\partial_t f_t+\opL_{\rho_t} f_t\Bigr{)}\,d\rho_tdt
-\frac{\gamma^2}{2}\int\limits_0^T\int
\limits_{\mathbb{R}^{2d}}\left|\nabla_p f_t\right|^2d\rho_tdt\bigg{]},
\label{def:I-gamma2}
\end{multline}
where $\opL_\nu$ is given by
\begin{equation}\label{App-def:L-VFP-rescaled}
\opL_\nu f = \gamma J\nabla (H+\psi*\nu)\cdot \nabla f - \gamma^2 \frac pm \cdot\nabla_p f + \gamma^2 \Delta_p f.
\end{equation}

In addition to the duality definition of the Fisher Information~\eqref{DOG-eq:Relative-Fisher-Information} we will use the Donsker-Varadhan duality characterization of the relative entropy~\eqref{DOG-eq:Relative-Entropy} for two probability measures (see e.g.\,~\cite[Lemma 1.4.3]{DupuisEllis97}) 
\begin{align*}
\RelEnt(\nu|\mu)=\sup\limits_{\phi\in C_c^\infty(\R^{2d})} \int_{\R^{2d}} \phi d\nu - \log\int_{\R^{2d}}e^\phi d\mu,
\end{align*}
which implies the corresponding characterization of the free energy~\eqref{def:FreeEnergy}
\begin{align}\label{def:VFP-FreeEnergy-Var}
\M{F}(\nu)=\sup\limits_{\phi\in C_c^\infty(\R^{2d})} \int_{\R^{2d}}\Bigl[ \phi+\frac 12 \psi\ast\nu \Bigr]d\nu -\log\int_{\R^{2d}} e^{\phi-H} dx +\log Z_H.
\end{align}

We first present some intermediate results which we will use to prove Theorem \ref{DOG-thm:VFP-Ent-Fisher-Inf-Bounds}.
\begin{lemma}
\label{lemma:app:bonds-on-rho}
Let $\rho\in C([0,T];\mathcal P(\R^{2d}))$.
\begin{enumerate}
\item The maps $t \mapsto \psi*\rho_t$ and $t \mapsto \nabla\psi*\rho_t$ are continuous from $[0,T]$ to $C_b(\R^{d})$;
\item If $I^\gamma(\rho),\RelEnt(\rho_0|Z_H^{-1}e^{-H})<\infty$, then $\int H\rho_t <\infty$ for all $t\in [0,T]$.
\end{enumerate}
\end{lemma}
\begin{proof}
The first part follows from the bound $\psi\in W^{1,1}(\R^d)\cap C^2_b(\R^d)$. Fix $\e>0$, $t\in [0,T]$, and take a sequence $t_n\to t$. For each $n$, choose $x_n\in \R^{2d}$ such that $|\psi*(\rho_t-\rho_{t_n})|(x_n)\geq \|\psi*(\rho_t-\rho_{t_n})\|_\infty - \e/2$. Since $\rho_{t_n}\to\rho_t$ narrowly, $\{\rho_{t_n}\}_n$ is tight, implying that $x_n$ can be chosen bounded; therefore there exists a subsequence (not relabelled) such that $x_n\to x$ as $n\to\infty$. Then 
\begin{align*}
|(\psi*\rho_t)(x_n)&-(\psi*\rho_{t_n})(x_n)| \leq |(\psi*\rho_t)(x_n)-(\psi*\rho_{t})(x)|+ |(\psi*\rho_t)(x)-(\psi*\rho_{t_n})(x)| \\
 &\hskip0.5\textwidth + |(\psi*\rho_{t_n})(x)-(\psi*\rho_{t_n})(x_n)|.
\end{align*}
The last term on the right-hand side satisfies
\[
|(\psi*\rho_{t_n})(x)-(\psi*\rho_{t_n})(x_n)|\leq \int_{\mathbb{R}^{2d}} |\psi(x-y)-\psi(x_n-y)|\rho_{t_n}(y,z)\,dy\,dz \to 0
\]
since $\psi(x_n-\cdot)\to \psi(x-\cdot)$ uniformly,
and a similar argument applies to the first term. The middle term converges to zero by the narrow convergence of $\rho_{t_n}$ to $\rho_t$. This proves that the function $t\mapsto \psi*\rho_t$ is continuous; a similar argument applies to $t\mapsto \nabla \psi*\rho_t$.

For the second part, we take in~\eqref{def:I-gamma2} the function $f(q,p,t) = \zeta(H(q,p))$, where $\zeta\in C^\infty([0,\infty))$ is a smooth, bounded, increasing truncation of the function $f(s) = s$, satisfying $0\leq \zeta'\leq 1$ and $\zeta''\leq 0$. Then we find\begin{align*}
\int_{\R^{2d}} \zeta(H)\rho_\tau &- \int_{\R^{2d}} \zeta(H) \rho_0 - I^\gamma(\rho)
\leq \int\limits_0^\tau\int\limits_{\mathbb{R}^{2d}}\bigg{(}-\gamma\zeta'\frac pm \cdot \nabla_q \psi*\rho_t  
+ \gamma^2\Bigl(\zeta'' + \tfrac12 {\zeta'}^2-\zeta'\Bigr)\frac {p^2}{m^2}
+\gamma^2 \zeta'\frac dm\bigg{)}\,d\rho_tdt\\
&\leq \int\limits_0^\tau\int\limits_{\mathbb{R}^{2d}} \biggl( \frac12 \zeta'|\nabla_q \psi*\rho_t|^2 + \gamma^2 \zeta'\frac dm\bigg{)}\,d\rho_tdt
\leq \frac\tau2 \|\nabla_q\psi\|^2_\infty + \gamma^2 \frac dm \, \tau . \\
\end{align*}
The result follows upon letting $\zeta$ converge to the identity.

Note that this inequality gives a bound on $\int H\rho_t$ for fixed $\g$, but this bound breaks down when $\g\to\infty$. The bound~\eqref{ineq:bound-Hrho}, which is directly derived from~\eqref{DOG-eq:VFP-Ent+FI-Bounds}, gives a $\g$-independent estimate.
 \end{proof}

In the next few results we study certain properties of an auxiliary PDE and its connection to the rate functional.
\begin{theorem}\label{FIR-thm:VFP-Aux-PDE-Res}
Given $\phi\in C_c^\infty(\R^{2d})$ and $\varphi\in C_c^\infty([0,T]\times\R^d)$, there exists a function $f\in L^1_{\mathrm {loc}}([0,T]\times\R^{2d})$ which satisfies the following equation a.e. in $L^1_{\mathrm{loc}}([0,T]\times\R^{2d})$ (i.e.\ for each compact set $K\subset [0,T]\times \R^{2d}$, the equation is satisfied with all weak derivatives and all terms in $L^1(K)$):\begin{subequations}
\begin{align}\label{FIR-eq:VFP-Aux-f-PDE}
&\partial_t f + \mathscr L_\rho f + \frac{\g^2}{2}|\nabla_p f|^2 + \g J\nabla H \cdot\nabla\psi\ast\rho_t= -\g^2 \Bigl( \Delta_p\varphi -\nabla_pH\cdot\nabla_p\varphi - \frac 12 |\nabla_p\varphi|^2 \Bigr), \\
& f|_{t=T}=\phi & \label{IR-eq:VFP-Aux-f-PDE-Final-Data}
\end{align}
\end{subequations}
where $\mathscr L_\rho$ is defined in~\eqref{App-def:L-VFP-rescaled}. The final-time condition~\eqref{IR-eq:VFP-Aux-f-PDE-Final-Data} is satisfied in the sense of traces in $L^1_{\mathrm{loc}}(\R^{2d})$ (which are well-defined since $\partial_t f\in L^1_{\mathrm{loc}}([0,T]\times\R^{2d})$). The solution satisfies $|f|\leq C(1+H)^{1/2}$ for each $t\in [0,T]$ and almost everywhere in $\mathbb{R}^{2d}$, for some constant $C>0$. Finally,  
\begin{align}\label{FIR-res:VFP-Dec-Exp-Map}
t\mapsto \int_{\R^{2d}} e^{f_t-H}dx \quad \text{is non-decreasing.}
\end{align}
\end{theorem}
\begin{proof}
The Hopf-Cole tranformation $f=2\log g$ and the time reversal $t\mapsto T-t$ transform equation~\eqref{FIR-eq:VFP-Aux-f-PDE} into
\begin{align}\label{FIR-eq:Aux-PDE-Hopf-Cole}
\partial_t g - \mathscr L_{\rho} g = - \frac{g}{2}\Bigl( -\g J\nabla H\cdot\nabla\psi\ast\rho_t - \g^2\Bigl[ \Delta_p\varphi -\nabla_pH\cdot\nabla_p\varphi - \frac 12 |\nabla_p\varphi|^2 \Bigr]\Bigr),
\end{align}
with initial datum (now at time zero) $g_0=e^{\phi/2}$. The analysis of equation~\eqref{FIR-eq:Aux-PDE-Hopf-Cole} is non-standard and therefore we study this equation separately  in Appendix~\ref{App:Aux-PDE}. The existence and uniqueness of a solution, with this initial value, follow from Corollary~\ref{Aux-Cor:WellPose-Original}. The solution $g$ satisfies~\eqref{FIR-eq:Aux-PDE-Hopf-Cole} a.e.\ in $L^1_{\mathrm{loc}}([0,T]\times\R^{2d})$ by Proposition~\ref{FIR-prop:L1-loc-Prop}. Furthermore, by Proposition~\ref{FIR-prop:VFP-Aux-PDE-Control} there exist constants $\alpha_1,\alpha_2,\beta_1,\beta_2,\omega_1,\omega_2$ such that 
\begin{align*}
\alpha_1\exp\left(-\beta_1 t\sqrt{\omega_1+ H}\right)\leq g \leq \alpha_2\exp\left(\beta_2 t\sqrt{\omega_2+ H}\right) .
\end{align*}
Finally, by Proposition~\ref{FIR-prop:Aux-PDE-Decreasing} we have
\begin{align*}
t\mapsto \int_{\R^{2d}} g_t^2 e^{-H}dx \quad \text{is non-increasing.}
\end{align*}
Transforming back to $f$ we find the result. 
\end{proof}
To prove the second main result on the auxiliary equation~\eqref{FIR-eq:VFP-Aux-f-PDE}, which is Proposition \ref{FIR-lem:VFPRate-Fn-Subst-Sol} below, we will need the following lemma. For the rest of this appendix
we write $*_t$ for convolution in time and $*_x$ for convolution in space ($x=(q,p)$). (The convolution $\psi*_x\rho$ is the same as the notation $\psi*\rho$ used in the rest of this paper.)
\begin{lemma}
\label{FIR-lem:Reg-Abstract-Aux-PDE}
Let $f$ satisfy 
\begin{align}
\label{FIR-eq:Reg-Abstract-Aux-PDE}
\partial_t f + \mathscr L_{\rho_t} f+\frac {\gamma^2}2 |\nabla_p f|^2=\Phi,
\end{align}
a.e.\ in $L^1_{\mathrm{loc}}(\R\times\R^{2d})$ with $\Phi\in L^1_{\mathrm{loc}}(\R\times\R^{2d})$.  Define $f_{\d}:=\nu_\d\ast_x f$ and $f_{\vep}:=\eta_\vep \ast_t f$, where $\eta_\vep = \eta_\vep(t)$ is a regularizing sequence in the $t$-variable and $\nu_\d=\nu_\d(q,p)$ is a regularizing sequence in the $(q,p)$-variables. Then we have
\begin{align}
&\partial_t f_{\d} +\mathscr L_{\rho_t}f_{\d}+\frac{\g^2}{2}|\nabla_p f_{\d}|^2 \leq \nu_\d\ast_x \Phi 
+\g\d\|d^2H\|_{L^\infty} (\nu_\d\ast_x|\nabla f|+\g \nu_\d\ast_x|\nabla_p f|)\nonumber \\
&\qquad\qquad\qquad\qquad\qquad\qquad + \g \Bigl(J\nabla \psi\ast_x\rho_t\cdot\nabla f_{\delta} -\nu_\d\ast_x (J\nabla \psi\ast_x\rho_t\cdot\nabla f)\Bigr),\label{FIR-eq:Smooth-Lem-Spatial}\\
&\partial_t f_{\vep} +\mathscr L_{\rho_t}f_{\vep}+\frac{\g^2}{2}|\nabla_p f_{\vep}|^2 \leq \eta_\vep\ast_t \Phi 
 + \g \Bigl(J\nabla \psi\ast_x\rho_t\nabla f_{\vep} -\eta_\vep\ast_t (J\nabla \psi\ast_x\rho_t\cdot\nabla f)\Bigr).\label{FIR-eq:Smooth-Lem-Time}
\end{align}
\end{lemma}
\begin{proof}[Proof of Lemma~\ref{FIR-lem:Reg-Abstract-Aux-PDE}]
Using~\eqref{FIR-eq:Reg-Abstract-Aux-PDE} and the definition of $\mathscr L_\rho$ we have
\begin{align}
0&=\int_{\R} \eta_\vep(t-\tau)\Bigl(\partial_t f +\g^2\Delta_p f -\g^2\nabla_pH\cdot\nabla_p f +\g J\nabla(H+\psi\ast_x\rho_t)\cdot\nabla f +\frac{\g^2}{2}|\nabla_p f|^2-\Phi \Bigr)(\tau,x)d\tau \nonumber\\
&=\Bigl( \partial_t f_{\vep} +\g^2\Delta_p f_{\vep} -\g^2\nabla_p H\cdot \nabla_p f_{\vep} +\g J\nabla(H+\psi\ast_x\rho_t)\cdot\nabla f_{\vep} 
\Bigr)(t,x) + \frac{\g^2}{2} \eta_\vep\ast_t|\nabla_p f|^2-\eta_\vep\ast_t \Phi (t,x)\nonumber\\
&\qquad \  +\g\int_{\R} \eta_\vep(t-\tau) \Bigl( J\nabla\psi\ast_x\rho_\tau -J\nabla\psi\ast_x\rho_t\Bigr)\nabla f(\tau,x)d\tau.
\end{align}
By Jensen's inequality we have $ \eta_\vep\ast_t|\nabla_p f|^2\geq |\nabla_p f_\vep|^2$. Substituting this inequality  into the relation above completes the proof of~\eqref{FIR-eq:Smooth-Lem-Time}. The proof of~\eqref{FIR-eq:Smooth-Lem-Spatial} follows similarly. 
\end{proof}

The next result connects the solution of the auxiliary equation~\eqref{FIR-eq:VFP-Aux-f-PDE} to the rate functional~\eqref{def:I-gamma2}.

\begin{proposition}\label{FIR-lem:VFPRate-Fn-Subst-Sol}
Let $f$ be the solution of~\eqref{FIR-eq:VFP-Aux-f-PDE}-\eqref{IR-eq:VFP-Aux-f-PDE-Final-Data} in the sense of Theorem~\ref{FIR-thm:VFP-Aux-PDE-Res}. Then for  $\tau\in [0,T]$ we have
\begin{multline}\label{FIR-eq:J-leq-I-Gen}
\int_{\R^{2d}} \rho_\tau\Bigl( f_\tau+\frac 12 \psi\ast_x\rho_\tau\Bigr) -\int_0^\tau\int_{\R^{2d}} \Bigl\{  \partial_t f + \mathscr L_\rho f + \frac{\g^2}{2}|\nabla_p f|^2 + \g J\nabla H \cdot\nabla\psi\ast_x\rho_t
\Bigr\}\, d\rho_tdt  \\ 
\leq I(\rho)+\M{F}(\rho_0)+\log\int_{\R^{2d}}e^{f_0-H}dx-\log Z_H.
\end{multline} 
\end{proposition}

\begin{proof}
We first show that for every $\tau\in[0,T]$,
\begin{align}\label{FIR-eq:LDRF-Alt-Def-Proof}
I(\rho) \geq \sup_{\tilde f\in \mathcal A} \int_{\R^{2d}} \rho\Bigl( \tilde f+\frac 12 \psi\ast_x\rho\Bigr)\Bigr|_0^\tau  -\int_0^\tau\int_{\R^{2d}} \Bigl\{  \partial_t \tilde f + \mathscr L_\rho \tilde f + \frac{\g^2}{2}|\nabla_p \tilde f|^2 + \g J\nabla H \cdot\nabla\psi\ast_x\rho_t
\Bigr\}\, d\rho_tdt,
\end{align}
where 
\begin{align*}
\mathcal A=\Bigl\{\tilde f\in C^{1,2}([0,T]\times\R^{2d}): |\partial_t \tilde f|, |\nabla \tilde f|^2, |\Delta \tilde f|\leq C(1+H) \Bigr\}.
\end{align*}
Formally, this follows from substituting in the rate functional~\eqref{def:I-gamma2} $f(t,x) = \bigl[\psi\ast_x\rho+ \tilde f\,\bigr](t,x)\chi_{[0,\tau]}(t)$  with $\tilde f\in \mathcal A$, and where $\chi_{[0,\tau]}$ is the characteristic function of the interval $[0,\tau]$. The rigorous proof follows by choosing in the rate functional~\eqref{def:I-gamma2} the function
\[
f_n = \delta_n*_t(\xi \delta_n *_t\psi*_x\rho) +  \tilde f\xi,
\]
for some $\tilde f\in\mathcal A$ and $\xi\in C_c^\infty((0,\tau))$. Here $\delta_n(t) := n\delta(nt)$ is an approximation of a Dirac. Upon rearranging the time convolutions, letting $n\to\infty$,  using Lemma~\ref{lemma:app:bonds-on-rho}, and letting $\xi$ converge to the function $\chi_{[0,\tau]}$, we recover~\eqref{FIR-eq:LDRF-Alt-Def-Proof}.

From~\eqref{FIR-eq:LDRF-Alt-Def-Proof} we now derive~\eqref{FIR-eq:J-leq-I-Gen}.
From here onwards we denote the expression in the supremum on the right hand side of~\eqref{FIR-eq:LDRF-Alt-Def-Proof} by $\mathcal J(\rho,\tilde f)$ and use the notation
\begin{align}\label{FIR-eq:J-I-Psi}
\Psi:=-\g^2 \Bigl( \Delta_p\varphi -\nabla_pH\cdot\nabla_p\varphi - \frac 12 |\nabla_p\varphi|^2 \Bigr).
\end{align}

Our aim is to substitute the solution $f$ of~\eqref{FIR-eq:VFP-Aux-f-PDE}-\eqref{IR-eq:VFP-Aux-f-PDE-Final-Data} into~\eqref{FIR-eq:LDRF-Alt-Def-Proof}. To do this, we first extend $f$ outside $[0,T]\times\R^{2d}$ by constants and define 
\begin{align*}
f_\d := \nu_\d \ast_x f, \qquad f_{\d,\vep}:=\eta_\vep\ast_t f_\d,
\end{align*}
where $\eta_\vep(t)$, $\nu_\d(q,p)$ are again regularizing sequences in time and space. The rest of the proof is divided into the following steps:
\begin{enumerate}
\item We first show that $\mathcal J(\rho,f_{\d,\vep})$ is well defined.
\item We then successively take the limits $\vep\rightarrow 0$ and $\delta\rightarrow 0$ in $\mathcal J(\rho,f_{\d,\vep})$. 
\item We finally show that the limit satisfies~\eqref{FIR-eq:J-leq-I-Gen}.
\end{enumerate}

\emph{Step 1.} Let us first show that $\mathcal J(\rho,f_{\d,\vep})$ is well defined. From Theorem~\ref{FIR-thm:VFP-Aux-PDE-Res} we know that $f$ satisfies $|f|\leq C(1+H)^{1/2}$, and therefore we find
\begin{align}
\label{est:f-delta-epsilon-derivatives}
|\partial_t f_{\d,\vep}|,\  |\Delta_p  f_{\d,\vep}|,\  |J\nabla\psi\ast_x\rho_t\cdot\nabla  f_{\d,\vep}|,\  |J\nabla H\cdot \nabla  f_{\d,\vep}|,\ |\nabla H\cdot\nabla  f_{\d,\vep}|\leq C(1+H), 
\end{align}
where the constant $C$  depends on $\delta$ and $\vep$. The last two objects are bounded since $|\nabla H|^2\leq C(1+H)$; similar estimates hold for $f_\delta$. These bounds combined with  
Lemma~\ref{lemma:app:bonds-on-rho} imply that the integrals in  $\mathcal J(\rho,f_{\d,\vep})$ are well defined and using~\eqref{FIR-eq:LDRF-Alt-Def-Proof} it follows that 
\begin{align*}
\mathcal J(\rho,f_{\d,\vep})\leq I(\rho).
\end{align*}

\emph{Step 2.} Now we consider the convergence of $\mathcal J(\rho,f_{\d,\vep})$ as $\vep\rightarrow 0$.
Since all the derivatives of $f$ in~\eqref{FIR-eq:VFP-Aux-f-PDE} are in $L^1_{\mathrm{loc}}((-2,T+2)\times\R^{2d})$ (recall that we have extended  $f$ by constant functions of $(q,p)$ outside $[0,T]$) the same is true for the corresponding derivatives of $f_{\d}:=\nu_\d\ast_x f$, and therefore using standard results, the following convergence results hold in $L^1_{\mathrm{loc}}(\R\times\R^{2d})$ as $\vep\rightarrow 0$,
\begin{align}\label{FIR-eq:f-vep,delta-to-delta}
f_{\d,\vep}\rightarrow f_\d, \ \partial_t f_{\d,\vep}\rightarrow\partial_t f_\d,  \ \nabla f_{\d,\vep}\rightarrow\nabla f_\d,  \ \Delta_p f_{\d,\vep}\rightarrow\Delta_p f_\d.
\end{align}

Let us first consider  the single-integral terms in $\mathcal J(\rho,f_{\d,\vep})$. Since $f_\d\in W^{1,1}(0,T; L^1(B_R))$ for any $R>0$, we have
\begin{align*}
f_{\d,\vep}\xrightarrow{\vep\rightarrow 0} f_\d \quad \text{in }  W^{1,1}(0,T; L^1(B_R)),
\end{align*}
which together with the trace theorem implies that  
\begin{align}
f_{\d,\vep}\Bigl|_{t=0,\tau}\xrightarrow{\vep\rightarrow 0} f_\d\Bigl|_{t=0,\tau}  \quad \text{ in } L^1(B_R) \text{ and a.e. along  a subsequence}.
\end{align}
Since the traces of $f_\d$ and $f_{\d,\vep}$ at $t=0,\tau$ are continuous in $(q,p)$, this convergence holds everywhere in~$B_R$.
Combining this convergence statement with the estimate~\eqref{est:f-delta-epsilon-derivatives}
and Lemma~\ref{lemma:app:bonds-on-rho} and using the dominated convergence theorem we find 
\begin{align*}
\int_{\R^{2d}}\rho_t f_{\d,\vep,t}\Bigl|_{t=0,\tau}\xrightarrow{\vep\rightarrow 0} \int_{\R^{2d}}  \rho_t f_{\d,t}\Bigl|_{t=0,\tau}.
\end{align*}

Now consider the double integral in $\mathcal J(\rho,f_{\d,\vep})$. Using the estimate~\eqref{FIR-eq:Smooth-Lem-Time} with the choice
\begin{align*}
\Phi=\partial_t f_\d+\mathscr L_{\rho_t} f_\d +\frac 12 |\nabla_p f_\d|^2,
\end{align*}
we have
\begin{align*}
&\limsup\limits_{\vep\rightarrow 0} \int_0^\tau\int_{\R^{2d}} \Bigl( \partial_t f_{\d,\vep}+\mathscr L_{\rho_t} f_{\d,\vep}+\frac {\g^2}2 |\nabla_p f_{\d,\vep}|^2+\g J\nabla H\cdot \nabla\psi\ast_x\rho_t \Bigr)\,d\rho_t dt \\
& \leq \limsup\limits_{\vep\rightarrow 0} \int_0^\tau\int_{\R^{2d}} \Bigl( \eta_\vep\ast_t\Bigl[\partial_t f_{\d}+\mathscr L_{\rho_t} f_{\d}+ \frac {\g^2}2 |\nabla_p f_\d|^2\Bigr]+\g J\nabla H\cdot \nabla\psi\ast_x\rho_t \Bigr)\,d\rho_t dt\\
&\qquad \quad \ \ +\int_0^\tau\int_{\R^{2d}} \Bigl( \g J\nabla\psi\ast_x\rho_t \cdot \nabla f_{\d,\vep} -\eta_\vep \ast_t(\g J\nabla \psi\ast_x\rho_t\cdot\nabla f_\d)\Bigr)\,d\rho_tdt 
\end{align*}
Since $t\mapsto \nabla \psi\ast_x\rho_t$ is continuous (see Lemma~\ref{lemma:app:bonds-on-rho}), it follows that for all $x\in \R^{2d}$
\begin{align*}
t\mapsto \int_{\R}\eta_\vep(t-s) \Bigl[\g J\nabla\psi\ast_x\rho_t -\g J\nabla\psi\ast_x\rho_s\Bigr]\nabla f_\d(s,x)\,ds \xrightarrow{\vep\rightarrow 0} 0 \text{ in }L^1(0,\tau).
\end{align*}
Using this convergence along with~\eqref{FIR-eq:f-vep,delta-to-delta} we find 
\begin{multline}\label{FIR-eq:Double-Int-vep-Limit}
\limsup\limits_{\vep\rightarrow 0} \int_0^\tau\int_{\R^{2d}} \Bigl( \partial_t f_{\d,\vep}+\mathscr L_{\rho_t} f_{\d,\vep}+\frac {\g^2}2 |\nabla_p f_{\d,\vep}|^2+\g J\nabla H\cdot \nabla\psi\ast_x\rho_t \Bigr)d\rho_t dt \\
\leq \int_0^\tau\int_{\R^{2d}}\Bigl( \partial_t f_{\d}+\mathscr L_{\rho_t} f_{\d}+\frac {\g^2}2 |\nabla_p f_{\d}|^2+\g J\nabla H\cdot \nabla\psi\ast_x\rho_t \Bigr)d\rho_t dt.
\end{multline}

Combining these terms and using $I(\rho)\geq \liminf_{\vep\rightarrow 0} \mathcal J(\rho,f_{\d,\vep})$ we have
\begin{align}\label{DOG-eq:J<=I-vep-0}
\int_{\R^{2d}} \rho\Bigl( f_{\d}+\frac 12 \psi\ast_x\rho\Bigr)\Bigl|_0^\tau - \int_0^\tau\int_{\R^{2d}}\Bigl( \partial_t f_{\d}+\mathscr L_{\rho_t} f_{\d}+\frac 12 |\nabla_p f_{\d}|^2+\g J\nabla H\cdot \nabla\psi\ast_x\rho_t \Bigr)d\rho_t dt 
\leq I(\rho) 
\end{align}

Now we study the $\d\rightarrow 0$ limit of~\eqref{DOG-eq:J<=I-vep-0}. Using a similar analysis as before, the following convergence results hold in $L^1_{\mathrm{loc}}(\R\times\R^{2d})$ as $\d\rightarrow 0$,
\begin{align*}
f_{\d}\rightarrow f, \ \partial_t f_{\d}\rightarrow\partial_t f,  \ \nabla f_{\d}\rightarrow\nabla f,  \ \Delta_p f_{\d}\rightarrow\Delta_p f.
\end{align*}
Since $f_T=\phi\in C_c^\infty(\R^{2d})$ (see Theorem~\ref{FIR-thm:VFP-Aux-PDE-Res}) and therefore $f_{\d,T}\rightarrow f_T$ everywhere, we have 
\begin{align}\label{FIR-eq:Pass-vep-0-Init-Point}
\int_{\R^{2d}}\rho_\tau f_{\d,\tau}\xrightarrow{\d\rightarrow 0} \int_{\R^{2d}}  \rho_\tau f_{\tau}. 
\end{align}

To pass to the limit in the right hand side of inequality~\eqref{FIR-eq:Double-Int-vep-Limit}, we use the estimate~\eqref{FIR-eq:Smooth-Lem-Spatial} with the choice $\Phi=\Psi-\g J\nabla H\cdot\nabla\psi\ast_x\rho_t$ (see~\eqref{FIR-eq:J-I-Psi} for the definition of $\Psi$), which leads to
\begin{align*}
&\limsup\limits_{\d\rightarrow 0} \int_0^\tau\int_{\R^{2d}}\Bigl( \partial_t f_{\d}+\mathscr L_{\rho_t} f_{\d}+\frac {\g^2}2 |\nabla_p f_{\d}|^2+\g J\nabla H\cdot \nabla\psi\ast_x\rho_t \Bigr)d\rho_t dt\\
&\leq \limsup\limits_{\d\rightarrow 0} \int_0^\tau\int_{\R^{2d}} \Bigl( \nu_\d\ast_x\Psi-\nu_\d\ast_x(\g J\nabla H\cdot\nabla\psi\ast_x\rho_t)+\g J\nabla H\cdot \nabla\psi\ast_x\rho_t  \Bigr)d\rho_t dt\\
& \qquad \  \ \ \ +  \int_0^\tau\int_{\R^{2d}} \Bigl(\g\d\|d^2H\|_{L^\infty} (\nu_\d\ast_x|\nabla f|+\g \nu_\d\ast_x|\nabla_p f|)  
+\g \Bigl[ J\nabla \psi\ast_x\rho_t\cdot \nabla f_{\delta} -\nu_\d\ast_x (J\nabla \psi\ast_x\rho_t\cdot\nabla f)\Bigr] \Bigr)d\rho_t dt\\
&=\int_0^\tau\int_{\R^{2d}} \Psi d\rho_tdt.
\end{align*}

The only term left is the single-integral term at $t=0$. Instead of passing to the limit, here we estimate as follows
\begin{align}\label{FIR-eq:Init-Val-Est}
\int_{\R^{2d}}\rho_0 \Bigl( f_{\d,0}+\frac 12 \psi\ast_x\rho_0\Bigr)\leq \M{F}(\rho_0)+ \log\int_{R^{2d}}e^{f_{\d,0}-H}-\log Z_H.
\end{align}
Let us first prove~\eqref{FIR-eq:Init-Val-Est}. Recall from the proof of Theorem~\ref{FIR-thm:VFP-Aux-PDE-Res} that 
\begin{align*}
f_0=2\log g_0 \leq 2\log \alpha_2 +2\beta_2 T\sqrt{\omega_2+H},
\end{align*}
where $\alpha_2,\beta_2,\omega_2$ are constants, and therefore 
\begin{align}\label{FIR-eq:f_d,0-Bound}
f_{\d,0}=\nu_\d\ast_x f_0\leq 2\log\alpha_2 +\beta_2T\Bigl(\d^2\|D^2\sqrt{\omega_2+H}\|_{L_\infty}+2\sqrt{\omega_2+H}\Bigr).
\end{align}
To arrive at the estimate above we have used 
\begin{align*}
\nu_\d\ast_x f(x)=\int f(x-y)\nu_\d(y)dy \leq \int \Bigl( |f(x)|+|\nabla f(x)| y +\frac 12 |y|^2 \|d^2 f\|_{L^\infty} \Bigr)\nu_\d(y)dy  \leq   |f(x)|+ \frac 12 \d^2 \|d^2 f\|_{L^\infty} ,
\end{align*}
for any $f\in C_b^2(\R^{2d})$ and $\nu_\d$ satisfying $\int\nu_\d=1$ and  $\int x\nu_\d(x)dx=0$. 

Furthermore, using the growth conditions on $H=p^2/2m+V(q)$ (see~\ref{cond:VFP:V1}) we find for the second derivative
\begin{align*}
d^2\sqrt{\omega_2+H}=-\frac{\nabla H\otimes\nabla H}{4(\omega+H)^{3/2}}+ \frac{d^2 H}{2\sqrt{\omega_2+H}} \quad\Longrightarrow\quad \left\|d^2\sqrt{\omega_2+H}\right\|_{L^\infty} <\infty,
\end{align*}
and therefore~\eqref{FIR-eq:f_d,0-Bound} implies that $|f_{\d,0}|\leq C(1+H)^{1/2}$. The estimate~\eqref{FIR-eq:Init-Val-Est} then follows by using a truncated version of $f_{\d,0}$ in the variational definition~\eqref{def:VFP-FreeEnergy-Var} of the free energy.  

Substituting~\eqref{FIR-eq:Init-Val-Est} into~\eqref{DOG-eq:J<=I-vep-0} we have
 \begin{multline}\label{FIR-eq:Gen-Finite-Par-J<I}
\int_{\R^{2d}} \rho\Bigl( f_{\d}+\frac 12 \psi\ast_x\rho_\tau\Bigr)\Bigr|_{t=\tau} -\int_0^\tau\int_{\R^{2d}} \Bigl\{  \partial_t f_{\d} + \mathscr L_\rho f_{\d} + \frac{\g^2}{2}|\nabla_p f_{\d}|^2 + \g J\nabla H \cdot\nabla\psi\ast_x\rho_t
\Bigr\} d\rho_tdt \\
\leq  I(\rho) + \M{F}(\rho_0)+ \log\int_{R^{2d}}e^{f_{\d,0}-H}-\log Z_H. 
\end{multline}
Using the bound $|f_{\d,0}|\leq C(1+H)^{1/2}$ and the dominated convergence theorem we find
\begin{align*}
 \log\int_{R^{2d}}e^{f_{\d,0}-H} \xrightarrow{\d\rightarrow 0}  \log\int_{R^{2d}}e^{f_{0}-H},
\end{align*}
and therefore passing to the limit $\d\rightarrow 0$ in~\eqref{FIR-eq:Gen-Finite-Par-J<I} gives
\begin{multline*}
\int_{\R^{2d}} \rho_\tau\Bigl( f_\tau+\frac 12 \psi\ast_x\rho_\tau\Bigr) -\int_0^\tau\int_{\R^{2d}} \Bigl\{  \partial_t f + \mathscr L_\rho f + \frac{\g^2}{2}|\nabla_p f|^2 + \g J\nabla H \cdot\nabla\psi\ast_x\rho_t
\Bigr\} d\rho_tdt  \\ 
\leq I(\rho)+\M{F}(\rho_0)+\log\int_{\R^{2d}}e^{f_0-H}dx-\log Z_H.
\end{multline*}
\end{proof}


We are now ready to prove Theorem~\ref{DOG-thm:VFP-Ent-Fisher-Inf-Bounds}. 

\begin{proof}[\textbf{Proof of Theorem~\ref{DOG-thm:VFP-Ent-Fisher-Inf-Bounds}}]
Combining~\eqref{FIR-eq:J-leq-I-Gen} with equation~\eqref{FIR-eq:VFP-Aux-f-PDE} we have
\begin{multline*}
\int_{\R^{2d}} \rho_\tau\Bigl( f_\tau+\frac 12 \psi\ast_x\rho_\tau\Bigr) \leq I(\rho)+\M{F}(\rho_0)+\log\int_{\R^{2d}} e^{f_0-H}dx-\log Z_H \\  - \g^2\int_0^\tau\int_{\R^{2d}} 
\Bigl( \Delta_p\varphi -\nabla_pH\cdot\nabla_p\varphi - \frac 12 |\nabla_p\varphi|^2\Bigr) d\rho_tdt.
\end{multline*}
Substituting this relation into the formula~\eqref{def:VFP-FreeEnergy-Var} for the free energy, and using $f|_{t=\tau}=\varphi$, we find 
\begin{align*}
\M{F}(\rho_\tau)&= \sup\limits_{\phi\in C_c^\infty(\R^{2d})} \int_{\R^{2d}} \Bigl( \phi + \frac 12 \psi\ast_x\rho_\tau \Bigr)\rho_\tau -\log\int_{\R^{2d}} e^{\phi-H} dx +\log Z_H \\ 
&\leq \sup\limits_{\phi\in C_c^\infty(\R^{2d})} I(\rho) + \M{F}(\rho_0|\mu)+\log\int_{\R^{2d}} e^{f_0-H}dx -\log\int_{\R^{2d}} e^{\phi-H} dx  
\\& \qquad \qquad \qquad \qquad - \g^2\int_0^\tau\int_{\R^{2d}} 
\Bigl( \Delta_p\varphi -\nabla_pH\cdot\nabla_p\varphi - \frac 12 |\nabla_p\varphi|^2\Bigr) d\rho_tdt.
\end{align*}
Rearranging and using~\eqref{FIR-res:VFP-Dec-Exp-Map} this becomes
\begin{align}
\M{F}(\rho_\tau) + \g^2\int_0^\tau\int_{\R^{2d}} 
\Bigl( \Delta_p\varphi -\nabla_pH\cdot\nabla_p\varphi - \frac 12 |\nabla_p\varphi|^2\Bigr) d\rho_tdt 
\leq   I(\rho) + \M{F}(\rho_0|\mu).
\label{ineq:FIR-pre-end}
\end{align}
Taking the supremum over $\varphi\in C_c^\infty(\R\times\R^{2d})$ and using a standard argument, based on $C^2$-seperability of $C_c^\infty$, we can move the supremum inside of the time integral and the definition of the relative Fisher Information~\eqref{DOG-eq:Relative-Fisher-Information} then gives
\[
\M{F}(\rho_\tau) + \frac{\gamma^2}2  \int_0^\tau \RelFI(\rho_t|\mu)\, dt \leq \M{F}(\rho_0) + I(\rho) .
\]
This completes the proof.
 \end{proof}


\section{Properties of the auxiliary PDE}\label{App:Aux-PDE}
In this appendix we will study the following equation in $[0,T]\times\R^{2d}$: 
\begin{equation}\label{FIR-eq:VFP-Main-Aux-PDE}
\begin{aligned}
&\partial_t g - J\nabla H\cdot \nabla g -J\nabla (\psi\ast\rho_t)\cdot \nabla g + \nabla_pH\cdot\nabla_pg - \Delta_p g -\frac{g}{2}\left(J\nabla H\cdot\nabla\psi\ast\rho_t-\Psi\right)=U,\\
& g|_{t=0}=g^0.
\end{aligned}
\end{equation}
In addition to providing well-posednes results (see Section~\ref{Aux-sec:Well-Posed}), in this section we also prove certain important properties of this equations such as a comparison principle and bounds at infinity (see Section~\ref{Aux-Sec:PDE-Prop}).  

Equation~\eqref{FIR-eq:Aux-PDE-Hopf-Cole} is a special case of~\eqref{FIR-eq:VFP-Main-Aux-PDE} with the choice
\begin{align*}
U=0, \ \ \ \Psi=-\Bigl(  \Delta_p\varphi -\nabla_p\varphi  \cdot\nabla_pH - \frac 12|\nabla_p\varphi|^2 \Bigr).
\end{align*}
Here and in the rest of this appendix we set $\g=1$, since the value of $\g$ plays no role in the discussion.

The results of this appendix are a generalization of~\cite[Appendix A]{Degond86}. In that reference Degond treats the case of equation~\eqref{FIR-eq:VFP-Main-Aux-PDE} without on-site and interaction potentials and without the friction term $\nabla_pH\cdot \nabla_p g$. We generalize the equation, while closely following his line of argument, and proving what are essentially similar results. 

The main difference in our treatment is the introduction of a weighted functional setting for the equation~\eqref{FIR-eq:VFP-Main-Aux-PDE}, in which the $L^2$-spaces, Sobolev spaces, and the weak formulation of the equation are all given a weight function $e^{-H}$. The choice of this weight function is closely connected to the fact that $e^{-H}$ is a stationary measure  both for the convective part of the equation $J\nabla H\cdot \nabla g$ and for the Ornstein-Uhlenbeck dissipative part $\nabla_pH\cdot\nabla_pg - \Delta_p g$. This weighted setting has the advantage of effectively eliminating all the unbounded coefficients in the equation.

\subsection{Well-posedness}\label{Aux-sec:Well-Posed}
Following Degond~\cite{Degond86} we introduce a change of variable
\begin{align}\label{FIR-eq:VFP-Lambda}
g \mapsto e^{\lambda t} g, \  \text{ with } \  \lambda\geq \frac{1}{2}\|\Psi\|_{L^\infty}+1,
\end{align}
which transforms~\eqref{FIR-eq:VFP-Main-Aux-PDE} into
\begin{equation}\label{FIR-eq:VFP-Trans-Aux}
\begin{aligned}
&\partial_t g - J\nabla H\cdot \nabla g -J\nabla (\psi\ast\rho_t)\cdot\nabla g + \nabla_pH\cdot\nabla_pg - \Delta_p g -\frac{g}{2}\left(J\nabla H\cdot\nabla\psi\ast\rho_t\right) + \Bigl(\lambda+\frac{1}{2}\Psi\Bigr)g=e^{-\lambda t}U,\\
& g|_{t=0}=g^0.
\end{aligned}
\end{equation}
In what follows we will study the well-posedness of~\eqref{FIR-eq:VFP-Trans-Aux}, and at the end of the section we will extrapolate the results to~\eqref{FIR-eq:VFP-Main-Aux-PDE}. 

Let us formally derive the weak formulation for~\eqref{FIR-eq:VFP-Trans-Aux}. Multiplying with a test function $\phi\in C_c^\infty([0,T)\times\R^{2d})$ and a weight $e^{-H}$, and using integration by parts, for the left-hand side of~\eqref{FIR-eq:VFP-Trans-Aux} we get
\begin{align*}
&\int_0^T\int_{\R^{2d}} \phi\Bigl\{ \partial_t g - J\nabla H\cdot \nabla g -J\nabla (\psi\ast\rho_t)\cdot\nabla g + \nabla_pH\cdot\nabla_pg - \Delta_p g    -\frac{g}{2}\left(J\nabla H\cdot\nabla\psi\ast\rho_t\right)  + \Bigl(\lambda+\frac{1}{2}\Psi\Bigr)g \Bigl\}e^{-H}\\
&=  \int_0^T\int_{\R^{2d}} \Bigl\{ g \Bigl( -\partial_t\phi + J\nabla H \cdot \nabla\phi +\frac{1}{2} J \nabla\psi\ast\rho_t\cdot \nabla \phi  +  \Bigl( \lambda + \frac 12 \Psi \Bigr)\phi  \Bigr)   
-\frac 12 \phi\, J\nabla\psi\ast\rho_t\cdot\nabla g + \nabla_p g\cdot\nabla_p \phi
\Bigr\} \, e^{-H} \\
& \qquad\qquad\qquad -\int_{\R^{2d}} g\phi\big|_{t=0} \,e^{-H}.
\end{align*}
The weight $e^{-H}$ causes cancellation of certain terms after integration by parts, as for instance   for the two convolution terms,
\begin{align*}
&\int_0^T\int_{\R^{2d}} \phi\Bigl(-J\nabla\psi\ast\rho_t\cdot\nabla g -\frac 12gJ\nabla H\cdot\nabla\psi\ast\rho_t \Bigr) e^{-H}\\
&=\int_0^T\int_{\R^{2d}} \phi\Bigl(-\frac 12 J\nabla\psi\ast\rho_t\cdot\nabla g -\frac 12 J\nabla\psi\ast\rho_t\cdot\nabla g -\frac 12gJ\nabla H\cdot\nabla\psi\ast\rho_t \Bigr) e^{-H}\\
&=\int_0^T\int_{\R^{2d}} \Bigl(-\frac 12 \phi J\nabla\psi\ast\rho_t\cdot\nabla g + \frac 12 gJ\nabla\psi\ast\rho_t\cdot\nabla \phi +\frac 12\phi gJ\nabla H\cdot\nabla\psi\ast\rho_t-\frac 12\phi gJ\nabla H\cdot\nabla\psi\ast\rho_t \Bigr) e^{-H}\\
&=\int_0^T\int_{\R^{2d}} \Bigl(-\frac 12 \phi J\nabla\psi\ast\rho_t\cdot\nabla g + \frac 12 gJ\nabla\psi\ast\rho_t\cdot\nabla \phi \Bigr) e^{-H}.
\end{align*}

These calculations suggest that we seek weak solutions in the space
\begin{align}\label{FIR-def:X-space}
X:=\Bigl\{  g\in L^2(0,T;L^2(\R^{2d};e^{-H})): \nabla_p g\in L^2(0,T;L^2(\R^{2d};e^{-H})) \Bigr\},
\end{align}
endowed with the norm 
\begin{align*}
\|g\|_X^2:= \|g\|^2_{L^2(L^2(e^{-H}))}+\|\nabla_p g\|^2_{L^2(L^2(e^{-H}))}.
\end{align*}
The subscript in the norm is  shorthand notation for $L^2(0,T;L^2(\R^{2d};e^{-H}))$. Note that $C_c^\infty((0,T)\times \R^{2d})$ is dense in $X$.

We will use $\|\cdot \|_{L^2}$ to indicate the $L^2$ norm without any weight, and $\langle \cdot,\cdot\rangle_{X',X}$ for the dual bracket between $X'$ (the dual of $X$) and $X$. 


For all $g\in X$ we can consider the combination $\partial_t g - J\nabla H\cdot\nabla g$ as a linear form on $C_c^\infty((0,T)\times \R^{2d})$ by interpreting the derivatives in the sense of distributions:
\[
\langle \partial_t g - J\nabla H\cdot\nabla g, \phi\rangle
:= -\int_0^T\int_{\R^{2d}} g (\partial_t\phi - J\nabla H\cdot\nabla \phi)e^{-H} \qquad\text{for }\phi\in C_c^\infty((0,T)\times \R^{2d}).
\]
Note that the weight function $e^{-H}$ yields no extra terms upon partial integration
If this linear form is bounded in the $X'$-norm, i.e.\ if the norm
\[
\|\partial_t g - J\nabla H\cdot\nabla g\|_{X'} := \sup \left\{
\int_0^T\int_{\R^{2d}} g (\partial_t\phi - J\nabla H\cdot\nabla \phi)e^{-H}: \quad \phi\in C_c^\infty((0,T)\times \R^{2d}), \ \|\phi\|_X \leq 1 \right\}
\]
is finite, then $\partial_t g - J\nabla H\cdot\nabla g\in X'$. We define $Y$ to be the space of such functions $g$:
\begin{equation}
\label{FIR-def:Y-space}
Y:=\Bigl\{ g\in X: \partial_t g - J\nabla H\cdot\nabla g\in X'\Bigr\}, \quad
\text{with norm }\|g\|^2_Y:=\|g\|^2_X + \|\partial_t g - J\nabla H\cdot\nabla g\|^2_{X'}.
\end{equation}

\medskip
We now define the variational equation (which is a weak form of~\eqref{FIR-eq:VFP-Trans-Aux}) to be
\begin{align}\label{FIR-eq:PDE-Weak-Form}
E_\lambda(g,\phi) = L_\lambda(\phi), \quad \forall \phi\in C_c^\infty([0,T)\times\R^{2d}),
\end{align}
where $E_\lambda:X\times C_c^\infty([0,T)\times\R^{2d})\rightarrow\R$ and $L_\lambda:C_c^\infty([0,T)\times\R^{2d})\rightarrow\R$ are given by
\begin{align}
&E_\lambda(g,\phi):= \int_0^T\int_{\R^{2d}} \Bigl\{ g \Bigl( -\partial_t\phi + J\nabla H \cdot \nabla\phi +\frac{1}{2} J \nabla\psi\ast\rho_t\cdot \nabla \phi  +  \Bigl( \lambda + \frac 12 \Psi \Bigr)\phi  \Bigr)  \nonumber \\ 
& \qquad \qquad\qquad\qquad\qquad\qquad-\frac 12 \phi\, J\nabla\psi\ast\rho_t\cdot\nabla g + \nabla_p g\cdot\nabla_p \phi
\Bigr\} \, e^{-H},\label{FIR-def:VFP-E-Bi-Form}\\
& L_\lambda(\phi):=\langle e^{-\lambda t}U,\phi\rangle_{X',X} + \int_{\R^{2d}}g^0\phi\big|_{t=0} \, e^{-H}. \label{FIR-def:VFP-L-Form}
\end{align}
 We use the subscript $\lambda$ to indicate that that the variational equation~\eqref{FIR-eq:PDE-Weak-Form} corresponds to the transformed equation~\eqref{FIR-eq:VFP-Trans-Aux}. 

We now state our main result. 
\begin{theorem}[Well-posedness] \label{FIR-thm:AuxPDE-Well-Posed}
Assume that 
\begin{align*}
\Psi\in C^2_c(\R^{2d}), \ \ U\in X', \  \text{ and } \  g^0\in L^2(\R^{2d};e^{-H}). 
\end{align*}
Then there exists a unique solution $g$ in $Y$ to the variational equation~\eqref{FIR-eq:PDE-Weak-Form}.
Furthermore the solution $g$ satisfies the initial condition in the sense of traces in $L^2(\R^{2d};e^{-H})$.
\end{theorem}

%
To prove Theorem~\ref{FIR-thm:AuxPDE-Well-Posed}, we require certain properties of $Y$. In the first lemma below, we prove an auxiliary result concerning the commutator of  a mollification with a multiplication. In the second lemma we prove that $C_c^\infty([0,T]\times\R^{2d})$ is dense in $Y$. In order to give meaning to the initial conditions (as required in Theorem~\ref{FIR-thm:AuxPDE-Well-Posed}) we need to prove a trace theorem. We prove this trace theorem and a Green formula (which gives meaning to `integration by parts') in the third lemma. At the end of this section we prove Theorem~\ref{FIR-thm:AuxPDE-Well-Posed}. 

\begin{lemma}\label{FIR-lem:Aux-Conv-Bound}
Define $\nu_\delta(x):=\delta^{-n}\nu(\frac x\delta)$ for some $\nu\in C_c^\infty(\R^{n})$, and consider $f\in W^{1,q}(\R^{n};
R^n)$, $h\in W^{1,r}(\R^n)$  where $1\leq q,r\leq \infty$ and $1\leq p<\infty$ satisfies $\frac 1p=\frac 1q+\frac 1r$. Then for any $\delta>0$ we have
\begin{align}\label{Aux-eq:Moll-Prop}
\|\nu_\delta\ast(f\cdot\nabla  h) -f\cdot \nu_\delta\ast\nabla h\|_{L^p}  
\leq \Bigl(\|\nabla f\|^p_{L^q}\Bigl(\int_{\R^n}|z|\,|\nabla\nu(z)|dz\Bigr)^{p}+\|\nu\|^p_{L^1}\|\div f\|^p_{L^q}\Bigr)^{1/p}\|h\|_{L^{r}} 
\end{align}
\end{lemma}
\begin{proof}
The argument of the norm on the left hand side of~\eqref{Aux-eq:Moll-Prop}  is
\begin{multline*}
\left(\nu_\d\ast(f\cdot \nabla h) - f\cdot\nu_\d\ast\nabla h\right)(x) = \int_{\R^n} \nu_\d (x-y) \left[f(x)-f(y)\right]\nabla h(y)dy\\
= \int_{\R^n} \Bigl(\nabla\nu_\d(x-y)\left[f(x)-f(y)\right] + \nu_\d(x-y)\div f(y) \Bigr)h(y)dy =: \mathrm{I}+\mathrm{II}.
\end{multline*}
Using  Young's and H{\"o}lder's inequalities on the second term gives
\begin{align*}
\|\mathrm{II}\|_{L^p} = \|\nu_\d\ast(h\div f)\|_{L^p} \leq \|\nu_\d\|_{L^1} \|h\div f\|_{L^p} \leq \|\nu_\d\|_{L^1} \|h\|_{L^r}\|\div f\|_{L^q}.
\end{align*}
For the first term we calculate, writing $\kappa_\d(z) := |z||\nabla\nu_\d(z)|$ and $k:=\|\kappa_\d\|_{L^1}$, that
\begin{align*}
\biggl|\int_{\R^n}\frac1k  \nabla\nu_\d(x-y)\bigl[f(x)&-f(y)\bigr]h(y)dy\biggr|^p
\leq \left(\frac1k \int_{\R^n} \kappa_\d(x-y)\frac{|f(x)-f(y)|}{|x-y|}\, |h(y)|dy\right)^p\\
&\leq \frac1k \int_{\R^n} \kappa_\d(x-y)\frac{|f(x)-f(y)|^p}{|x-y|^p}\, |h(y)|^pdy\\
&\leq\frac{\alpha^{q/p}}k\frac pq \int_{\R^n} \kappa_\d(x-y)\frac{|f(x)-f(y)|^q}{|x-y|^q}dy
+ \frac1{k\alpha^{r/p}} \frac pr \int_{\R^n} \kappa_\d(x-y) |h(y)|^r dy,
\end{align*}
and therefore
\begin{align*}
\|\mathrm{I}\|^p_{L^p} &=k^p\int_{\R^n}\Bigl|\frac1k\int_{\R^n} \nabla\nu_\d(x-y)\left[f(x)-f(y)\right]h(y)dy \Bigr|^pdx \\
&\leq {\alpha^{q/p}}k^{p-1}\frac pq \int_{\R^n}\int_{\R^n} \kappa_\d(x-y)\frac{|f(x)-f(y)|^q}{|x-y|^q}dydx
+ \frac1{\alpha^{r/p}}k^{p-1} \frac pr \int_{\R^n}\int_{\R^n} \kappa_\d(x-y) |h(y)|^r dydx
 \\
&\leq {\alpha^{q/p}}k^{p-1}\frac pq \,k \,\|\nabla f\|_q^q 
   + \frac1{\alpha^{r/p}}k^{p-1}\,\frac pr \, k\, \|h\|_r^r.
\end{align*}
By optimizing over $\alpha$ we find 
\begin{align*}
\|\mathrm{I}\|^p_{L^p} & \leq k^p\|\nabla f\|_{L^q}^p \| h\|_{L^r}^p= \|\kappa_{\d}\|^p_{L^1}\|\nabla f\|_{L^q}^p \| h\|_{L^r}^p.
\end{align*}
Combining these estimates and using 
\begin{align}\label{Aux-eq:Reg-Seq-Ind-L1}
\int_{\R^n}|z| |\nabla\nu_\d(z)|dz = \delta^{-n}\int_{\R^n}\frac{|z|}{\delta} |\nabla\nu|\left(\frac z\delta\right)dz = \int_{\R^n}|\tilde z| |\nabla\nu(\tilde z)|d\tilde z 
\end{align}
we obtain the claimed result. 
\end{proof}

\begin{lemma}\label{FIR-lem:VFP-Density}
Let $Y$ be the space defined in \eqref{FIR-def:Y-space}. Then $C_c^\infty([0,T]\times\R^{2d})$ is dense in $Y$. 
\end{lemma}
\begin{proof}
We prove this lemma in two steps. In the first step we approximate functions in $Y$ by spatially compactly supported functions. In the second step we approximate functions in $Y$ with spatially compact support  by smooth functions. 

In both steps we construct an approximating sequence that converges strongly in $X$ and weakly in $X'$; it then follows from Mazur's lemma that a convex combination of this sequence converges strongly in both $X$ and $X'$, and therefore in $Y$.

\textit{Step 1.} For an arbitrary $g\in Y$,  define $g_R(t,x):=g(t,x)\rchi_R(\sqrt{H(x)})$, where $\rchi_R\in C_c^\infty(\R;\R)$ is given by
\begin{align}\label{FIR-def:Smooth-Char-Fn}
\rchi_R(x)=\begin{cases}1, &|x|\leq R \\ 0, & |x|>2R\end{cases},  \ \ \text{ with }  \ \ \|\nabla\rchi_R\|_{L^\infty}\leq \frac CR.
\end{align}
Note that $g_R$ is compactly supported in $\R^{2d}$.
Using the dominated convergence theorem we find
\begin{align*}
\|g_R-g\|^2_{X}=\int_0^T\int_{\R^{2d}}\Bigl[ (1-\rchi_R)^2(g^2+|\nabla_p g|^2) + g^2 |\nabla_p\rchi_R|^2\Bigr]e^{-H} \xrightarrow{R\rightarrow\infty} 0.
\end{align*}
Here we have used $|\nabla H|^2\leq C(1+H)$ and the estimate 
\begin{align*}
|\nabla_p \rchi_R|^2=(\rchi'_R(\sqrt{H}))^2\frac{1}{4H}|\nabla_pH|^2\leq C.
\end{align*}

To conclude the first part of this proof we need to show that 
\begin{align}\label{Aux-eq:Dense-Compact-Step}
\langle \partial_tg_R -J\nabla H\cdot\nabla g_R,\phi\rangle_{X',X} \xrightarrow{R\rightarrow\infty} \langle \partial_tg -J\nabla H\cdot\nabla g,\phi\rangle_{X',X}, \ \ \forall\phi\in X.
\end{align}
Let $\phi\in C_c^\infty((0,T)\times\R^{2d})$. Then
\begin{align*}
\bigl|\langle \partial_t g_R &-J\nabla H\cdot\nabla g_R,\phi\rangle_{X',X}\bigr|=
\left|\int_0^T\int_{\R^{2d}} g_R(\partial_t\phi- J\nabla H\cdot\nabla\phi)e^{-H}\right|\\
&\leq  \left|\int_0^T\int_{\R^{2d}} g\Bigl[\partial_t(\phi(\rchi_R\circ\sqrt{H}))- J\nabla H\cdot\nabla((\rchi_R\circ\sqrt{H})\phi)\Bigr]e^{-H}\right|
  + \left|\int_0^T\int_{\R^{2d}} g\phi J\nabla H\cdot\nabla(\rchi_R\circ\sqrt{H})e^{-H}\right|\\
& \leq  \|\partial_t g -J\nabla H\cdot\nabla g\|_{X'}\|\phi\|_{X}, 
\end{align*}
where we have used $J\nabla H\cdot\nabla(\rchi_R\circ \sqrt{H})=0$ to arrive at the final inequality.  As a result  
\begin{align}\label{Aux-eq:Dense-Compact-Step-Dual-Bounds}
\|\partial_tg_R -J\nabla H\cdot\nabla g_R\|_{X'} \leq \|\partial_tg -J\nabla H\cdot\nabla g\|_{X'},
\end{align}
and using the dominated convergence theorem we find
\begin{align}\label{Aux-eq:Dense-Compact-Step-DCT}
\langle \partial_tg_R -J\nabla H\cdot\nabla g_R,\phi\rangle_{X',X} \xrightarrow{R\rightarrow\infty} \langle \partial_tg -J\nabla H\cdot\nabla g,\phi\rangle_{X',X}, \ \ \forall\phi\in C_c^\infty((0,T)\times\R^{2d}).
\end{align}
Estimate~\eqref{Aux-eq:Dense-Compact-Step-Dual-Bounds} together with the convergence statement~\eqref{Aux-eq:Dense-Compact-Step-DCT} implies that~\eqref{Aux-eq:Dense-Compact-Step} holds. As mentioned above, Mazur's lemma then gives the existence of a sequence that converges strongly in $Y$.
 
\textit{Step 2. } In this step we approximate spatially compactly supported functions $g\in Y$ by smooth functions. Using a partition of unity (in time), it is sufficient to consider 
\begin{align*}
\mathcal A:= \{g\in Y: g \text{ has compact support in }[0,T)\times\R^{2d}\}.
\end{align*}
We will show that these functions can be approximated by functions in $C_c^\infty([0,T)\times\R^{2d})$. 

For any $g\in \mathcal A$, we define its translation to the left in time over $\tau>0$ as $g_\tau(t, x):=g(t+\tau, x)$. Furthermore define $g_{\tau,\d}=\nu_\d\ast g_\tau$, where $\nu_\d$ is a symmetric regularising sequence in $\R\times\R^{2d}$. Note that $g_{\tau,\d}\in C_c^\infty([0,T)\times\R^{2d})$ when $\delta$ is small enough. Using standard results it follows that $g_{\tau,\d}\rightarrow g$ as $\tau,\d\rightarrow 0$ in $X$. We will now show that 
\begin{align}\label{Aux-eq:Dense-Left-Tra-Bound}
|\langle \partial_t g_{\tau,\d}-J\nabla H\cdot \nabla g_{\tau,\d},\phi\rangle_{X',X}| \leq C\|g\|_
X\|\phi\|_X+ \|\partial_t f-J\nabla H\cdot\nabla g\|_{X'}\|\phi\|_X,
\end{align}
where $C$ is independent of $\tau$ and $\d$ and of the test function $\phi$. For any $\phi\in C_c^\infty((0,T)\times\R^{2d})$,  
\begin{align}
&\langle \partial_tg_{\tau,\d} -J\nabla H\cdot\nabla g_{\tau,\d},\phi\rangle_{X',X} = -\int_0^T\int_{\R^{2d}} (\nu_\d\ast g_{\tau}) (\partial_t\phi - J\nabla H\cdot\nabla\phi)e^{-H}\nonumber\\ 
&= -\int_0^T\int_{\R^{2d}} g_\tau\Bigl[\nu_\d\ast (\partial_t\phi \,e^{-H}) + \nu_\d\ast(J\nabla e^{-H}\cdot\nabla\phi)\Bigr]\\ 
&= -\int_0^T\int_{\R^{2d}} g_\tau \bigl[\partial_t (\nu_\d\ast\phi) -J\nabla H\cdot(\nu_\d\ast\nabla\phi)\bigr]e^{-H} \nonumber\\
& \  - \int_0^T\int_{\R^{2d}} g_\tau \Bigl[ \nu_\d\ast(\partial_t\phi\, e^{-H}) - (\nu_\d\ast\partial_t\phi)e^{-H}\Bigr] - \int_0^T\int_{\R^{2d}} g_\tau\Bigl[\nu_\d\ast (J\nabla e^{-H}\cdot\nabla\phi)- J\nabla e^{-H}\cdot(\nu_\d\ast\nabla\phi)\Bigr].\label{FIR-eq:VFP-Reg-Lem}
\end{align}
We now estimate each term in the right hand side of~\eqref{FIR-eq:VFP-Reg-Lem}. For the first term, extending the time integral to $\R$ and using a change of variables we find
\begin{align*}
\biggl|\int_\R\int_{\R^{2d}} g( t+\tau,x) &\Bigl(\partial_t (\nu_\d\ast\phi) -J\nabla H\cdot(\nu_\d\ast\nabla\phi)\Bigr)(t,x)e^{-H(x)}dxdt\biggr|\\
&=\biggl|\int_\R\int_{\R^{2d}} g(s,x) \Bigl(\partial_t (\nu_\d\ast\phi) -J\nabla H\cdot(\nu_\d\ast\nabla\phi)\Bigr)(s-\tau,x)e^{-H(x)}dxds\Biggr| \\
&=\biggl|\int_\R\int_{\R^{2d}} g(s,x) \Bigl(\partial_t (\eta \nu_\d\ast\phi) -J\nabla H\cdot(\nu_\d\ast\nabla\phi)\eta\Bigr)(s-\tau,x)e^{-H(x)}dxds\biggr|\\
&\leq \|\partial_t g-J\nabla H\cdot \nabla g\|_{X'}\|\phi_\d(\cdot-\tau,\cdot)\eta\|_X.
\end{align*}
Here $\eta\in C_c([0,T))$ is any smooth function satisfying $0\leq \eta\leq 1$ and $\eta(t) = 1$ for $t\in \supp_t g$,
and the final inequality follows by the definition of $Y$ and  $\phi_\d(\cdot-\tau,\cdot)\eta\in C_c^\infty((0,T)\times\R^{2d})$. Using $\eta\leq 1$ and a change of variable we obtain
\begin{align*}
\|\phi_\d(\cdot-\tau,\cdot)\eta\|^2_X \leq \int_0^T \int_{\R^{2d}} \Bigl[ |\phi_\d(t-\tau,x)|^2+ |\nabla_p\phi_\d(t-\tau,x)|^2\Bigr]e^{-H(x)}dxdt\leq  \|\phi\|^2_X,
\end{align*}
and therefore for the for first term on the right hand side of~\eqref{FIR-eq:VFP-Reg-Lem} we have
\begin{equation*}
\left|\int_0^T\int_{\R^{2d}} g(t-\tau, x) \Bigl(\partial_t (\nu_\d\ast\phi) -J\nabla H\cdot(\nu_\d\ast\nabla\phi)\Bigr)(x,t)e^{-H(x)}dxdt \right| \leq \|\partial_t g-J\nabla H\cdot\nabla g\|_{X'}\|\phi\|_X.
\end{equation*}

For the final term in the right hand side of~\eqref{FIR-eq:VFP-Reg-Lem}, using $\div(J\nabla e^{-H})=0$ and applying Lemma~\ref{FIR-lem:Aux-Conv-Bound} with $f=J\nabla e^{-H}$, $h=\phi$ and $r=p=2$, $q=\infty$,  we find 
\begin{align*}
\biggl|\int_0^T\int_{\R^{2d}} g_\tau &\Bigl(\nu_\d\ast (J\nabla e^{-H}\cdot\nabla\phi)- J\nabla e^{-H}\cdot(\nu_\d\ast\nabla\phi)\Bigr)\biggr|\\
&  \leq \|g_\tau\|_{L^2(S)} \left\|\nu_\d\ast (J\nabla e^{-H}\cdot\nabla\phi)- J\nabla e^{-H}\cdot(\nu_\d\ast\nabla\phi)\right\|_{L^2( S)} \\
&  \leq \|g\|_{L^2(S)}\, \left\|D^2 e^{-H}\right\|_{L^\infty(\R^{2d})}\, \|\phi\|_{L^2( S)}\Bigl(\int_{ S} |z| |\nabla \nu(z)|dzdt\Bigr)
 \leq \frac{C}{\alpha} \|g\|_X\|\phi\|_X.
\end{align*}
Here $S:=\supp g$, $D^2e^{-H}$ is the Hessian of $e^{-H}$ and $\alpha:=\inf_{x\in S}e^{-H(x)}>0$. Repeating a similar calculation for the second term on the right hand side of~\eqref{FIR-eq:VFP-Reg-Lem}, we find
\begin{align*}
 \int_0^T\int_{\R^{2d}} g_\tau \Bigl[\nu_\d\ast(\partial_t\phi\, e^{-H}) - (\nu_\d\ast\partial_t\phi)e^{-H}\Bigr] \leq C\|g\|_X\|\phi\|_X.
 \end{align*}
Combining all the terms we find~\eqref{Aux-eq:Dense-Left-Tra-Bound}. As a result, $\|\partial_tg_{\tau,\d} -J\nabla H\cdot\nabla g_{\tau,\d}\|_{X'}$ is bounded independently of~$\tau$ and~$\d$. Using the dominated convergence theorem we also have for all $\phi\in C_c^\infty((0,T)\times \R^{2d})$
\begin{align*}
\forall \tau>0: \ &\langle \partial_tg_{\tau,\d} -J\nabla H\cdot\nabla g_{\tau,\d},\phi\rangle_{X',X} \xrightarrow{\d\rightarrow 0} \langle \partial_tg_\tau -J\nabla H\cdot\nabla g_\tau,\phi\rangle_{X',X},\quad\text{and}\\ 
&\langle \partial_tg_{\tau} -J\nabla H\cdot\nabla g_{\tau},\phi\rangle_{X',X} \xrightarrow{\tau\rightarrow 0} \langle \partial_tg -J\nabla H\cdot\nabla g,\phi\rangle_{X',X}
\end{align*}

Taking two sequences $\tau_n\to0$ and $\delta_n\to0$ such that the translation and convolution operations above are allowed, we use the boundedness of $\partial_tg_{\tau,\d} -J\nabla H\cdot\nabla g_{\tau,\d}$ in the separable space $X'$ to extract a subsequence that converges in the weak-star topology; we then use the density of $C_c^\infty((0,T)\times \R^{2d})$ in $X$ and the convergence of $g_{\tau,\delta}$ to identify the limit. 
Again using Mazur's lemma it follows that there exists a strongly converging sequence. This concludes the proof of the lemma. 

\end{proof}

\begin{lemma}\label{FIR-thm:VFP-Green}
Let $g\in Y$. Then $g$ admits (continuous) time trace values in $L^2(e^{-H})$.
Furthermore, for any $g,\tilde g \in Y$ we have
\begin{align}\label{FIR-eq:VFP-Green}
\langle \partial_t g- J\nabla H\cdot\nabla g, \tilde g\rangle _{X',X} + \langle \partial_t \tilde g- J\nabla H\cdot\nabla \tilde g,  g\rangle _{X',X}  = \int_{\R^{2d}} g\tilde g \, e^{-H} \Bigl|_{t=0}^{t=T}.
\end{align}
\end{lemma}
\begin{proof}
We will prove that the mapping
\begin{align*}
C_c^\infty([0,T]\times\R^{2d})\ni g \mapsto (g(0), g(T))\in L^2(e^{-H})\times L^2(e^{-H}),
\end{align*}
can be continuously extended to $Y$. This implies that any $f\in Y$ admits trace values in $L^2(e^{-H})$ since  $C_c^\infty([0,T]\times\R^{2d})$ is dense in $Y$ by Lemma~\ref{FIR-lem:VFP-Density}. The proof of~\eqref{FIR-eq:VFP-Green} follows by applying integration by parts to smooth functions and then passing to the limit in $Y$.

Consider $\eta\in C^\infty([0,T])$ with $0\leq \eta\leq 1$, $\eta(t)=1$ for $t\in [0, T/3]$, and $\eta(t)=0$ for $t\in [2T/3,T]$. We have for any $g\in C_c^\infty([0,T]\times\R^{2d})$
\begin{align}\label{FIR-eq:Trace-Val-Lem}
\|g|_{t=0}&\|^2_{L^2(e^{-H})} =\int_{\R^{2d}} g^2\big|_{t=0} e^{-H} =\int_{\R^{2d}} g^2\eta^2\big|_{t=0} e^{-H} 
=-2\int_0^T\int_{\R^{2d}} g\eta \, \partial_t(g\eta) e^{-H} \nonumber\\ 
&=-2\int_0^T\int_{\R^{2d}} g\eta ( \partial_t(g\eta) -  J\nabla H\cdot\nabla(g\eta)) e^{-H} -2\int_0^T\int_{\R^{2d}}  g\eta \, J\nabla H\cdot\nabla(g\eta) e^{-H}\nonumber\\
&=2\langle (\partial_t - J\nabla H \cdot\nabla)(g\eta),g\eta\rangle_{X',X}
+\int_0^T\int_{\R^{2d}}  J\nabla e^{-H}\cdot\nabla(g^2\eta^2) \nonumber\\
&= 2\langle (\partial_t - J\nabla H \cdot\nabla)(g\eta),g\eta\rangle_{X',X} 
\leq 2\|(\partial_t - J\nabla H \cdot\nabla)(g\eta)\|_{X'}\|g\eta\|_{X},
\end{align}
where the final equality follows by the anti-symmetry of $J$. Note that $\|g\eta\|_{X}\leq \|g\|_X$. Furthermore
\begin{align*}
\|(\partial_t - J\nabla H \cdot\nabla)(g\eta)\|_{X'} 
&= \sup\limits_{\substack{\phi\in C_c^\infty((0,T)\times \R^{2d})\\ \|\phi\|_X=1}} \int_0^T\int_{\R^{2d}} g\eta (\partial_t\phi-J\nabla H\cdot\nabla\phi)e^{-H} \\
&=\sup\limits_{\phi} \int_0^T\int_{\R^{2d}} g(\partial_t(\phi\eta)-J\nabla H\nabla\phi\eta)e^{-H} - \int_0^T\int_{\R^{2d}} g\phi\, \partial_t\eta e^{-H}\\
&\leq  \|\partial_t g- J\nabla H\cdot\nabla g\|_{X'}+ \|\partial_t \eta\|_\infty \|g\|_X \leq  C\|g\|_Y.
\end{align*}
Substituting back into~\eqref{FIR-eq:Trace-Val-Lem} we find
\begin{align*}
\|g|_{t=0}\|^2_{L^2(e^{-H})} \leq C\|g\|_Y,
\end{align*} 
which completes the proof for the initial time. The proof for the final time proceeds similarly.
\end{proof}

Now we are ready to prove Theorem~\ref{FIR-thm:AuxPDE-Well-Posed}. We will make use of a result of Lions~\cite{Lions61}, which we  state here  for convenience. 
\begin{theorem}\label{FIR-thm:Lions-Deg} Let $F$ be a Hilbert space, equipped with a norm $\|\cdot\|_{F}$ and an inner product $(\cdot,\cdot)$. Let $\Theta$ be a subspace of $F$, provided with a prehilbertian norm $\|\cdot\|_\Theta$, such that the injection $\Theta\hookrightarrow F$ is continuous. Consider a bilinear form $E$:
\begin{align*}
E:F\times\Theta \ni (g,\phi) \mapsto E(g,\phi)\in \R
\end{align*}
 such that $E(\cdot,\phi)$ is continuous on $F$ for any fixed $\phi\in \Theta$, and such that 
 \begin{align}\label{FIR-eq:Lions-Coercive}
 |E(\phi,\phi)|\geq \alpha \|\phi\|^2_\Theta, \quad \forall \phi\in \Theta, \text{ with }\alpha>0.
 \end{align}
Then, given a continuous linear form $L$ on $\Theta$, there exists a solution $g$ in $F$ of the problem
\begin{align*}
E(g,\phi)=L(\phi), \quad \forall \phi\in \Theta. 
\end{align*}
\end{theorem}

\begin{proof}[Proof of Theorem~\ref{FIR-thm:AuxPDE-Well-Posed}] 
We will use Theorem~\ref{FIR-thm:Lions-Deg} to show the existence of a solution to the variational equation~\eqref{FIR-eq:PDE-Weak-Form}. We choose $F=X$ and $\Theta=C_c^\infty([0,T)\times\R^{2d})$ with 
\begin{align*}
\|\phi\|_{\Theta}^2 =\|\phi\|^2_X + \frac 12 \|\phi|_{t=0}\|^2_{L^2(e^{-H})}, \quad \forall\phi\in \Theta.
\end{align*}
By definition $\Theta\hookrightarrow X$. 

The bilinear form $E_\lambda$ defined in \eqref{FIR-eq:PDE-Weak-Form} satisfies property~\eqref{FIR-eq:Lions-Coercive}, since
\begin{align*}
E_\lambda(\phi,\phi)&=\int_0^T\int_{\R^{2d}} \Bigl\{ -\frac 12 \partial_t\phi^2 + \frac 14 J\nabla\psi\ast\rho_t\cdot\nabla \phi^2 + \Bigl( \lambda+\frac 12 \Psi\Bigr) \phi^2-\frac 14 J\nabla\psi\ast\rho_t\cdot\nabla\phi^2 +|\nabla_p\phi|^2 \Bigr\}\,e^{-H} \\
&\geq  \frac 12 \|\phi|_{t=0}\|^2_{L^2(e^{-H})} + \min\Bigl\{ 1, \lambda - \frac 12 \|\Psi\|_{L^\infty}\Bigr\} \|\phi\|_X^2 \geq  \|\phi\|^2_{\Theta}, 
\end{align*}
where we have used~\eqref{FIR-eq:VFP-Lambda}. 

Since all the conditions of Theorem~\ref{FIR-thm:Lions-Deg} are satisfied, the variational equation~\eqref{FIR-eq:PDE-Weak-Form}  
admits a solution $g$ in $X$. We have 
\begin{align*}
&\int_0^T\int_{\R^{2d}} g (\partial_t\phi - J\nabla H\cdot\nabla \phi)e^{-H} \\
&= \int_0^T\int_{\R^{2d}} g \, \Bigl\{\frac 12 J\nabla\psi\ast\rho_t\cdot\nabla \phi +\Bigl(\lambda+\frac 12 \Psi\Bigr)\phi 
 - \frac 12 \phi J\nabla\psi\ast\rho_t\cdot\nabla g + \nabla_pg\cdot\nabla_p\phi\Bigr\}e^{-H} + L_\lambda(\phi) 
\leq C\|g\|_X\|\phi\|_X,
\end{align*}
where we have used $ J\nabla\psi\ast\rho_t\cdot\nabla \phi=-\nabla_q\psi\ast\rho_t\cdot\nabla_p \phi $. Note that $C>0$ is independent of $\phi$, and therefore the solution $g$ belongs to $Y$.

Next we show that $g^0$ appearing in the definition of $L_\lambda$ in \eqref{FIR-def:VFP-L-Form} is the initial value for the solution $g$ of~\eqref{FIR-eq:PDE-Weak-Form}. 
Choose $\phi(t,x)=\hat\phi(x)\bar\phi_\vep(t)$, where $\hat\phi\in C_c^\infty(\R^{2d})$ and the sequence $\bar\phi_\vep$ satisfies $\bar\phi_\vep(0)=1$, $\bar\phi_\vep(t)\rightarrow 0$ for any $t\in(0,T)$ and $\bar\phi_\vep'\rightarrow -\delta_0$ (Dirac delta at $t=0$). Substituting $\phi$ in~\eqref{FIR-eq:PDE-Weak-Form} we find
\begin{align}
\label{FIR-eq:VFP-Trace-Value-Init-Value}
-\int_{0}^T\int_{\R^{2d}} g\hat\phi(x) \bar\phi'_\vep(t) e^{-H} = \int_{\R^{2d}} g^0 \hat\phi(x) e^{-H} +o(1)
\end{align}
as $\vep\to 0$.
By Lemma~\ref{FIR-thm:VFP-Green}, $g$ admits trace values in $L^2(\R^{2d};e^{-H})$, and therefore passing $\vep\rightarrow 0$ in~\eqref{FIR-eq:VFP-Trace-Value-Init-Value} we find
\begin{align*}
\int_{\R^{2d}}\left[g(0,x)-g^0(x)\right]\hat\phi(x) e^{-H}dx  = 0 , \quad \forall \hat\phi\in C_c^\infty(\R^{2d}).
\end{align*}

Finally we prove the uniqueness  in $Y$ of the solution of~\eqref{FIR-eq:PDE-Weak-Form}. Consider two solutions $g_1,g_2\in Y$ and let $g=g_1-g_2$. Since the initial data and the right-hand side $U$ in~\eqref{FIR-eq:PDE-Weak-Form} vanish,  we have $E_\lambda(g,\phi)=0$ for all $\phi\in C_c^\infty([0,T)\times \R^{2d})$. Taking a sequence $\phi_n\in C_c^\infty([0,T)\times \R^{2d})$ that converges in $X$ to $g$, we find
\begin{align*}
0&= \lim_{n\to\infty} E_\lambda(g,\phi_n)\\
&= \lim_{n\to\infty} \langle \partial_t g - J \nabla H\cdot \nabla g,\phi_n\rangle_{X',X}+
\\
&\qquad + \int_0^T\int_{\R^{2d}}  \, \Bigl\{g\left(\frac 12 J\nabla\psi\ast\rho_t\cdot\nabla \phi_n +\Bigl(\lambda+\frac 12 \Psi\Bigr)\phi_n\right) 
 - \frac 12 \phi_n J\nabla\psi\ast\rho_t\cdot\nabla g + \nabla_pg\cdot\nabla_p\phi_n\Bigr\}e^{-H}\\
&= \langle \partial_t g - J \nabla H\cdot \nabla g,g\rangle_{X',X} +\int_0^T\int_{\R^{2d}} \Bigl\{ \Bigl(\lambda+\frac 12 \Psi\Bigr)g^2+ |\nabla_pg|^2  \Bigr\}e^{-H}
\stackrel{\eqref{FIR-eq:VFP-Green}}\geq \frac 12 \int_{\R^{2d}} g^2|^{}_{t=T}\,e^{-H} + \|g\|^2_X \geq 0.
\end{align*}
This proves uniqueness.
\end{proof}

\begin{remark}
\label{FIR-rem:general-phi}
Using the same technique as in the uniqueness proof above we can prove the following result.
If $g\in Y$ satisfies $E_\lambda(g,\phi) = L_\lambda(\phi)$ for all $\phi\in C_c^\infty([0,T)\times \R^{2d})$, then for all  $\phi\in C([0,T]\times \R^{2d})$ we have 
\[
E_\lambda(g,\phi) = L_\lambda(\phi) - \int_{\R^{2d}} g\phi\big|_{t=T} \, e^{-H}
= \langle e^{-\lambda t}U,\phi\rangle_{X',X} - \int_{\R^{2d}}g\phi\big|_{t=0}^{t=T} \, e^{-H}.
\]
\end{remark}

Theorem~\ref{FIR-thm:AuxPDE-Well-Posed} proves the well-posedness of the variational equation~\eqref{FIR-eq:PDE-Weak-Form} which is a weak form for the time-rescaled equation~\eqref{FIR-eq:VFP-Trans-Aux}. Transforming back, we also conclude the well-posedness of the variational equation corresponding to the original equation~\eqref{FIR-eq:VFP-Main-Aux-PDE}. We state this in the following corollary.

\begin{corollary}\label{Aux-Cor:WellPose-Original}
Assume that 
\begin{align*}
\Psi\in C^2_c(\R\times\R^{2d}), \ \ U\in X', \  \text{ and } \  g^0\in L^2(\R^{2d};e^{-H}). 
\end{align*}
Then there exists a unique solution $g$ to the variational equation
\begin{align}\label{FIR-eq:PDE-Weak-Form-Lambda=0}
E(g,\phi) = L(\phi), \quad \forall \phi\in C_c^\infty([0,T)\times\R^{2d}),
\end{align}
in the class of functions $Y$. Here $E:X\times C_c^\infty([0,T)\times\R^{2d})\rightarrow\R$ and $F:C_c^\infty([0,T)\times\R^{2d})\rightarrow\R$ are given by
\begin{align}\label{FIR-eq:E-L=0}
&E(g,\phi):= \int_0^T\int_{\R^{2d}} \Bigl\{ g \Bigl( -\partial_t\phi + J\nabla H \cdot \nabla\phi +\frac{1}{2} J \nabla\psi\ast\rho_t\cdot \nabla \phi  +  \frac 12 \Psi\phi  \Bigr)-\frac 12 \phi\, J\nabla\psi\ast\rho_t\cdot\nabla g + \nabla_p g\cdot\nabla_p \phi
\Bigr\} \, e^{-H},\\
& L(\phi):=\langle U,\phi\rangle_{X',X} + \int_{\R^{2d}}g^0\phi\big|_{t=0} \, e^{-H}. 
\end{align}
\end{corollary}


\subsection{Bounds and Regularity Properties}\label{Aux-Sec:PDE-Prop}

Having discussed the well-posedness of equation~\eqref{FIR-eq:VFP-Main-Aux-PDE}, in this section we derive some properties of its solution. These properties play an important role in the proof of Theorem~\ref{FIR-thm:VFP-Aux-PDE-Res}. 

\subsubsection{Comparison principle and growth at infinity}
We first provide an auxiliary lemma which we require to prove the comparison principle.

\begin{lemma}\label{FIR-lem:VFP-Comp-Aux-Lem}
For $g\in Y$, define $g^-\in X$ by $g^-:=\max\{-g,0\}$.
Then 
\begin{align}\label{FIR-eq:VFP-Max-Min}
\langle \partial_tg-J\nabla H\cdot\nabla g, g^-\rangle_{X',X}  = -\frac 12\int_{\R^{2d}} (g^-)^2\Bigr|_{t=0}^{t=T} e^{-H}.
\end{align}
\end{lemma}
\begin{proof}
Since $C_c^\infty([0,T]\times \R^{2d})$ is dense in $Y$ by Lemma~\ref{FIR-lem:VFP-Density}, it is sufficient to prove~\eqref{FIR-eq:VFP-Max-Min} for $g\in C_c^\infty([0,T]\times\R^{2d})$. For $g\in C_c^\infty([0,T]\times\R^{2d})$, $g^-\in X\cap \mathrm{Lip}(\R^{2d})$ and there exists a sequence $\phi_n\in C_c^\infty([0,T]\times\R^{2d})$ such that $\phi_n\rightarrow g^-$ in $X$. We have
\begin{align*}
\langle \partial_tg -J\nabla H \cdot \nabla g, g^-\rangle_{X',X} 
&= \lim_{n\to\infty} \langle \partial_tg -J\nabla H \cdot \nabla g, \phi_n\rangle_{X',X}\\
&= \lim\limits_{n\rightarrow \infty} \int_0^T\int_{\R^{2d}} \phi_n (\partial_tg-J\nabla H\cdot\nabla g)e^{-H} \\
&= \int_0^T\int_{\R^{2d}} g^- (\partial_tg-J\nabla H\cdot\nabla g)e^{-H} \\
&= -\int_0^T\int_{\R^{2d}} g^- (\partial_tg^--J\nabla H\cdot\nabla g^-)e^{-H}
\stackrel{\eqref{FIR-eq:VFP-Green}}=-\frac 12\int_{\R^{2d}} (g^-)^2\Bigl|_{t=0}^{t=T}e^{-H}.
\end{align*}
\end{proof}

We now prove the comparison principle. 
\begin{proposition}[Comparison principle] \label{FIR-prop:VFP-Comp-Princ}
Let $g$ be the solution given by Corollary~\ref{Aux-Cor:WellPose-Original}.  Then
\begin{enumerate}
\item $g^0\geq 0$  \ and  \ $U\geq 0$ \  $\Longrightarrow$  \ $g\geq 0$.
\item $g^0\in L^\infty(\R^{2d})$  \ and  \ $U\in L^1(0,T; L^\infty(\R^{2d}))$ \  $\Longrightarrow$  \ $g\in L^\infty([0,T]\times \R^{2d})$ with
\begin{align*}
\|g(t)\|_{L^\infty} \leq \|g^0\|_{L^\infty} +\int_0^t\|U(s)\|_{L^\infty} ds.
\end{align*}
\end{enumerate}
\end{proposition}
\begin{proof}
Let  $g$ be the solution of the transformed variational equation~\eqref{FIR-eq:PDE-Weak-Form} provided by Theorem~\ref{FIR-thm:AuxPDE-Well-Posed},
which reads explicitly
\begin{align*}
0&= \langle\partial_tg-J\nabla H\cdot\nabla g,\phi\rangle_{X',X} - \langle e^{-\lambda t}U,\phi\rangle_{X',X} -\int_{\R^{2d}}g^0\phi_0\,e^{-H}\\
& + \int_0^T\int_{\R^{2d}} \Bigl\{ g \Bigl( \frac{1}{2} J \nabla\psi\ast\rho_t\cdot \nabla \phi  +  \Bigl( \lambda + \frac 12 \Psi \Bigr)\phi  \Bigr)  
-\frac 12 \phi\, J\nabla\psi\ast\rho_t\cdot\nabla g + \nabla_p g\cdot\nabla_p \phi
\Bigr\} \, e^{-H}.
\end{align*}
Consider a sequence $\phi_n\to g^-$ in $X$ as $n\rightarrow \infty$, with $\phi_n\geq0$. Then by the assumptions on $U$ and $g^0$ we have
\[
\langle e^{-\lambda t}U,\phi_n\rangle_{X',X} +\int_{\R^{2d}}g^0\phi_n|^{}_{t=0}\,e^{-H} \geq0,
\]
and therefore
\begin{align*}
0&\leq \lim_{n\to\infty} \langle\partial_tg-J\nabla H\cdot\nabla g,\phi_n\rangle_{X',X}\\
& \quad  + \int_0^T\int_{\R^{2d}} \Bigl\{ g \Bigl( \frac{1}{2} J \nabla\psi\ast\rho_t\cdot \nabla \phi_n  +  \Bigl( \lambda + \frac 12 \Psi \Bigr)\phi_n  \Bigr)  
-\frac 12 \phi_n\, J\nabla\psi\ast\rho_t\cdot\nabla g + \nabla_p g\cdot\nabla_p \phi_n
\Bigr\} \, e^{-H}\\
&= \langle\partial_tg-J\nabla H\cdot\nabla g,g^-\rangle_{X',X}\\
&\quad +\int_0^T\int_{\R^{2d}} \Bigl\{ g \Bigl( \frac{1}{2} J \nabla\psi\ast\rho_t\cdot \nabla g^-    +  \Bigl( \lambda + \frac 12 \Psi \Bigr)g^-  \Bigr)  
-\frac 12 g^-\, J\nabla\psi\ast\rho_t\cdot\nabla g + \nabla_p g\cdot\nabla_p g^-
\Bigr\} \, e^{-H}\\
&= -\frac 12\int_{\R^{2d}} (g^-)^2\Bigr|_{t=0}^{t=T} e^{-H} 
- \int_0^T\int_{\R^{2d}}\Bigl\{\Bigl(\lambda+\frac 12 \Psi\Bigr)|g^-|^2 +|\nabla_pg^-|^2\Bigr\}e^{-H},
\end{align*}
where  the last equality follows by Lemma~\ref{FIR-lem:VFP-Comp-Aux-Lem}. 
Since $g^-|_{t=0}=0$ and $\lambda \geq \frac12 \|\Psi\|_\infty+1$ by assumption~\eqref{FIR-eq:VFP-Lambda}, this implies that $g^-=0$.

This completes the proof of the first part of Proposition~\ref{FIR-prop:VFP-Comp-Princ}. The second part is a simple consequence of the first part, by applying the first part to the function $\tilde g\in Y$, $\tilde g(t) := \|g^0\|_\infty + \int_0^t \|U(s)\|_{L^\infty}\, ds - g(t)$, which satisfies an equation of the same form.
\end{proof}

In the next result we use the comparison principle to prove explicit bound on the solution of equation~\eqref{FIR-eq:VFP-Main-Aux-PDE} when $U=0$.
\begin{proposition}[Growth] \label{FIR-prop:VFP-Aux-PDE-Control} Assume that $\inf H=0$ and  $0 < \alpha_1 \leq g^0\leq \alpha_2<\infty$. The the solution for the variational problem~\eqref{FIR-eq:PDE-Weak-Form-Lambda=0} with $U=0$ satisfies
\begin{align*}
\alpha_1\exp\left(-\beta_1 t\sqrt{\omega_1+ H}\right)\leq g \leq \alpha_2\exp\left(\beta_2 t\sqrt{\omega_2+ H}\right) 
\end{align*}
for some fixed constants $\beta_1,\beta_2,\omega_1,\omega_2>0$.
\end{proposition}
 \begin{proof} 
We first prove the second inequality in Proposition \ref{FIR-prop:VFP-Aux-PDE-Control}. For some constants $\beta_2>0,\omega_2>1$ to be specified later, we define $g_2:=\alpha_2\exp(\beta_2 t\sqrt{\omega_2+ H}) \in Y$, such that $g_2|_{t=0}=\alpha_2$.  
We will show that $g_2-g$ satisfies the assumptions of Proposition~\ref{FIR-prop:VFP-Comp-Princ}. 

Substituting $g_2-g$ in~\eqref{FIR-eq:E-L=0} and using the smoothness of $g_2$ we find 
\begin{align*}
E(g_2-g, \phi)=\langle U_2,\phi\rangle_{X',X}+\int_{\R^{2d}}\left(g_2|_{t=0}-g^0\right)\phi\, e^{-H}
\end{align*}
with 
\begin{align*}
U_2 = \partial_t g_2 - J\nabla H\cdot \nabla g_2 -J\nabla (\psi\ast\rho_t)\cdot\nabla g_2 + \nabla_p H\cdot\nabla_pg_2 - \Delta_p g_2 -\frac{g_2}{2}\left(J\nabla H\cdot\nabla\psi\ast\rho_t-\Psi\right).
\end{align*}
By construction $g_2|_{t=0}-g^0\geq 0$. We now show that $U_2\geq 0$. We calculate
\begin{align*}
\partial_t g_2 - J\nabla H\cdot \nabla g_2 -J\nabla (\psi\ast\rho_t)\cdot \nabla g_2 + \nabla_pH\cdot\nabla_pg_2 - \Delta_p g_2 -\frac{g_2}{2}\left(J\nabla H\cdot\nabla\psi\ast\rho_t-\Psi\right)\\
\geq g_2\Bigl( \frac{1}{2} \beta_2\sqrt{\omega_2+H} - \frac 12\left(J\nabla H\cdot\nabla\psi\ast\rho_t-\Psi\right) + \frac 12 \beta_2\sqrt{\omega_2 + H} - c\beta_2 t -\tilde c\beta_2^2t^2 \Bigr),
\end{align*}
where the constants $c,\tilde c$ are independent of $\beta_2$ and $\omega_2$, using the uniform bounds on $\Delta H$ and the bound  $|\nabla H|^2\leq C(1+H)$. Because of this growth condition on $\nabla H$, we can choose $\beta_2,\omega_2$ large enough such that
\begin{align*}
\frac{1}{2} \beta_2\sqrt{\omega_2+H} \geq  \frac 12\left(J\nabla H\cdot\nabla\psi\ast\rho_t-\Psi\right). 
\end{align*}
Then we choose $\omega_2$ even larger such that for any $t\in [0,T]$
\begin{align*}
\frac 12 \beta_2\sqrt{\omega_2+H}\geq \frac 12\beta_2\sqrt{\omega_2} \geq c\beta_2 t+\tilde c\beta_2^2t^2.
\end{align*}
For these values of $\beta_2,\omega_2$ we therefore have
\begin{align*}
U_2=\partial_t g_2 - J\nabla H\cdot \nabla g_2 -J\nabla (\psi\ast\rho_t)\cdot \nabla g_2 + \nabla_pH\cdot\nabla_pg_2 - \Delta_p g_2 -\frac{g_2}{2}\left(J\nabla H\cdot\nabla\psi\ast\rho_t-\Psi\right)\geq 0.
\end{align*}
Using the comparison principle of Lemma~\ref{FIR-prop:VFP-Comp-Princ} we then obtain
\begin{align*}
g\leq \alpha_2\exp\left(\beta_2 t\sqrt{\omega_2+ H}\right). 
\end{align*}

Proceeding similarly it also follows that $g_1:=\alpha_1\exp(-\beta_1 t\sqrt{\omega_1+H})$ is a subsolution for~\eqref{FIR-eq:VFP-Main-Aux-PDE} for appropriately chosen $\beta_1$ and $\omega_1$, and the first inequality in Proposition \eqref{FIR-prop:VFP-Aux-PDE-Control} follows.
\end{proof}

In the next result we make a specific choice for $\Psi$ (which corresponds to the Fisher Information for the VFP equation) and show that with this choice, the $L^2(e^{-H})$ norm of the solution of~\eqref{FIR-eq:VFP-Main-Aux-PDE} decreases in time. 
\begin{proposition}\label{FIR-prop:Aux-PDE-Decreasing}
The solution $g$ for the variational problem~\eqref{FIR-eq:PDE-Weak-Form-Lambda=0} (in the sense of Corollary~\ref{Aux-Cor:WellPose-Original}) with $U=0$ and 
\begin{align}\label{FIR-def:VFP-Aux-Psi-Explicit}
\Psi = -\Bigl(  \Delta_p\varphi -\nabla_p\varphi  \cdot\nabla_pH - \frac 12|\nabla_p\varphi|^2 \Bigr),
\end{align}
for some $\varphi\in C^\infty_c([0,T]\times\R^{2d})$, satisfies
\begin{align*}
\int_{\R^{2d}}g^2\Bigr|_0^T e^{-H}\leq 0.  
\end{align*} 
\end{proposition}
\begin{proof}

Let $g\in Y$ be the solution given by Corollary~\ref{Aux-Cor:WellPose-Original}.  Since $g\in X$, there exists a sequence $\phi_n\in C_c^\infty((0,T)\times\R^{2d})$ such that $\phi_n\rightarrow g$ in $X$. Furthermore $\partial_t g - J\nabla H\cdot \nabla g\in X'$ and we have
\begin{align*}
\langle \partial_t g - J\nabla H\cdot\nabla g,g\rangle_{X',X} = \lim\limits_{n\rightarrow\infty} \langle \partial_t g -J\nabla H\cdot \nabla g,\phi_n\rangle_{X',X}.   
\end{align*}
Using the same approximation arguments as in the proof of the comparison principle we find
\begin{align*}
\frac 12 \int_{\R^{2d}}g_t^2 e^{-H}\Bigl|_{t=0}^{t=T}&=\langle \partial_t g - J\nabla H\cdot\nabla g,g\rangle_{X',X} \\
&= \lim\limits_{n\rightarrow\infty}  \int_0^T\int_{\R^{2d}} \Bigl( \frac 12 \phi_n J\nabla\psi\ast \rho_t\cdot\nabla g - \frac 12 g J\nabla\psi\ast \rho_t\cdot\nabla\phi_n - \nabla _p g\nabla_p\phi_n -\frac 12 g\phi_n \Psi 
\Bigr)e^{-H}\\
& =\int_0^T\int_{\R^{2d}} \Bigl( -|\nabla_p g|^2-\frac 12 g^2\Psi\Bigr)e^{-H}.
\end{align*}
Using~Lemma~\ref{FIR-thm:VFP-Green} and substituting~\eqref{FIR-def:VFP-Aux-Psi-Explicit} into this relation we find
\begin{align*}
\frac 12 \int_{\R^{2d}} g^2 e^{-H}\Bigr|_0^T &= \int_0^T\int_{\R^{2d}} \Bigl( -|\nabla_p g|^2+\frac 12 g^2\Bigl[ \Delta_p\varphi -\nabla_p\varphi\cdot \nabla _p H-\frac 12 |\nabla_p\varphi|^2\Bigr] \Bigr)e^{-H} \\
&= -\int_0^T\int_{\R^{2d}} \Bigl( |\nabla_p g|^2+g\nabla_p\varphi\cdot\nabla_p g + \frac 14 g^2 |\nabla_p\varphi|^2\Bigr)e^{-H}\leq 0.
\end{align*}
where the second equality follows by applying integration by parts to the $\Delta_p\varphi$ term. This completes the proof. 
\end{proof}

\subsubsection{Regularity}
In this section we prove certain regularity properties for the solution of equation~\eqref{FIR-eq:VFP-Main-Aux-PDE}. We first present a general result regarding regularity of kinetic equations. This result is a combination of Theorem 1.5 and Theorem 1.6~\cite{Bouchut02}. The main difference is that we assume more control on the second derivative with respect to momentum, which also gives us a stronger regularity in the position variable. 

\begin{proposition}\label{Aux-Prop:Bouchut-Type-Res}
Assume that 
\begin{align}\label{Aux-eq:Reg-Abstract-PDE}
\partial_t f+p\cdot \nabla_q f -\sigma\Delta_p f=g \quad \text{in }\R\times\R^{2d}
\end{align}
holds with $\sigma>0$ and 
\begin{align*}
f,g\in L^2(\R\times\R^{2d}), \ \ \nabla_p f,\nabla_pg\in L^2(\R\times\R^{2d}).
\end{align*}
Then $\Delta_p f, \nabla_qf \in L^2(\R\times\R^{2d})$, $\partial_t f\in L^2_{\mathrm{loc}}(\R\times\R^{2d})$ and 
\begin{align*}
\|\nabla_q f\|_{L^2}\leq C\Bigl(\|\nabla_p g\|_{L^2}+ \|f\|_{L^2}\Bigr).
\end{align*}
\end{proposition}
\begin{proof}
From~\cite[Theorem 1.5]{Bouchut02} it follows that $\Delta_p f\in L^2(\R\times\R^{2d})$ with 
\begin{align*}
\sigma\|\Delta_p f\|_{L^2}\leq C_d \|g\|_{L^2},
\end{align*}
for a constant $C_d$ that only depends on the dimension $d$.
This implies that the Hessian in the $p$-variable satisfies $D^2_pf\in L^2(\R\times\R^{2d})$ as well. 

To prove the Proposition, we first assume that $f,g\in C_c^\infty([0,T]\times\R^{2d})$. We will later extend the results to the low-regularity situation via regularization arguments. 

Writing $(f,g) = \int_{\R^\times \R^{2d}} fg$ and using integration by parts we have
\begin{align}
\|\partial_{q_j} f\| _{L^2}^2 &=\left( \partial_{q_j} f, \partial_{q_j} f \right) =\left( \partial_{q_j} f, \partial_{p_j}(\partial_t+p\nabla_q)f - (\partial_t+p\nabla_q)\partial_{p_j}f \right) \notag\\
&=\left( \partial_{q_j} f, \partial_{p_j}(\partial_t+p\nabla_q)f  \right) + \left( \partial_{q_j}(\partial_t+p\nabla_q) f, \partial_{p_j}f \right)  
= 2\left( \partial_{q_j} \partial_{p_j} f, \sigma\Delta_p f  \right) + 2\left( \partial_{q_j} f, \partial_{p_j}g  \right)\notag\\
&\leq 0 + 2\|\partial_{q_j}f\|_{L^2}\|\partial_{p_j}g\|_{L^2} 
\label{Aux-eq:Part-Int-Reg-Calc}
\end{align}
Here we have used the (hypoelliptic) relation $\partial_{q_j}=\partial_{p_j}(\partial_t+p\nabla_q) - (\partial_t+p\nabla_q)\partial_{p_j}$ to arrive at the second equality. The final inequality follows since $f$ is real-valued, which implies that $|\hat f|^2$ is an even function and therefore 
\begin{align*}
\left( \partial_{q_j} \partial_{v_j} f_{\d,R}, \Delta_p f \right) = \int_{\R^{2d}} \zeta_j \eta_j |\eta|^2|\hat f|^2=0,
\end{align*}
where $\zeta,\eta$ are the Fourier variables corresponding to $q,p$.

Inequality~\eqref{Aux-eq:Part-Int-Reg-Calc} gives
\begin{align}\label{Aux-eq:Bouchut-Smoothed-Bounds}
\|\partial_{q_j} f\| _{L^2} \leq 2 \|\partial_{p_j}g\|_{L^2}.
\end{align}
Since $\nabla_q f, \Delta_p f,g\in L^2(\R\times\R^{2d})$, using~\eqref{Aux-eq:Reg-Abstract-PDE} we have $\partial_t f\in L^2_{\mathrm{loc}}(\R\times\R^{2d})$. This proves the result for smooth and compactly supported $f$ and $g$.

Let us now consider general $f,g\in L^2(\R\times\R^{2d})$ as in the Proposition, and define $f_\d:=\nu_\d\ast f$ and $g_\d:=\nu_\d\ast g$, where $\nu_\d$ is a regularizing sequence in $\R\times\R^{2d}$. Then we have
\begin{align*}
\partial_t f_\d + p\cdot \nabla_q f_\d - \Delta_p f_\d = g_\d + \bar g_\d,
\end{align*}
where $\bar g_\d = p\cdot \nabla_q f_\d - \nu_\d\ast(p\nabla_q f)$. Next we define $f_{\d,R}:=f_\d\rchi_R$ and $g_{\d,R}:=g_\d\rchi_R$, where 
\begin{align*}
 \rchi_R(x)=\rchi_1\Bigl(\frac x R\Bigr), \ \text{ where } \rchi_1\in C_c^\infty(\R^{2d}), \ \rchi_1(x)=1  \text{ for } |x|\leq 1, \ \rchi_1(x)=0 \text{ for } |x|\geq 2. 
\end{align*}
Then we have
\begin{align*}
\partial_t f_{\d,R} + p\cdot \nabla_q f_{\d,R} - \Delta_p f_{\d,R} = \left(g_\d + \bar g_\d\right)\rchi_R + \bar g_{\d,R}=: g_{\d,R}, 
\end{align*}
where
\begin{align}\label{Aux-eq:Cutoff-Error-Reg}
\bar g_{\d,R}=f_\d p\cdot \nabla_q\rchi_R-f_\d\Delta_p\rchi_R+\nabla_p f_\d\cdot\nabla_p\rchi_R. 
\end{align}
Note that $f_{\d,R},g_{\d,R}\in C_c^\infty(\R\times\R^{2d})$. To apply~\eqref{Aux-eq:Bouchut-Smoothed-Bounds} we need to show that $g_{\d,R},\nabla_p g_{\d,R}\in L^2(\R\times\R^{2d})$. In fact we will show that $g_{\d,R},\nabla_p g_{\d,R}$ are bounded in $L^2(\R\times\R^{2d})$ independently of $\d$ and $R$ with 
\begin{align}
\label{FIR-ineq:Bouchut-g}
\|\nabla_p g_{\d,R}\|_{L^2}\leq C\left(\|\nabla_p g\|_{L^2}+\|f\|_{L^2}+\|\nabla_p f\|_{L^2}\right).
\end{align}
Combining with estimate~\eqref{Aux-eq:Bouchut-Smoothed-Bounds}, we have $\nabla_q f\in L^2(\R\times\R^{2d})$ with 
\begin{align*}
\|\nabla_q f\|_{L^2}=\lim\limits_{\d\rightarrow 0, R\rightarrow\infty}\|\nabla_q f_{\d,R}\|_{L^2} \leq C\left(\|\nabla_p g\|_{L^2}+\|f\|_{L^2}+\|\nabla_p f\|_{L^2}\right).
\end{align*}

Now we prove that $g_{\d,R}$ satisfies inequality~\eqref{FIR-ineq:Bouchut-g}. Since the equations are defined in a distributional sense, for any $\phi\in C_c^\infty(\R\times\R^{2d})$ we have
\begin{align*}
\int_{\R^{1+2d}} \bar g_\d \phi 
&= \int_{\R^{1+2d}} \left[-f_\d p\cdot\nabla_q\phi+fp\cdot\nabla_q\nu_\d\ast\phi \right]
= \int_{\R^{1+2d}} \left[-f\nu_\d\ast(p\cdot\nabla_q\phi)+fp\cdot\nabla_q\nu_\d\ast\phi \right] \\
&\leq \|f\|_{L^2} \|\nu_\d\ast(p\cdot\nabla_q\phi)+p\cdot\nabla_q\nu_\d\ast\phi\|_{L^2}\\
&\leq   \|f\|_{L^2}\|\kappa_\d\|_{L^1}\|\phi\|_{L^2}\leq C\|f\|_{L^2}\|\phi\|_{L^2}.
\end{align*}
where $\kappa_\d(q,p)=|p||\nabla_q\nu_\d(q,p)|$. Here the final inequality follows from Lemma~\ref{FIR-lem:Aux-Conv-Bound} since $\|\kappa_\d\|_{L^1}\leq C$ independent of $\d$ (recall~\eqref{Aux-eq:Reg-Seq-Ind-L1}). As a result of this calculation it follows that
\begin{align*}
\|\bar g_\d\|_{L^2} \leq C\|f\|_{L^2},
\end{align*}
where $C$ is independent of $\d$.

A similar calculation for $\nabla_p \bar g_\d$ gives, using implicit summation over repeated indices,
\begin{align*}
\int_{\R^{2d}} \bar g_\d \partial_{p_j} \phi 
&=\int_{\R^{2d}} \left[-(\nu_\d \ast f) (p_i\partial_{q_ip_j}\phi)+fp_i\partial_{q_i}(\nu_\d\ast\partial_{p_j}\phi)\right]=\int_{\R^{2d}} \left[\partial_{q_i}\phi\,\partial_{p_j}(p_i\nu_\d\ast f)+fp_i\partial_{p_j}(\nu_\d\ast\partial_{q_i}\phi)\right]\\
&=\int_{\R^{2d}} \left[\partial_{q_i}\phi\,\partial_{p_j}(p_i\nu_\d\ast f)-\nu_\d\ast\partial_{q_i}\phi\, \partial_{p_j}(fp_i)\right]\\
&=\int_{\R^{2d}} \left[\partial_{q_i}\phi\,\bigl(p_i\nu_\d\ast\partial_{p_j}f + \delta_{ij} \nu_\d\ast f\bigr)
 -\nu_\d\ast\partial_{q_i}\phi\, \bigl(p_i\partial_{p_j}f + \delta_{ij}f\bigr)\right]\\
&=\int_{\R^{2d}} \partial_{p_j}f \left[\nu_\d\ast(p_i\partial_{q_i}\phi)-p_i\nu_\d\ast\partial_{q_i}\phi\right] \leq C\|\partial_{p_j}f\|_{L^2}\|\nabla p_i\|_{\infty} \|\phi\|_{L^2},
\end{align*}
where $C$ is independent of $\d$, implying
\begin{align*}
\|\partial_{p_j }\bar g_\d\| \leq C\|\nabla_p f\|_{L^2}.
\end{align*}

Now let us consider $\bar g_{\d,R}$ (defined in~\eqref{Aux-eq:Cutoff-Error-Reg}). Since $|\nabla_p \rchi_R|\leq 1/R$ and $|\Delta_p \rchi_R|\leq 1/R^2$, it follows that
\begin{align*}
\|\bar g_{\d,R}\|_{L^2} \leq C\|f_\d p\cdot\nabla_q\rchi_R\|_{L^2}+\frac{C}{R^2}\|f_\d\|+\frac{C}{R}\|\nabla_pf_\d\| \leq C\|f\|_{L^2} + \frac{C}{R^2}\|f\|+\frac{C}{R}\|\nabla_pf\|,
\end{align*}
i.e.\ $\bar g_{\d,R}$ is bounded in $L^2(\R\times\R^{2d})$ independent of $\d, R$. A similar calculation shows that 
\begin{align*}
\|\partial_{p_j}\bar g_{\d,R}\|_{L^2} \leq C(R)\xrightarrow{R\rightarrow\infty}0. 
\end{align*}
This completes the proof.
\end{proof}

We now use Proposition~\ref{Aux-Prop:Bouchut-Type-Res} to prove regularity properties of equation~\eqref{FIR-eq:VFP-Main-Aux-PDE}.

\begin{proposition}\label{FIR-prop:L1-loc-Prop} 
Let $g$ be the solution of the variational problem~\eqref{FIR-eq:PDE-Weak-Form-Lambda=0} (in the sense of Corollary~\ref{Aux-Cor:WellPose-Original})  with $U=0$ and with initial datum $g^0\in X$. If $g^0\in C^3(\R^{2d})\cap X$, then $g$ satisfies
\begin{align*}
\partial_t g,\nabla g,\Delta_p g\in L^2_{\mathrm{loc}}([0,T]\times\R^{2d}).
\end{align*}
\end{proposition}
\begin{proof}
Let $g$  be the solution of the variational problem~\eqref{FIR-eq:PDE-Weak-Form-Lambda=0} in the sense of Corollary~\ref{Aux-Cor:WellPose-Original}, but on the time interval $[0,\infty)$; since Corollary~\ref{Aux-Cor:WellPose-Original} guarantees existence and uniqueness on any finite interval, this $g$ is well defined.  We extend $g$ to all $t$ by setting
\[
g(t) := \begin{cases}
g^0 & t\leq 0\\
g(t) & t>0
\end{cases}
\]

We next recast the variational problem~\eqref{FIR-eq:PDE-Weak-Form-Lambda=0} in the form used in Proposition~\ref{Aux-Prop:Bouchut-Type-Res}. Changing $p$ to $-p$ and rearranging~\eqref{FIR-eq:PDE-Weak-Form-Lambda=0} we find, also using Remark~\ref{FIR-rem:general-phi}, for all $\phi\in C_c^\infty(\R\times\R^{2d})$
\begin{align*}
\int_0^T\int_{\R^{2d}} &\Bigl\{ g\left(-\partial_t\phi-p\cdot\nabla_q\phi\right) +\nabla_pg\cdot\nabla_p\phi\Bigr\}e^{-H}\\
&=\int_0^T\int_{\R^{2d}} \Bigl\{ -g\nabla_qV\cdot \nabla_p\phi-\frac 12g\nabla_q\psi\ast\rho_t\cdot\nabla_p\phi-\frac12 g\Psi\phi+\frac 12 \phi\nabla_q\psi\ast\rho_t\cdot\nabla_pg  \Bigr\}e^{-H}
- \int_{\R^{2d}} g\phi\Big|_{t=0}^{t=T}e^{-H}.
\end{align*}
With the choice $\phi=\tilde \phi e^H$, where $\tilde \phi\in C_c^\infty(\R\times\R^{2d})$ we rewrite this as
\begin{multline}\label{Aux-eq:PDE-Reg-Weak-to-Strong}
\int_0^T\int_{\R^{2d}} \Bigl\{ g\left(-\partial_t\tilde \phi-p\cdot\nabla_q\tilde \phi\right) +\nabla_pg\cdot\nabla_p\tilde\phi\Bigr\}\\
=\int_0^T\int_{\R^{2d}} \Bigl\{ \nabla_pg\cdot\nabla_p H\tilde\phi -g\nabla_qV\cdot \nabla_p\tilde \phi-\frac 12g\nabla_q\psi\ast\rho_t\cdot\nabla_p\tilde\phi-\frac 12g\nabla_q\psi\ast\rho_t\cdot\nabla_pH\tilde\phi\\
-\frac12 g\tilde\Psi\phi+\frac 12 \tilde\phi\nabla_q\psi\ast\rho_t\cdot\nabla_pg  \Bigr\}
- \int_{\R^{2d}} g\tilde \phi\Big|_{t=0}^{t=T}.
\end{multline}
After combining this expression with similar expressions for the regions $t>T$ and $t<0$, we find that these expressions form  the distributional version of the equation
\begin{align}
\label{FIR-eq:reg-full-space}
\partial_t  g -p\nabla_q g-\Delta_p g= G \qquad \text{in } \R\times \R^{2d},
\end{align}
where 
\begin{align}\label{Aux-eq:Reg-G-def}
G=\begin{cases}
-p\nabla_q g^0 - \Delta_p g^0 & t< 0\\
\nabla_pg\cdot\nabla_q V-\nabla_q\psi\ast\rho_t\cdot \nabla_p g-\nabla_p g\cdot\nabla_p H-\frac 12 g\left(\nabla_q\psi\ast\rho_t\cdot\nabla_pH+\Psi\right) & t>0.
\end{cases}
\end{align}
Since $g,\nabla_p g\in L^2(0,T;L^2(e^{-H}))\subset L^2_{\mathrm{loc}}(\R\times \R^{2d})$ and by assumption $g^0\in C^3(\R^{2d})$, it follows that $G\in L^2_{\mathrm{loc}}(\R\times \R^{2d})$. After a smooth truncation, Theorem~1.5 of \cite{Bouchut02} implies that $\Delta_p g\in L^2_{\mathrm{loc}}(\R\times\R^{2d})$. Using this additional regularity in the definition of $G$~\eqref{Aux-eq:Reg-G-def}, it then follows that $\nabla_p G\in L^2_{\mathrm{loc}}(\R\times\R^{2d})$. 
Applying Proposition~\ref{Aux-Prop:Bouchut-Type-Res} to a truncated version of~\eqref{FIR-eq:reg-full-space} then implies the result.
\end{proof}

\begin{remark}
From Proposition~\ref{FIR-prop:L1-loc-Prop} it follows that the solution for the variational problem~\eqref{FIR-eq:PDE-Weak-Form-Lambda=0} satisfies the original equation~\eqref{FIR-eq:VFP-Main-Aux-PDE} (with the choice $U=0$)
\begin{equation*}
\begin{aligned}
&\partial_t g - J\nabla H\cdot \nabla g -J\nabla (\psi\ast\rho_t)\cdot \nabla g + \nabla_pH\cdot\nabla_pg - \Delta_p g -\frac{g}{2}\left(J\nabla H\cdot\nabla\psi\ast\rho_t-\Psi\right)=0,\\
& g|_{t=0}=g^0,
\end{aligned}
\end{equation*}
in $L^1_{\mathrm{loc}}([0,T]\times\R^{2d})$ (i.e.\ all derivatives are in $L^1_{\mathrm{loc}}$).
\end{remark}

\section{Proof of Theorem~\ref{DOG-thm:DOG-Compactness}}\label{DOG-App-Sec:HIR}
In this section we prove Theorem~\ref{DOG-thm:DOG-Compactness}. We will use the following alternative definition of the rate functional
\begin{equation}\label{DOG-def:Rate-Fn-h-Def}
\begin{aligned}
I(\rho)=
\begin{cases}
\displaystyle
\frac{1}{2}\int\limits_0^T\int\limits_{\mathbb{R}^{2d}}|h_t|^2\,d\rho_t dt
\quad
 &\parbox{11cm}{$\text{if } \partial_t\rho_t=\vep^{-1}\div(\rho J\nabla H)+\Delta_p\rho-\div_p(\rho_t h_t), 
\ \text{for } h\in L^2(0,T;L^2_\nabla(\rho))$,\\
\hbox{}\qquad\qquad and $\rho|_{t=0} = \rho_0$,}
\\[0.2cm]
+\infty & \text{otherwise},
\end{cases}
\end{aligned}
\end{equation} 
where $\vep>0$ is fixed.

\begin{proof}[Proof of Theorem~\ref{DOG-thm:DOG-Compactness}]
We first show that the estimate~\eqref{DOG-eq:Ham-Bounds-DOG} holds.  Since $\rho$ satisfies $I(\rho)<C$, using the defintion~\eqref{DOG-def:Rate-Fn-h-Def} of the rate functional we find that there exists $h\in L^2(0,T;L^2_\nabla(\rho))$ such that for any  $f\in C_c^2(\R^2)$ 
\begin{align}\label{DOG-eq:Compact-Weak-Form}
\frac{d}{dt}\int_{\R^2}fd\rho_t =\int_{\R^2}\Big{(}\frac{1}{\vep}J\nabla H\cdot \nabla f +\Delta_p f +\nabla_p f\cdot h_t\Big)d\rho_t.
\end{align}
Formally substituting $f=H$ in~\eqref{DOG-eq:Compact-Weak-Form} and using the growth conditions on $H$ (see~\ref{cond:DOG-H-Growth}) we find
\begin{align*}\label{DOG-eq:Compact-Weak-Form-H}
\partial_t \int_{\R^2} Hd\rho_t =\int_{\R^2} \left(\Delta_p H + \nabla_p H\cdot h_t\right)d\rho_t
&\leq C+\frac12 \int_{\R^2}|\nabla _p H|^2 \, d\rho_t + 
\frac12 \int_{\R^2} |h_t|^2\, d\rho_t\\
&\leq C+C \int_{\R^2}H \, d\rho_t + 
\frac12 \int_{\R^2} |h_t|^2\, d\rho_t.
\end{align*}
The bound $\int H\rho_t^\vep<C$ then follows by applying a Gronwall-type estimate, integrating in time over $[0,T]$, and using the fact that $h\in L^2(0,T;L^2_\nabla(\rho))$. To make the choice  $f=H$ admissible in the definition~\eqref{DOG-def:Rate-Fn-h-Def} of the rate functional we use a two-step approximating argument. We first extend the class of admissible functions from $C^2_c(\R^{2})$ to 
\[
\mathcal A:=\Bigl\{F\in C^2_b(\R^{2}): \sup_{x\in \R^{2}} (1+|x|) |F(x)|  <\infty\Bigr\}.
\]
For a given $F\in\mathcal{A}$, define the sequence  $f_k(x)=F(x)\xi_k(x)\in C^2_c(\R^{2})$, where $\xi_k\in C^\infty_c(\mathbb{R})$ is a sequence of smoothed characteristic functions converging pointwise to one, with $0\leq\xi_k\leq1$,  $|\nabla \xi_k|\leq 1/k$, and $|d^2\xi_k|\leq 1/k^2$. Then $|\nabla H\cdot \nabla f_k|$, $\Delta_p f_k$, and $|\nabla_pf|^2$ are bounded uniformly and converge pointwise to the corresponding terms with $f_k$ replaced by $f$; convergence follows by the Dominated Convergence Theorem. In the second step, we extend $\mathcal{A}$ to include $H(q,p)$ by using an approximating sequence $\mathcal{A}\ni g_k(q,p)=H(q,p)\psi_k(H(q,p))$ where $\psi_k:\mathbb{R}\rightarrow\mathbb{R}$ is defined as $\psi_k(s):=(1+|s|/k)^{-2}$. Note that $\psi_k\rightarrow 1$ pointwise as $k\rightarrow\infty$. Proceeding as described in the formal calculations above we find
\begin{align*}
\partial_t\bigg{(}\int g_kd\rho_t
\bigg{)}\leq C\biggl(1+ \int g_kd\rho_t+\int|h_t|^2 d\rho_t\biggr),
\end{align*}
where $C$ is independent of $k$ and $\vep$. Using a Gronwall-type estimate, integrating in time over $[0,T]$ and applying the monotone convergence theorem we find~\eqref{DOG-eq:Ham-Bounds-DOG}.

Next we prove~\eqref{est:DOG-F-I}. The main idea of the proof is to consider a modified equation for which an estimate of the type~\eqref{est:DOG-F-I} holds, and then arrive at~\eqref{est:DOG-F-I} by passing to an appropriate limit.  

We consider the following modification of equation~\eqref{DOG-eq:Ran-Ham-Evo-R2},
\begin{align}\label{DOG-eq:DOG+Friction}
\partial_t\rho=-\frac{1}{\vep}\div(\rho J\nabla H)+\alpha\div_{p}(\rho\nabla_p H )+\Delta_p\rho^\vep,
\end{align}
where $\alpha>0$. Essentially, we have added a friction term to equation~\eqref{DOG-eq:Ran-Ham-Evo-R2}, as a result of which $\mu^\alpha(dqdp)=Z_\alpha^{-1}e^{-\alpha H(q,p)}dqdp$ is a stationary measure for~\eqref{DOG-eq:DOG+Friction} ($Z_\alpha$ is the normalization constant).

The rate functional corresponding to~\eqref{DOG-eq:DOG+Friction} is
\begin{equation}\label{DOG-def:Mod-Rate-Fn-h-Def}
\begin{aligned}
I_\alpha(\rho)=
\begin{cases}
\displaystyle
\frac{1}{2}\int\limits_0^T\int\limits_{\mathbb{R}^{2}}|h^\alpha_t|^2\,d\rho_t dt
\quad
 &\parbox{11cm}{$\text{if } \partial_t\rho_t=-\vep^{-1}\div(\rho J\nabla H)+\Delta_p\rho+\div_p(\rho[\alpha \nabla_p H-h_t^\alpha]), $\\[\jot]
\hbox{}\qquad\qquad $\text{ for } h^\alpha\in L^2(0,T;L^2_\nabla(\rho_t)),$ and $\rho|_{t=0} = \rho_0$,}
\\[0.2cm]
+\infty & \text{otherwise}.
\end{cases}
\end{aligned}
\end{equation}
Note that equation~\eqref{DOG-eq:DOG+Friction} is a special case of the VFP equation (with the choice $\psi=0$) and therefore the proof of Theorem~\ref{DOG-thm:VFP-Ent-Fisher-Inf-Bounds} also applies to this case. We follow the proof up to~\eqref{ineq:FIR-pre-end} (adding a constant $\alpha$ to the friction) to find for any $\tau\in [0,T]$
\begin{align*}
\RelEnt(\rho_\tau|\mu^\alpha) + \int_0^\tau\int_{\R^{2}} 
\Bigl( \Delta_p\varphi -\alpha \nabla_pH\cdot\nabla_p\varphi - \frac 12 |\nabla_p\varphi|^2\Bigr) d\rho_tdt 
\leq   I_\alpha(\rho) + \RelEnt(\rho_0|\mu^\alpha),
\end{align*}
for any $\varphi\in C_c^\infty(\R\times\R^2)$. Using the definition of relative entropy we have
\begin{multline}\label{DOG:eq:FIR-alpha-Pass}
\M{F}(\rho_\tau) + \int_0^\tau\int_{\R^{2}} 
\Bigl( \Delta_p\varphi -\alpha \nabla_pH\cdot\nabla_p\varphi - \frac 12 |\nabla_p\varphi|^2\Bigr) d\rho_tdt 
\leq   I_\alpha(\rho) + \M{F}(\rho_0) +\alpha\int_{\R^2} H\rho_\tau-\alpha\int_{\R^2} H\rho_0.
\end{multline}
Below we show that  $I_\alpha(\rho)\rightarrow I(\rho)$ as $\alpha\rightarrow 0$. Then passing to the limit $\alpha\rightarrow 0$ in~\eqref{DOG:eq:FIR-alpha-Pass} we find
\begin{align*}
\M{F}(\rho_\tau) + \int_0^\tau\int_{\R^{2}} 
\Bigl( \Delta_p\varphi - \frac 12 |\nabla_p\varphi|^2\Bigr) d\rho_tdt 
\leq   I(\rho) + \M{F}(\rho_0),
\end{align*}
where we have used $|\nabla_p H|^2\leq C(1+H)$ along with the estimate~\eqref{DOG-eq:Ham-Bounds-DOG}. The required inequality~\eqref{est:DOG-F-I} then follows by taking the supremum over $\varphi\in C_c^\infty(\R\times\R^{2})$.  

To complete the proof we show that $I_\alpha(\rho)\rightarrow I(\rho)$ as $\alpha\rightarrow 0$. Using the definition of the rate functionals for the original equation~\eqref{DOG-def:Rate-Fn-h-Def} and the modified equation~\eqref{DOG-def:Mod-Rate-Fn-h-Def}, we  write the rate functional for the modified equation as
\begin{align*}
I_\alpha(\rho)=\frac 12\int_0^T\int_{\R^2} |h_t^\alpha|^2d\rho_tdt &= \frac 12\int_0^T\int_{\R^2} |h_t -\alpha\nabla_pH|^2d\rho_tdt \\ 
&=\frac 12\int_0^T\int_{\R^2}\Bigl( |h_t |^2 + \alpha^2 |\nabla_p H|^2 - 2\alpha \nabla_p H\cdot h_t \Bigr)d\rho_tdt  \xrightarrow{\alpha\rightarrow 0} I(\rho),
\end{align*}
where we have used $|\nabla_p H|^2\leq C(1+H)$ and the estimate~\eqref{DOG-eq:Ham-Bounds-DOG} to arrive at the convergence statement. Note that~\eqref{DOG-eq:Ham-Bounds-DOG} along with the definition of the rate functionals implies that $I(\rho)<\infty$ iff  $I_\alpha(\rho)<\infty$. 
\end{proof}

\end{appendices}
\section*{Acknowledgements}
The authors would like to thank the anonymous referee for valuable suggestions and comments. US thanks Giovanni Bonaschi, Xiulei Cao, Joep Evers and Patrick van Meurs for insightful discussions regarding Theorem~\ref{DOG-thm:VFP-Compactness} and Theorem~\ref{DOG-lem:DOG-Local-Eq-Const-Level-Sets}. MAP and US kindly acknowledge support from the Nederlandse Organisatie voor Wetenschappelijk Onderzoek (NWO) VICI grant 639.033.008. MHD was supported by the ERC Starting Grant 335120.  Part of this work has appeared in Oberwolfach proceedings \cite{PeletierDuongSharma13}.\\

\bibliographystyle{alphainitials}  
\newcommand{\etalchar}[1]{$^{#1}$}

\end{document}